\documentclass[preprint,12pt]{elsarticle}

\usepackage{tikz}
\usepackage{pgfplots}
\usepackage{verbatim}

\usepackage{amssymb,amsmath,amsthm,amsfonts,bm}
\usepackage{times,fancyhdr}

\usepackage{wasysym}

\newcommand{\const}{\mathop{\rm const}\nolimits}
\newcommand{\grad}{\mathop{\rm grad}\nolimits}
\newcommand{\rot}{\mathop{\rm rot}\nolimits}
\renewcommand{\div}{\mathop{\rm div}\nolimits}

\newtheorem{theorem}{Theorem}
\newtheorem{lemma}{Lemma}

\journal{arXiv.org}

\begin{document}

\begin{frontmatter}

\title{Numerical Methods for Solving Convection-Diffusion Problems}

\author{A. Churbanov}
\ead{achur@ibrae.ac.ru}
\author{P. Vabishchevich}
\ead{vabishchevich@gmail.com}
\address{Nuclear Safety Institute, Russian Academy of Sciences, \\52, B. Tulskaya, 115191 Moscow, Russia}

\begin{abstract}
Convection-diffusion equations provide the basis for describing
heat and mass transfer phenomena as well as processes of continuum mechanics.
To handle flows in porous media, the fundamental issue is to model correctly
the convective transport of individual phases.
Moreover, for compressible media, the pressure equation itself is just a time-dependent
convection-diffusion equation.

For different problems, a convection-diffusion equation may be be written in various forms. 
The most popular formulation of convective transport employs the divergent (conservative) form.
In some cases, the nondivergent (characteristic) form seems to be preferable.
The so-called skew-symmetric form of convective transport operators 
that is the half-sum of the operators in the divergent and nondivergent forms
is of great interest in some applications.

Here we discuss the basic classes of discretization in space: finite difference schemes
on rectangular grids, approximations on general polyhedra (the finite volume method),
and finite element procedures. The key properties of discrete operators
are studied for convective and diffusive transport.
We emphasize the problems of constructing approximations for convection and diffusion operators
that satisfy the maximum principle at the discrete level --- they are called monotone approximations.

Two- and three-level schemes are investigated for transient problems.
Unconditionally stable explicit-implicit schemes are developed for convection-diffusion problems.
Stability conditions are obtained both in finite-dimensional Hilbert spaces and in
Banach spaces depending on the form in which the convection-diffusion equation is written.
\end{abstract}


\end{frontmatter}

\section{Convection-diffusion problems} 

Convection-diffusion problems are governed by typical mathematical models, which are common in fluid and gas dynamics.
Heat and mass transfer is conducted not only via diffusion, but appears
due to motion of a medium, too. Here we present typical examples
of model convection-diffusion problems, which use various forms for the terms describing convective transport.

\subsection{Basic problems of continuum mechanics} 

Principal features of physical and chemical processes in fluid mechanics
\cite{landau_fluid_1987,batchelor_introduction_2000}
result from motion of a medium due to various forces.
Heat and mass transfer phenomena in a moving medium may be treated as the simplest examples of these peculiarities.

Let ${\bm v}({\bm x},t)$ be the velocity of a liquid at a point ${\bm x}$ and
at a time moment $t$, whereas $\rho$ is its density. The thermal state of the liquid
is governed by the equation of heat conduction
\begin{equation}\label{1}
c_p \rho \left (\frac{\partial T}{\partial t} + ({\bm v}\cdot\grad) T \right )
= \div (k \grad T) + q,
\end{equation}
where $T$ stands for the temperature, $c_p$ is the specific heat capacity at a constant
pressure, $k$ denotes the thermal conductivity of the liquid and $q$ describes the intensity
of volumetric heat sources.

The temperature at a given spatial point is governed not only by conduction (diffusion) of heat, 
but also by motion (convection) of fluid volumes.

The second typical example is the diffusion equation for a multicomponent mixture 
\cite{crank_mathematics_1980,cussler_diffusion}. 
We assume that a liquid is heterogeneous, more exactly, it is a mixture of two components. 
In this case, the mixture composition may be described by the concentration $c$ associated with only one
component. The corresponding equation for the concentration
(neglecting the diffusion flux caused by the temperature gradient) has the form
\begin{equation}\label{2}
\frac{\partial (\rho c)}{\partial t} + \grad ({\bm v}\rho c) =
\div (\rho k \grad c),
\end{equation}
where $k$ denotes the diffusivity and $\rho$ is treated
as the total density of the liquid.

The equation (\ref{2}) may be rewritten as
\begin{equation}\label{3}
\frac{\partial m}{\partial t} + \div ({\bm v} m)
+\div \left (m \frac k\rho \grad \rho \right )
=\div (k \grad m),
\end{equation}
where $m = \rho c$ is the mass of one of the components in a volume unit.
The equation (\ref{3}) may be reduced to the form
\begin{equation}\label{4}
\frac{\partial m}{\partial t} + \div (\tilde{\bm v} m)
=\div (k \grad m),
\end{equation}
where the expression
\[
\tilde{\bm v} = {\bm v} + \frac k\rho \grad \rho
\]
describes the effective convective transport.

Using the continuity equation
\begin{equation}\label{5}
\frac{\partial \rho}{\partial t} + \div ({\bm v}\rho) = 0,
\end{equation}
we obtain
\[
\frac{\partial (\rho c)}{\partial t} + \div ({\bm v}\rho c) =
\rho \left (\frac{\partial c}{\partial t} + ({\bm v}\cdot\grad)c \right ).
\]
Therefore, under the natural assumptions on the positiveness of $\rho$,
from (\ref{2}), we arrive at the equation
\begin{equation}\label{6}
\frac{\partial c}{\partial t} + ({\bm v}\cdot \grad)c -
\frac k\rho \grad \rho\cdot \grad c
=\div (k \grad c).
\end{equation}
Similarly to (\ref{4}), rewrite (\ref{6}) as
\begin{equation}\label{7}
\frac{\partial c}{\partial t} + (\tilde{\bm v}\cdot \grad)c
=\div (k \grad c),
\end{equation}
where now
\[
\tilde{\bm v} = {\bm v} - \frac k\rho \grad \rho.
\]
Thus, we come to the equation for the concentration, where
convective transport has the nondivergent form, as it takes place in the heat equation (\ref{1}).
In equation (\ref{4}) as well as in the continuity equation (\ref{5}),
convective transport is written in the divergent form.

More complete models of heat and mass transfer include also an equation
that describes the motion of the medium itself and determines, in particular, the velocity $\bm v$. 
For simplicity, we restrict ourselves to the Navier-Stokes equation
for an incompressible ($\rho = \const$) homogeneous medium. In this case, the momentum equation 
seems like this:
\begin{equation}\label{8}
\rho \left (\frac{\partial {\bm v}}{\partial t} + ({\bm v}\cdot \grad){\bm v} \right )
= - \grad p + \eta \div \grad {\bm v},
\end{equation}
whereas the continuity equation (\ref{5}) is reduced to
\begin{equation}\label{9}
\div {\bm v} = 0.
\end{equation}
Here $p$ denotes the pressure and $\eta=\const$ stands for the dynamic viscosity of the fluid.

The equations (\ref{8}) may be treated as the equations
of convective and diffusive (due to the viscosity) transport of each 
individual component of the velocity $\bm v$. In this situation, 
in order to evaluate the pressure $p$, it is necessary to involve equation (\ref{9}).

If we eliminate the pressure from equation (\ref{8}),
then, for the  vorticity ${\bm w} = \rot {\bm v}$, we obtain the equation
\[
\rho \left (\frac{\partial {\bm w}}{\partial t} + ({\bm v}\cdot \grad){\bm w}
- ({\bm w} \cdot \grad) {\bm v} \right )
= \eta \div \grad {\bm w}.
\]
It is easy to see that the dynamics of the vorticity for an incompressible fluid
is determined by a specific convective and diffusive transport.

More sophisticated models that include convective and diffusive transport as the most important element
are used in modeling compressible flows. It should be noted that
convective-diffusive transport is essential for predictions of various
gas and fluid flows. In particular, environmental problems are of great importance:
pollutants spreading in the atmosphere and water basins,
contaminants transport in groundwaters and so on.

\subsection{Various forms of the hydrodynamics equations} 

In theoretical studying applied problems, the conservative form of the hydrodynamics equations is in common use. 
In this case, the equations have the divergent form and express directly the corresponding laws of conservation (for the mass, momentum and energy).
On the other hand, we should pay attention to the nondivergent (characteristic) form of the hydrodynamics equations,
which is connected with the representation that is derived via differentiating the convective transport terms. 
The paper \cite{VabForm} presents a new form of the hydrodynamics equations that is characterized
by writing the convective terms in the skew-symmetric form.
New quantities --- the so-called SD--variables (\textbf{S}quare root from \textbf{D}ensity)
that are based on using not the density itself but the square root from the density 
--- are used as unknown variables.
Physical and mathematical arguments in favor of introducing this form of the hydrodynamics equations are discussed below.

The system of hydrodynamics equations includes, first of all,
the scalar equation of continuity and the vector equation of momentum.
In more common cases, there may be several motion equations as well as continuity equations
--- we speak of models for multicomponent media.
Furthermore, the system of equations may be supplemented with an energy equation.
Usually, the following scalar equation of convection-diffusion serves as the basic equation in continuum
mechanics and heat and mass transfer (see, e.g., \cite{Patankar:1980:NHT,Wesseling}), i.e.,
\begin{equation}\label{10}
  \frac{\partial (\varrho \varphi)}{\partial t} +
  \div (\varrho {\bm v} \varphi) =
  \div (D \grad \varphi) ,
\end{equation}
where $\varphi$ is a desired scalar function and $D$ denotes the diffusivity.
This equation is written in the conservative (divergent) form. 
Concerning to equation (\ref{10}), a number of problems are discussed in the literature,
such as the construction of discretization in space and in time,
the investigation of convergence of the approximate solution to the exact one for the
corresponding boundary value problem \cite{SamVabConv,SamVabHeatTran2}. 

The main peculiarities of the system of fluid dynamic equations
become evident in studying the system of two scalar equations
that includes not only equation(\ref{10}), but also the continuity equation
\begin{equation}\label{11}
  \frac{\partial \varrho }{\partial t} +
  \div (\varrho {\bm v} ) = 0 .
\end{equation}
Just this system of equations (\ref{10}), (\ref{11}) is said to be the basic system of scalar hydrodynamic equations. 

In investigation of transport phenomena in continuum mechanics, the primary features of transport
of scalar variables are represented in equation (\ref{10}). 
Concerning vector fields, the coordinatewise representation may be unsuitable. 
Thus, it seems reasonable to supplement the system of equations (\ref{10}), (\ref{11}) 
with the vector convection-diffusion equation
\begin{equation}\label{12}
  \frac{\partial (\varrho {\bm u})}{\partial t} +
  \div (\varrho {\bm v} \otimes {\bm u} )  =
  \div (D \grad {\bm u}),
\end{equation}
where ${\bm u}$ is the desired vector function.
This system of equations (\ref{10})--(\ref{12}) is called the basic system of hydrodynamics equations. 

Taking into account
\[
\div (\varrho {\bm v} \varphi ) = 
\varphi \div (\varrho {\bm v} ) +
\varrho {\bm v} \cdot \grad \varphi
\]
and the continuity equation (\ref{11}), we get
\begin{equation}\label{13}
\varrho \frac{\partial \varphi}{\partial t} + 
\varrho {\bm v} \cdot \grad \varphi =
  \div (D \grad \varphi) .
\end{equation} 
Similarly, we can rewrite equation (\ref{12}) as
\begin{equation}\label{14}
\varrho \frac{\partial {\bm u}}{\partial t} + 
\varrho {\bm v} \cdot \grad {\bm u} =
  \div (D \grad {\bm u}) .
\end{equation} 
The equations (\ref{13}), (\ref{14}) are written in the nonconservative (nondivergent) form. 
It should be noted that the continuity equation (\ref{11}) cannot be written in the nonconservative form. 
Therefore, the basic system of hydrodynamics equations may be written in the conservative 
(\ref{10})--(\ref{12}) or in the partially nonconservative form (\ref{11}), (\ref{13}), (\ref{14}). 
Only for an incompressible medium, where equation (\ref{11}) takes the form
\[
  \div {\bm v} = 0,
\]
it is possible to speak about the nonconservative form of the equations.

Let us write the operator of convective transport in the skew-symmetric form \cite{SamVabConv,SamVabHeatTran2} as
\begin{equation}\label{15}
  \mathcal{C} \theta = \frac{1}{2} \div ({\bm v} \theta) 
  + \frac{1}{2} {\bm v} \cdot \grad \theta , 
\end{equation} 
i.e., as the half-sum of the operators of convective transport 
in divergent (conservative) and nondivergent (nonconservative) forms. 

In the basic system of fluid dynamics equations (\ref{10})--(\ref{12}),
instead of $\varrho, \varphi, {\bm u}$, we introduce new unknown variables:
\begin{equation}\label{16}
  s = (\varrho)^{1/2},
  \quad \zeta = (\varrho)^{1/2} \varphi ,
  \quad {\bm w} = (\varrho)^{1/2} {\bm u} .
\end{equation} 
The main peculiarity of these unknowns consists in using the square root from the density $s = (\varrho)^{1/2}$
instead the density $\varrho$ itself. That is why we speak of SD-variables (\textbf{S}quare root from \textbf{D}ensity). 

For the new unknowns, the system of equations (\ref{10})--(\ref{12}) may be rewritten in the following way:
\begin{equation}\label{17}
  \frac{\partial s }{\partial t} +
  \frac{1}{2} \div ({\bm v} s) 
  + \frac{1}{2} {\bm v} \cdot \grad s = 0, 
\end{equation}
\begin{equation}\label{18}
  \frac{\partial \zeta }{\partial t} +
  \frac{1}{2} \div ({\bm v} \zeta) 
  + \frac{1}{2} {\bm  v} \cdot \grad \zeta =
  \frac{1}{s} \div \left ( D \grad \left (\frac{\zeta}{s} \right) \right) ,
\end{equation}
\begin{equation}\label{19}
  \frac{\partial {\bm w}}{\partial t} +
  \frac{1}{2} \div ({\bm v} \otimes {\bm w}) 
  + \frac{1}{2} {\bm v} \cdot \grad {\bm w} =
  \frac{1}{s} \div \left ( D \grad \left (\frac{{\bm w}}{s} \right) \right) .
\end{equation}
In this case, all three equations involve the convective terms that are written in the skew-symmetric form. 

As a typical example of using the new variables,
we study the Navier-Stokes equations for a viscous compressible medium, 
which express the conservations laws for the mass, momentum, and energy.
The continuity equation has the form (\ref{11}). 
Usually, the momentum equation is written in the conservative form 
\begin{equation}\label{20}
  \frac{\partial (\varrho {\bm v}) }{\partial t} + 
  \div (\varrho {\bm v} \otimes {\bm v}) = \div \mathsf{N} 
  - \grad p .
\end{equation}
Here
\[
  \mathsf{N} = - \frac{2}{3} \mu \div {\bm v} \, \mathsf{E} + 2 \mu \mathsf{S} ,
\]
As for $\mathsf{S}$, the coordinatewise representation seems like this:
\[
  \mathsf{S}_{ij} =
  \frac{1}{2} \left (
  \frac{\partial v_i}{\partial x_j}
  + \frac{\partial v_j}{\partial x_i}
  \right ) .
\]
Now introduce the energy equation
\begin{equation}\label{21}
\frac{\partial (\varrho e) }{\partial t} + 
\div (\varrho {\bm v} e) 
= \div ( k \grad T) + p \div {\bm v} +
\mathsf{N}: \grad {\bm v} .
\end{equation} 
The term $\mathsf{N}: \grad {\bm v}$ describes the heat dissipation due to the fluid viscosity
and  $\mathsf{N}: \grad {\bm v}$ is the scalar product of tensors:
\[
  \begin{array}{ll}
  \mathsf{N}: \grad {\bm v}   =    &  
  {\displaystyle  N_{xx} \frac{\partial v_x}{\partial x}
  + N_{xy} \frac{\partial v_x}{\partial y}
  + N_{xz} \frac{\partial v_x}{\partial z} } \\
    & + {\displaystyle N_{yx} \frac{\partial v_y}{\partial x}
  + N_{yy} \frac{\partial v_y}{\partial y}
  + N_{yz} \frac{\partial v_y}{\partial z}  } \\
    &  +  {\displaystyle N_{zx} \frac{\partial v_z}{\partial x}
  + N_{zy} \frac{\partial v_y}{\partial y}
  + N_{zz} \frac{\partial v_z}{\partial z} .   } \\
\end{array}
\]
Let us introduce the following new unknown variables:
\begin{equation}\label{22}
  s = (\varrho)^{1/2},
  \quad \mathbf{w} = (\varrho)^{1/2} {\bm v} ,
  \quad \zeta = (\varrho)^{1/2} e .
\end{equation} 
For the variables (\ref{22}), the system of the Navier-Stokes equations
(\ref{11}), (\ref{20}), (\ref{21}) has the following form:
\begin{equation}\label{23}
  \frac{\partial s }{\partial t} +
  \frac{1}{2} \left (\div \left ( \frac{\mathbf{w}}{s}  s \right ) 
  + \frac{\mathbf{w}}{s}  \cdot \grad s \right ) = 0, 
\end{equation}
\begin{equation}\label{24}
  \frac{\partial \mathbf{w}}{\partial t} +
  \frac{1}{2} \left (\div \left (\frac{\mathbf{w}}{s}  \otimes \mathbf{w} \right ) 
  + \frac{\mathbf{w}}{s}  \cdot \grad \mathbf{w} \right ) =
  \frac{1}{s}\div \mathsf{N} 
  - \frac{1}{s} \grad p ,
\end{equation}
\begin{equation}\label{25}
\begin{split}
  \frac{\partial \zeta }{\partial t} & +
  \frac{1}{2} \left (\div \left (\frac{\mathbf{w}}{s}  \zeta \right ) 
  + \frac{\mathbf{w}}{s}  \cdot \grad \zeta  \right ) \\
  & = \frac{1}{s} \div ( k \grad T) + 
  \frac{p}{s} \div \left (\frac{\mathbf{w}}{s}\right ) +
  \frac{1}{s}\mathsf{N}: \grad \left (\frac{\mathbf{w}}{s}\right ) .
\end{split}
\end{equation}
The system of equations (\ref{23})--(\ref{25}) needs to be supplemented with some equation of state. 
It should be highlighted that using the variables (\ref{22}), the convective terms are written in the skew-symmetric form. 

\subsection{The pressure problem for multiphase flows in porous media} 

The system of governing equations for multiphase flows includes 
\cite{Peaceman,Aziz} the continuity equation for each phase, 
where $\alpha =1,2,\dots, m$ is the phase index.
The mass conservation law for each individual phase is expressed by the following equation:
\begin{equation}\label{26}
  \frac{\partial (\phi \, b_\alpha S_\alpha)}{\partial t} +
  \div (b_\alpha  {\bm u}_\alpha) = 
  - b_\alpha q_\alpha,
  \quad \alpha =1,2,\dots, m .
\end{equation}
Here $\phi$ stands for the porosity,
$b_\alpha$ is the phase density,
$S_\alpha$ denotes the phase saturation,
${\bm u}_\alpha$ is the velocity, and
$q_\alpha$ describe the volumetric mass sources.

For simplicity, we neglect the capillary and gravitational forces.
In this simplest case, the equation of fluid motion in porous media 
has the form of Darcy's law, where the velocity is directly determined by the common pressure:
\begin{equation}\label{27}
  {\bm u}_\alpha =
  - \frac{k_\alpha}{\mu_\alpha} \, \mathsf{k} \cdot \grad p ,
  \quad \alpha =1,2,\dots, m .
\end{equation}
In (\ref{27}),  $\mathsf{k}$ is the absolute permeability
(in general, a symmetric second-rank tensor),
$k_\alpha$ denote the relative permeabilities,
$\mu_\alpha$ stands for the phase viscosity,
and $p$ is the pressure.

The unknown variables in the system of equations (\ref{26}), (\ref{27})  
are the phase saturations $S_\alpha, \ \alpha =1,2,\dots, m$ 
and the pressure ($m+1$ unknowns in all).
In the simplest case, the coefficients in equations (\ref{26}), (\ref{27}) are
defined as some relations
\[
  \phi = \phi(p),
  \quad b_\alpha = b_\alpha(p),
  \quad q_\alpha = q_\alpha(S_\alpha),
  \quad k_\alpha = k_\alpha(S_\alpha),
  \quad \mu_\alpha = \const .
\]
The summation of saturations over all phases yields
\begin{equation}\label{28}
  \sum_{\alpha =1}^{m} S_\alpha = 1.
\end{equation}
Substituting (\ref{27}) in (\ref{26}) and taking into account (\ref{28}), 
we obtain a system of $m+1$ equations for $m+1$ unknowns.

The system of equations (\ref{26})--(\ref{28}) provides the basis for the description 
of multiphase flows in porous media. We have no separate equation for the pressure in this system.
The equations (\ref{26}) may be treated as the transport equation for each individual phase, 
whereas the algebraic relation (\ref{28}) may be considered as the equation for the pressure.

Let us consider more convenient forms for the system (\ref{26})--(\ref{28}), 
which lead to the typical problems of mathematical physics for the pressure.
It should be noted that such equivalent formulations do exist only at the differential level.
At the discrete level, such equivalence of formulations is not valid even for linear problems.
Thus, a proper choice of the original form of the governing equations is essential for calculations.

The most natural way to derive the equation for the pressure is the following.
Dividing each equation (\ref{26}) by $\phi \, b_\alpha > 0$ and adding them together, we get
\begin{equation}\label{29}
  \left ( \sum_{\alpha =1}^{m} \frac{S_\alpha}{\phi \, b_\alpha}
  \frac{d (\phi \, b_\alpha) }{d p} \right ) 
  \frac{\partial p}{\partial t} =
  \sum_{\alpha =1}^{m} \frac{1}{\phi \, b_\alpha}
  \div  \left ( \frac{b_\alpha k_\alpha}{\mu_\alpha} \, \mathsf{k} \cdot \grad p \right ) -
  \frac{1}{\phi} \sum_{\alpha =1}^{m} q_\alpha .
\end{equation}
Under the natural assumption for compressible fluids that
\[
  \frac{d (\phi \, b_\alpha) }{d p}  > 0,
  \quad \alpha =1,2,\dots, m ,
\]
equation (\ref{29}) for the pressure is the standard parabolic equation of second order.
In particular, the maximum principle holds for its solutions \cite{friedman1964partial}. 

In accordance with (\ref{29}), we solve in $\Omega$ the boundary value problem for the equation
\begin{equation}\label{30}
   \frac{\partial u}{\partial t} +
   \sum_{\alpha =1}^{m} a_\alpha({\bm x}) \mathcal{L}_\alpha u = f({\bm x},t) ,
\end{equation}
where $a_\alpha({\bm x}) \geq \varrho_\alpha, \ \varrho_\alpha  > 0, \ \alpha = 1,2, ..., m$, 
and the elliptic operators $\mathcal{L}_\alpha$ are defined by
\begin{equation}\label{31}
   \mathcal{L}_\alpha u = - \frac{\partial }{\partial x_1} 
   \left (k_\alpha ({\bm x}) \frac{\partial u}{\partial x_1} \right ) 
   - \frac{\partial }{\partial x_2} \left (k_\alpha ({\bm x}) \frac{\partial u}{\partial x_2} \right ) ,
   \quad  \alpha = 1,2, ..., m,
\end{equation} 
under the standard assumptions $0 < \kappa_\alpha \leq k_\alpha  \leq \overline{\kappa}_\alpha$.

In some cases (incompressible media), it is reasonable to consider the steady-state problem.
The boundary value problem is formulated for the equation
\begin{equation}\label{32}
   \mathcal{L} u = f({\bm x}),
   \quad \mathcal{L} = 
   \sum_{\alpha =1}^{m} a_\alpha({\bm x}) \mathcal{L}_\alpha  ,
\end{equation}
which is supplemented by the boundary conditions.

From (\ref{31}) and (\ref{32}), we have the representation
\begin{equation}\label{33}
  \mathcal{L} = \sum_{\alpha =1}^{m} \mathcal{L}_\alpha,
  \quad \mathcal{L}_\alpha = \mathcal{D}_\alpha  + \mathcal{C}_\alpha, 
  \quad \alpha = 1,2, ..., m, 
\end{equation} 
where
\begin{equation}\label{34}
  \mathcal{D}_\alpha u  = - \div (d_\alpha ({\bm x}) \grad u),
\end{equation} 
\begin{equation}\label{35}
  \mathcal{C}_\alpha u  = \textbf{w}_{\alpha} \grad u .
\end{equation}   
The effective diffusivity and convection velocity for the individual phase $\alpha$ are
\[
  d_\alpha =  k_\alpha a_\alpha,
  \quad \textbf{w}_{\alpha}= k_{\alpha} \grad  a_\alpha .   
\] 
Then the pressure operator takes the form of the convection-diffusion operator 
with convective term in the nondivergent form.

\section{Time-dependent problems of convection-diffusion}

Convection-diffusion equations provide important examples of second-order parabolic equations.
In particular, they are considered as the basic equations for modeling 
continuum mechanics phenomena. Some aspects of numerical solving time-dependent problems
of convection-diffusion are discussed here. In these equations,
convective terms are formulated in the divergent, nondivergent, and
skew-symmetric forms. Some essential results are presented for a model initial-boundary
value problem with Dirichlet boundary conditions for the differential
equation of convection-diffusion. These results will serve us as a check
point in developing difference schemes.
Discrete operators of diffusion and convection are constructed and analyzed
with respect to their primary properties.

\subsection{Differential problems}

Time-dependent problems of convection-diffusion are treated as evolutionary
operator equations in the corresponding spaces. To investigate them, we start
with a study on properties of differential operators describing convective and diffusive transport.
As the basic problem, we consider a time-dependent problem of convection-diffusion
with Dirichlet boundary conditions in a rectangle. The convective terms are written in various forms.
We distinguish a class of model time-dependent problems of convection-diffusion with
a constant coefficient of diffusive transport (it is independent of time but depends on spatial coordinates). 
As for coefficients of convective transport, in applied problems, they are variable both in space and in time.

In a rectangle
\[
\Omega = \{ \bm{x} \ | \ \bm{x} =\left(x_1, x_2\right), \quad 0 < x_\alpha <
l_\alpha, \quad \alpha = 1,2 \} ,
\]
we study the time-dependent convection-diffusion
equation with the convective terms written in the nondivergent form:
\begin{equation} \label{36}
\begin{split}
\frac{\partial u}{\partial t} 
& + \sum_{\alpha = 1}^{2} v_\alpha\left(\bm{x}, t\right) \frac{\partial u}{\partial x_\alpha} \\
& - \sum_{\alpha = 1}^{2} \frac{\partial}{\partial x_\alpha} \left(
k\left(\bm{x}\right) \frac{\partial u}{\partial x_\alpha} \right) 
= f\left(\bm{x}, t\right), \quad \bm{x} \in \Omega, \quad 0 < t \leq T,
\end{split}
\end{equation}
considered under the standard assumptions
$\kappa_1  \leq k\left(\bm{x}\right) \leq \kappa_2$, $\kappa_1 > 0$, $T > 0$.
This equation is supplemented with homogeneous Dirichlet boundary conditions
\begin{equation} \label{37}
u(\bm{x}, t) = 0,
\quad \bm{x} \in \partial \Omega, \quad 0 < t \leq T .
\end{equation}
For the unique solvability of the unsteady problem, it is supplemented with the initial condition
\begin{equation} \label{38}
u\left(\bm{x}, 0\right) = u^0(\bm{x}), \quad \bm{x} \in \Omega.
\end{equation}

The second example is the time-dependent equation of convection-diffusion
with the convective transport written in the divergence form:
\begin{equation} \label{39}
\begin{split}
\frac{\partial u}{\partial t} 
& + \sum_{\alpha = 1}^{2} \frac{\partial}{\partial x_\alpha}\left(v_\alpha\left(\bm{x}, t\right) u\right) \\
& - \sum_{\alpha = 1}^{2} \frac{\partial}{\partial x_\alpha} \left(
k\left(\bm{x}\right) \frac{\partial u}{\partial x_\alpha} \right) 
= f\left(\bm{x}, t\right), \quad \bm{x} \in \Omega, \quad 0 < t \leq T .
\end{split}
\end{equation}
And finally, the primary object of our investigation is the convection-diffusion equation
with the convective terms written in the skew-symmetric form:
\begin{equation} \label{40}
\begin{split}
\frac{\partial u}{\partial t} 
& + \frac{1}{2} \sum_{\alpha = 1}^{2} \left(
v_\alpha\left(\bm{x}, t\right) \frac{\partial u}{\partial x_\alpha} + 
 \frac{\partial }{\partial x_\alpha}\left(v_\alpha\left(\bm{x}, t\right)
u\right)
 \right) \\
& - \sum_{\alpha = 1}^{2} \frac{\partial}{\partial x_\alpha} \left(
k\left(\bm{x}\right) \frac{\partial u}{\partial x_\alpha} \right) 
= f\left(\bm{x}, t\right), \quad \bm{x} \in \Omega, \quad 0 < t \leq T .
\end{split}
\end{equation}

We consider a set of functions $u({\bm x})$ that satisfy the boundary condition
(\ref{37}). The transient convection-diffusion problem may be formulated
as the operator-differential equation
\begin{equation}\label{41}
 \frac{d u}{d t} + {\cal A} u = f(t), \quad {\cal A} =
{\cal C} + {\cal D}.
\end{equation}
Here $\cal D$ is the diffusive transport operator that is defined by the expression
\begin{equation}\label{42}
 {\cal D} u = - \sum_{\alpha = 1}^{2} \frac{\partial}{\partial x_\alpha} \left(
 k\left(\bm{x}\right) \frac{\partial u}{\partial x_\alpha} \right) .
\end{equation}

According to (\ref{36}), (\ref{39}), (\ref{40}), the convective transport operator
is written in distinct forms. For the  convective transport operator in the nondivergent form,
from (\ref{36}), we set ${\cal C} = {\cal C}_1$, where
\begin{equation}\label{43}
  {\cal C}_1 u =
  \sum_{\alpha = 1}^{2} v_\alpha\left(\bm{x}, t\right) \frac{\partial u}{\partial x_\alpha}.
\end{equation}
Similarly, from (\ref{39}),  we have ${\cal C} = {\cal C}_2$, where now
\begin{equation}\label{44}
 {\cal C}_2 u = 
 \sum_{\alpha = 1}^{2} \frac{\partial}{\partial x_\alpha}\left(v_\alpha\left(\bm{x}, t\right) u\right).
\end{equation}
Taking into account (\ref{40}), the convective-transport operator in the skew-symmetric form is
\[
{\cal C} = {\cal C}_0 = \frac{1}{2}
({\cal C}_1 + {\cal C}_2),
\]
and
\begin{equation}\label{45}
   {\cal C}_0 u = 
\frac{1}{2} \sum_{\alpha = 1}^{2} \left(
v_\alpha\left(\bm{x}, t\right) \frac{\partial u}{\partial x_\alpha} + 
 \frac{\partial }{\partial x_\alpha}\left(v_\alpha\left(\bm{x}, t\right)
u\right)
 \right).
\end{equation}
Now we highlight the basic properties of the above-mentioned
operators of diffusive and convective transport.

\subsection{Properties of the operators of diffusive and convective transport}

The solution of a discrete problem should inherit the basic properties of the corresponding differential
problem. This can be achieved, in particular, if the grid operators have the same primary
properties as the differential ones.

As usually, let ${\cal H} =  {L}_2 (\Omega)$ be a Hilbert space 
for arbitrary functions $u({\bm x})$ and $w({\bm x})$ equal zero on $\partial \Omega$.
The diffusive transport operator defined by (\ref{42}) is self-adjoint in ${\cal H}$ on the set
of functions satisfying the boundary conditions (\ref{37}):
\begin{equation}\label{46}
 {\cal D} = {\cal D}^*.
\end{equation}

Note also that the diffusive transport operator under consideration at the above-mentioned restrictions
is positive definite, i.e., the estimate
\begin{equation}\label{47}
 {\cal D} \ge \kappa_1 \lambda_0 {\cal E},
\end{equation}
is valid, 
where $\cal E$ denotes the identity operator and $\lambda_0  > 0$ is the minimal eigenvalue 
of the Laplace operator with the Dirichlet boundary conditions. For the rectangle $\Omega$,
we have
\[
\lambda_0 = \pi^2 \left (\frac {1}{l_1^2} + \frac {1}{l_2^2}
\right ) .
\]
The estimate (\ref{47}) follows from
\[
({\cal D}u,u) \ge
\kappa_1 (\nabla u,\nabla u)
  \ge \kappa_1\lambda_0 (u,u).
\]

We now consider the convective transport operator in various formulations 
(see (\ref{43}), (\ref{44}), and (\ref{45})).
Taking into account the homogeneous boundary conditions (\ref{37}), we have
\[
({\cal C}_1u,w) = \sum_{\alpha=1}^{2} \int \limits_\Omega v_\alpha
\frac {\partial u}{\partial x_\alpha} w d{\bm x} = \null -
\sum_{\alpha=1}^{2} \int \limits_\Omega \frac {\partial}{\partial
x_\alpha} (v_\alpha w) u d {\bm x} = - (u,{\cal C}_2 w).
\]
Thus, we see that the convective transport operators in the divergent and nondivergent
forms are the adjoints of each other (with a precision of the sign):
\begin{equation}\label{48}
 {\cal C}_1^* = - {\cal C}_2.
\end{equation}

In view of (\ref{48}), the convective transport operator in the skew-symmetric form (\ref{45}) 
is skew-symmetric ($({\cal C}_0 u,u) = 0$):
\begin{equation}\label{49}
 {\cal C}_0 = - {\cal C}_0^*.
\end{equation}

Under the condition of incompressibility
\begin{equation}\label{50}
 \div {\bm v} \equiv
\sum_{\alpha=1}^{2} \frac{\partial v_\alpha}{\partial x_\alpha} = 0, \quad {\bm x} \in
\Omega ,
\end{equation}
the convective transport operator in the nondivergent (\ref{43}) and divergent (\ref{44}) forms are also
skew-symmetric. In constructing discrete approximations of the convective transport
operators, the principal moment is that the skew-symmetric property of the operator ${\cal C}_0$ is valid
for any $v_\alpha ({\bm x},t)$, $\alpha = 1,2$ including the compressible case.

It seems useful to give the upper bound for the convective transport operator. 
For (\ref{43}), (\ref{44}), we have
\[
({\cal C}_1u,u) = -
({\cal C}_2u,u) = \frac 12 \sum_{\alpha=1}^{2} \int \limits_\Omega
v_\alpha \frac {\partial u^2}{\partial x_\alpha} d{\bm x} = -
\frac 12 \int \limits_\Omega u^2 \div {\bm v} \ d{\bm x}
\]
and therefore
\begin{equation}\label{51}
      \left|({\cal C}_\alpha u,u) \right| \le
      \frac 12 \|\div{\bm v} \|_{C(\Omega)}~ \|u\|^2,
      \quad \alpha = 1,2,
\end{equation}
where
\[
      \|w\|_{C(\Omega)}  = \max_{{\bm x} \in \Omega} |w({\bm x})|.
\]
Thus, for the convective transport operators defined in accordance with (\ref{43}), (\ref{44})
(${\cal C} = {\cal C}_\alpha, \ \alpha =1,2$), we obtain
\begin{equation}\label{52}
 \left|({\cal C}u,u)\right| \le
{\cal M}_1 \|u\|^2,
\end{equation}
where a constant ${\cal M}_1$ depends only on
$\div{\bm v}$ and, in accordance with (\ref{51}), it follows that 
\begin{equation}\label{53}
 {\cal M}_1 =
  \frac 12 \|\div{\bm v} \|_{C(\Omega)}.
\end{equation}

In addition, we present the estimates of subordination of the convective transport operator to the
diffusive transport operator:
\begin{equation}\label{54}
 \|{\cal C}u\|^2 \le {\cal M}_2
({\cal D}u,u),
\end{equation}
with a constant ${\cal M}_2$ depending on the velocity.

For the nondivergent operator of convection (\ref{43}), we have
\[
\begin{split}
\|{\cal C}_1 u\|^2  & = \int \limits_\Omega \left(
\sum_{\alpha=1}^{2} v_\alpha \frac {\partial u}{\partial x_\alpha}
\right)^2 d{\bm x} \le 2 \sum_{\alpha=1}^{2} \int \limits_\Omega
v_\alpha^2 \left(\frac {\partial u}{\partial x_\alpha} \right)^2
d{\bm x} \\
& \le 2 \max _\alpha \left\{ \|v_\alpha^2\|_{C(\Omega)}
\right\} \frac {1}{\kappa_1} \sum_{\alpha=1}^{2} \int \limits_\Omega
k \left(\frac {\partial u}{\partial x_\alpha} \right)^2 d{\bm x} \\
& \le \frac {2}{\kappa_1} \max _\alpha
\left\{\|v_\alpha^2\|_{C(\Omega)}\right\}~ ({\cal D} u,u),
\end{split}
\]
i.e., in the inequality (\ref{54}), for ${\cal C}= {\cal C}_1$, we can assume
\begin{equation}\label{55} 
{\cal M}_2 = \frac {2}{\kappa_1} \max _\alpha
\left\{\|v_\alpha^2\|_{C(\Omega)}\right\}.
\end{equation}

Similarly, for ${\cal C}= {\cal C}_2$ (see (\ref{44})), we obtain
\[
\begin{split}
\|{\cal C}_2 u\|^2 & = \int \limits_\Omega \left(
\sum_{\alpha=1}^{2} v_\alpha \frac {\partial u}{\partial x_\alpha}
+ \div  {\bm v} u \right)^2 d{\bm x} \\
& \le 2 \int \limits_\Omega \left(
\sum_{\alpha=1}^{2} v_\alpha \frac {\partial u}{\partial x_\alpha}
\right)^2 d{\bm x} + 2 \int \limits_\Omega (\div  {\bm v})^2 u^2
d{\bm x}.
\end{split}
\]

Taking into account the Friedrichs inequality
\begin{equation}\label{56}
 \int \limits_\Omega u^2 d {\bm x} \le
{\cal M}_0 \sum_{\alpha=1}^{2} \int \limits_\Omega
\left(\frac{\partial u}{\partial x_\alpha} \right)^2 d{\bm x},
\end{equation}
where the constant ${\cal M}_0= \lambda_0^{-1}$, we obtain at
${\cal C}= {\cal C}_2$ the estimatr (\ref{54}) with
\begin{equation}\label{57}
 {\cal M}_2 = \frac {2}{\kappa_1} \left (2
\max_\alpha \left\{\|v_\alpha^2\|_{C(\Omega)}\right\}+ {\cal M}_0
\|\div{\bm v}\|_{C(\Omega)}^2 \right ).
\end{equation}

In a similar way, for the case ${\cal C}= {\cal C}_0$, we have
\[
\|{\cal C}_0 u\|^2 = \frac 14 \|{\cal
C}_1 u + {\cal C}_2 u\|^2 \le \frac 12 \|{\cal C}_1 u\|^2 + \frac
12 \|{\cal C}_2 u\|^2,
\]
i.e.,
\begin{equation}\label{58}
{\cal M}_2 = \frac {1}{\kappa} \left (3 \max _\alpha
\left\{\|v_\alpha^2\|_{C(\Omega)}\right\}+ {\cal M}_0 \|\div{\bm
v}\|_{C(\Omega)}^2 \right ).
\end{equation}
The above estimates (\ref{52}), (\ref{54})  serve as a reference point in studying discrete analogs of the
convective transport operator.

Summarizing the above properties of the convective transport operator, we obtain the
following statement.

\begin{theorem}\label{t-1}
The convective transport operators have the following properties:
\begin{itemize}
\item
the convective transport operators in the divergent and nondivergent forms are adjoint
to each other up to the sign --- the equality (\ref{48});
\item 
the convective transport operator in the skew-symmetric form is skew-symmetric
--- the equality(\ref{49});
\item
the convective transport operators in the divergent and nondivergent forms are
bounded --- the a priori estimates (\ref{52}) and (\ref{53});
\item
the convective transport operators are subordinated to the diffusion operator -- the estimate
(\ref{54}) with the constant ${\cal M}_2 $, defined according to (\ref{55}), (\ref{57}), (\ref{58}).
\end{itemize}
\end{theorem}

It seems reasonable to construct difference operator of convective and diffusive transport in such
a way that they do have the same properties.

\subsection{A priori estimates}

We restrict ourselves to elementary a priori estimates for the time-dependent
equation (\ref{41}) supplemented with the initial condition
\begin{equation}\label{59}
  u(0) = u^0 .
\end{equation}
They are based on the above-established properties (see Theorem~\ref{t-1}) of the operators
of diffusive and convective transport.

\begin{theorem}\label{t-2}
For the solution of the problem (\ref{41}), (\ref{59}),
under the conditions (\ref{47}), (\ref{52}), (\ref{54}),
the following a priori estimate holds:
\begin{equation}\label{60}
\begin{split}
  \|u(t)\|^2 & \le \exp (2{\cal M}_1 t) \|u_0\|^2 \\
  & +
  \frac{1}{2}  \int\limits_0^t \exp (2{\cal M}_1 (t-\theta))
  \|f(\theta)\|^2_{{\cal D}^{-1}} d\theta,
\end{split}
\end{equation} 

\begin{equation}\label{61}
\begin{split}
  \|u(t)\| & \le \exp (\frac 14 {\cal M}_2 t) \|u_0\| \\
  & +
  \int\limits_0^t \exp (\frac 14 {\cal M}_2 (t-\theta)) \|f(\theta)\| d\theta,
\end{split}
\end{equation} 

\begin{equation}\label{62}
\begin{split}
  \|\nabla u(t)\|^2 & \le \frac {\kappa_2}{\kappa_1}
  \exp ({\cal M}_2 t) \|\nabla u_0\|^2 \\
  & +
  \frac {1}{\kappa}
  \int\limits_0^t \exp ({\cal M}_2 (t-\theta)) \|f(\theta)\|^2 d\theta,
\end{split}
\end{equation} 
where $\kappa_1 \le k({\bm x}) \le \kappa_2, \ {\bm x}\in \Omega \cup \partial \Omega$.
\end{theorem}
\proof
Multiplying equation (\ref{41}) scalarly by $u(t)$, we get
\begin{equation}\label{63}
  \frac{1}{2} \frac {d}{dt} \|u\|^2 + ({\cal D}u,u) = -
  ({\cal C}u,u) + (f,u).
\end{equation} 
Taking into account (\ref{52}) and the inequality
\[
  (f,u) \le ({\cal D}u,u) + \frac 14 \|f\|^2_{{\cal D}^{-1}} ,
\]
from (\ref{63}), it follows that
\[
  \frac {d}{dt} \|u\|^2 \le 2{\cal M}_1 \|u\|^2 +
  \frac 12 \|f\|^2_{{\cal D}^{-1}} .
\]
Using Gronwall's lemma, we obtain from this inequality the required estimate (\ref{60}).

From (\ref{54}), we have
\[
  |- ({\cal C}u,u)| \le
  \|{\cal C}u\| \|u\| \le
  ({\cal D}u,u)  + \frac 14 {\cal M}_2 \|u\|^2.
\]
This allows to obtain from (\ref{63}) the inequality
\[
  \frac {d}{dt} \|u\| \le \frac 14 {\cal M}_2 \|u\| + \|f\|,
\]
which immediately implies the estimate (\ref{61}).

It remains to derive the estimate (\ref{62}). To do this, multiply equation
(\ref{41}) scalarly by $du/dt$ and obtain
\[
  \left\|\frac {du}{dt} \right\|^2 +
  \frac{1}{2} \frac {d}{dt} ({\cal D}u,u) =
  - \left ({\cal C} u, \frac {du}{dt} \right ) + \left (f, \frac {du}{dt} \right ) .
\]
For the right-hand side, we have
\[
  - \left ({\cal C} u, \frac {du}{dt} \right ) + \left (f, \frac {du}{dt} \right ) \le
  \left\|\frac {du}{dt} \right\|^2 +
  \frac{1}{2}  \|{\cal C} u\|^2 + \frac{1}{2} \|f\|^2.
\]
In view of (\ref{54}), we get the inequality
\begin{equation}\label{64}
  \frac {d}{dt}({\cal D} u,u) \le
  {\cal M}_2 ({\cal D}u,u) + \|f\|^2.
\end{equation} 
By
\[
  \kappa_1 \|\nabla u\|^2 \le ({\cal D}u,u)
  \le \kappa_2 \|\nabla u\|^2,
\]
from (\ref{64}), we obtain the desired estimate (\ref{62}).
\qed

The resulting estimates (\ref{60})--(\ref{62}) provide the continuity of the solution
of (\ref{41}), (\ref{59}) with respect to the initial data and the right-hand side. In these
estimates, the essential issue is that for the solution norm of the problem with
the homogeneous right-hand side, it is allowed an exponential growth with a
growth increment that depends on the constants ${\cal M}_1,~ {\cal M}_2$. It is necessary to
allow such a behavior for the solution at the discrete level. Thus, we need to
introduce the concept of $\varrho$-stability for the corresponding difference schemes.

\subsection{The maximum principle and a priori estimates}

Considering boundary value problems both for parabolic equations of the
second order in space and for second-order elliptic equations, special attention is
paid to the maximum principle \cite{protter_maximum_1967}. The corresponding statement
is formulated as follows.

\begin{theorem}\label{t-3}
Assume that in the Cauchy problem (\ref{41}), (\ref{59}) the right-hand side 
$f({\bm x},t) > 0~~ (f({\bm x},t) < 0)$ and the initial data $u_0({\bm x}) \ge 0~~ (u_0({\bm x})
\le 0)$, then $u({\bm x},0) \ge 0~~ (u({\bm x},0) \le 0)$ 
for all ${\bm x} \in \Omega$ and $ t > 0$.
\end{theorem}

Note that it is possible to use the maximum
principle in a stronger form that employs weak inequalities for the
right-hand side, i.e., the non-negativity of the solution takes place under
the condition of the non-negativity of the right-hand side and the initial data.

Here are some a priori estimates for the convection-diffusion problem (\ref{41}), (\ref{59}),
which are derived via the maximum principle. In the above-considered time-dependent
problems with Dirichlet boundary conditions, we can easily construct
a majorant function~--- a wide range of the appropriate estimates is
given, e.g.,  in the book~\cite{ladyzhenskaia_linear_1968}.
We also give an estimate for the convection-diffusion equation with convective
terms in the divergence form~--- the estimate in ${\cal L}_1{\Omega})$.

\begin{theorem}\label{t-4}
The solution of the problem (\ref{36})--(\ref{38})
satisfies the following a priori estimate in ${\cal L}_\infty{\Omega})$:
\begin{equation}\label{65}
  \|u({\bm x},t)\|_\infty \le
  \|u({\bm x},0)\|_\infty +
  \int_0^t \|f({\bm x},\theta)\|_\infty \,d\theta ,
\end{equation} 
whereas the solution of the problem (\ref{37})--(\ref{39}) satisfies the a priori estimate 
\begin{equation}\label{66}
  \|u({\bm x},t)\|_1 \le
  \|u({\bm x},0)\|_1 + \int_0^t \|f({\bm x},\theta)\|_1 d\theta .
\end{equation} 
\end{theorem}

The estimates (\ref{65}), (\ref{66}) complement the above a priori estimates
(\ref{60})--(\ref{62}) in the Hilbert spaces ${\cal L}_2(\Omega)$ and $W_2^1(\Omega)$.

\section{Discretization in space} 

In numerical solving transient problems, firstly we construct
discretization in space. The resulting operator-differential equation
should inherit the basic properties of the differential problem, i.e,
we speak of the positiveness (non-negativity) and
self-adjointness of the diffusive transport operator as well as
the adjointness of the convective transport operators in the corresponding finite-dimensional
Hilbert spaces. We consider the standard finite difference approximations
for model unsteady convection-diffusion problems in a rectangular domain.
In addition, we discuss the problem of constructing approximations by means of
the finite volume method.

\subsection{Difference operators} 

In a rectangle $\Omega$, we introduce a uniform in each direction grid.
For grids in individual directions $x_\alpha,~ \alpha =1,2$,
we use notation
\[
   \bar{\omega}_\alpha = \{x_\alpha~ |~ x_\alpha = i_\alpha h_\alpha,
   \quad i_\alpha = 0,1, \, \ldots, \, N_\alpha,
   \quad N_\alpha h_\alpha = l_\alpha\} ,
\]
where $\omega$ stands for the set of interior nodes.
On the set of grid functions that are equal to zero on the set of boundary nodes
$\partial \omega$ $(\bar{\omega} = \bar{\omega}_1 \times \bar{\omega}_2 =  \omega \cup \partial \omega)$, 
we define the Hilbert space $H = L_2({\omega})$ with the following scalar product and norm:
\[
  (y,w) \equiv \sum_{{\bm x} \in \omega}
  y({\bm x}) w({\bm x}) h_1 h_2,
  \quad \|y\| \equiv (y,y)^{1/2} .
\]

We use the standard index-free notations of the theory of difference schemes~\cite{2001Samarskii}.
For the backward difference derivative, we have the representation
\[
  u_{\overline x } \equiv \frac{u_i-u_{i-1}}h .
\]
Similarly, for the forward difference derivative, we get

\[
  u_x\equiv \frac{u_{i+1}-u_i}{h} .
\]
Using the three-point stencil (nodes $x_{i-1}$, $x_i$, $x_{i+1}$), we can employ
the central difference derivative
\[
  u_{\stackrel{\circ }{x}}\equiv \frac{u_{i+1}-u_{i-1}}{2h}.
\]
For the operator of the second derivative, we have
\[
  u_{\overline xx} = \frac{u_x- u_{\overline x}} h = \frac{u_{i-1}-
2u_i+ u_{i+1}}{h^2}. 
\]

The 2D difference diffusive transport operator is represented as the sum of the 1D ones:
\begin{equation}\label{67}
  D  = \sum_{\alpha =1}^{2} D^{(\alpha)},
  \quad D^{(\alpha)} y = - (a^{(\alpha)}
  y^{\bar {x}_\alpha})_{x_\alpha} ,~ \alpha =1,2,~
  {\bm x} \in \omega .
\end{equation}
For smooth diffusion coefficients, we can assume
\[
  a^{(1)} ({\bm x}) = k(x_1-0.5h_1, x_2),
  \quad a^{(2)} ({\bm x}) = k(x_1,x_2-0.5h_2) .
\]

Properties of the elliptic difference operator $D$
are well-known \cite{2001Samarskii,SamNik}.
For the 2D self-adjoint operator $D$, we have the lower bound
\begin{equation}\label{68}
  D = D^* \ge  \frac{1}{M_0} \kappa_1 E,
  \quad M_0 = \frac {8}{l_1^2} + \frac {8}{l_2^2}.
\end{equation}
We present also the upper bound for the diffusive transport operator, i.e.,
\begin{equation}\label{69}
  D \le M_3 E
\end{equation} 
with the constant
\[
\begin{split}
  M_3 & = \frac {4}{h_1^2} \max_{{\bm x} \in \omega}
   \frac {a^{(1)}({\bm x}) + a^{(1)}(x_1+h_1,x_2)}{2} \\
  & + \frac {4}{h_2^2} \max_{{\bm x} \in \omega}
   \frac {a^{(2)}({\bm x}) + a^{(2)}(x_1,x_2+h_2)}{2} .
\end{split}
\]

Now we consider the difference analogs of the differential convective transport operators written
in various forms. For the operator in the nondivergent form ${\cal C}_1$,
we put into the correspondence the 2D the difference convective transport operator
\begin{equation}\label{70}
  C_1  = \sum_{\alpha =1}^{2} C_1^{(\alpha)},
  \quad C_1^{(\alpha)} y =
  b^{(\alpha)} y^{\stackrel{\circ}{x}_\alpha} ,
  \quad \alpha =1,2,
  \quad {\bm x} \in \omega .
\end{equation}
In the simplest case of smooth enough convective transport coefficients, we assume
\[
  b^{(\alpha)} ({\bm x},t) = v_\alpha ({\bm x},t),
  \quad {\bm x} \in \omega .
\]

Similarly, for the approximation of the convective transport operator in the divergent form ${\cal C}_2$,
we employ the difference operator
\begin{equation}\label{71}
  C_2  = \sum_{\alpha =1}^{2} C_2^{(\alpha)},
  \quad C_2^{(\alpha)} y =
  (b^{(\alpha)}  y)_{\stackrel{\circ}{x}_\alpha} ,
  \quad \alpha =1,2,
  \quad {\bm x} \in \omega .
\end{equation} 

The approximation of the 2D convective transport operator in the skew-symmetric form is based on
the representation $C_0 = 0.5 (C_1 + C_2)$ such that
\begin{equation}\label{72}
\begin{split}
  C_0  = & \sum_{\alpha =1}^{2} C_0^{(\alpha)} , 
  \quad C_0^{(\alpha)} y = \frac{1}{2} 
  (b^{(\alpha)} y^{\stackrel{\circ}{x}_\alpha} +
  (b^{(\alpha)}  y)_{\stackrel{\circ}{x}_\alpha}) , \\
  & \alpha =1,2,
  \quad {\bm x} \in \omega . 
\end{split}
\end{equation} 

\begin{lemma}\label{l-1}
The difference operators $C_\alpha,~ \alpha =0,1,2$
have the following properties of adjointness:
\begin{equation}\label{73}
   C_1^* = - C_2,
   \quad C_0^* = - C_0
\end{equation} 
in the space of grid functions $H$.
\end{lemma}
\proof
It is easy to check directly that the 1D convective transport operators in the divergent
and nondivergent forms are adjoint to each other up to the sign. Taking this into account,
we have
\[
\begin{split}
  (C_1y,w) &= (C_1^{(1)} y,w) + (C_1^{(2)} y,w) \\
  & = -(y,C_2^{(1)} w) - (y, C_2^{(2)} w) \\
  &= -(y,C_2 w) = (y,C_1^* w).
\end{split}  
\]
The skew-symmetric property of the operator $C_0$ follows from its definition as the half-sum of
the operators $C_1$ and $C_2$.
\qed

Similar properties can be proved for the 2D convective transport operators constructed with the
use of coefficients $v_\alpha({\bm x},t),~ \alpha = 1,2$  shifted by a half-step in the 
corresponding directions. Such staggered grids are in common use 
in computational fluid dynamics \cite{TannehillAndersonPletcher1997,PeyretTaylor1983}.

\begin{figure}[ht] 
  \begin{center}
  \begin{tikzpicture}[scale=0.75]
   \draw (0,0) -- (8,0);
   \draw (4,-4) -- (4,4);
   \filldraw [black] (0,0) circle (0.1);
   \filldraw [black] (4,0) circle (0.1);
   \filldraw [black] (8,0) circle (0.1);
   \filldraw [black] (4,-4) circle (0.1);
   \filldraw [black] (4,4) circle (0.1);
   \draw [black] (2-0.1,0-0.1) rectangle +(0.2,0.2);
   \draw [black] (6-0.1,0-0.1) rectangle +(0.2,0.2);
   \draw [black] (4,-2) circle (0.1);
   \draw [black] (4,2) circle (0.1);
   \draw [dotted] (2,-2) --(6,-2) -- (6,2) -- (2,2) -- (2,-2);
   \draw(7,-0.5) node {$x_1$};
   \draw(3.5,3) node {$x_2$};  
   \end{tikzpicture}
   \caption{Control volume: $\Box$ --- node for
    $v_1({\bm x}),t$; \quad   $ \ocircle$   --- node for $v_2({\bm x},t)$} 
   \label{f-1}
  \end{center}
\end{figure}

Let us refer the convective-transport coefficient $v_1({\bm x},t)$ with respect to the variable $x_1$ to
the nodes of the grid which is shifted by a half-step along this direction. The grid for the
coefficient $v_2({\bm x},t)$ is shifted along $x_2$ by $0.5h_2$ (see Fig.~\ref{f-1}).

For the difference convective transport operator in the nondivergent form, we get
\[
  C_1^{(1)} y = \frac{1}{2} 
  (b^{(1)}(x_1-0.5h_1,x_2,t) y^{\bar {x}_1} +
  b^{(1)}(x_1+0.5h_1,x_2,t) y^{x_1}),
\]
\[
  C_1^{(2)} y = \frac{1}{2} 
  (b^{(2)}(x_1,x_2-0.5h_2,t) y^{\bar {x}_2} +
  b^{(2)}(x_1,x_2+0.5h_2,t) y^{x_2}),
\]
\begin{equation}\label{74}
  C_1  = \sum_{\alpha =1}^{2} C_1^{(\alpha)} ,
 \quad  {\bm x} \in \omega .
\end{equation} 

For the difference convective transport operator in the divergent form, we employ the
representation
\begin{equation}\label{75}
  C_2 y = C_1 y + \left (b^{(1)} _{\stackrel{\bullet}{x_1}} +
  b^{(2)} _{\stackrel{\bullet}{x_2}} \right ) y ,
  \quad {\bm x} \in \omega .
\end{equation} 

The following notation is used here for the difference
derivative of the grid function given at half-integer nodes:
\[
  b_{\stackrel{\bullet}{x}} \equiv
  \frac {b(x+0.5h) - b(x-0.5h)}{h}.
\]
This expression is a difference analog of the differential equality
\[
  {\cal C}_2 u = {\cal C}_1 u + \div {\bm v}~ u
\]
with a special approximation of $\div {\bm v}$.

For the skew-symmetric convective transport operator
$C_0 = 0.5(C_1 + C_2)$, from (\ref{74}), (\ref{75}), we obtain
\[
\begin{split}
C_0^{(1)} y & = \frac {1}{2h_1}
  b^{(1)}(x_1+0.5h_1,x_2,t) y(x_1+h_1,x_2) \\
  & - \frac {1}{2h_1}
  b^{(1)}(x_1-0.5h_1,x_2,t) y(x_1-h_1,x_2),
\end{split}  
\]
\[
\begin{split}
  C_0^{(2)} y & = \frac {1}{2h_2}
  b^{(2)}(x_1,x_2+0.5h_2,t) y(x_1,x_2+h_2) \\
  & - \frac {1}{2h_2}
  b^{(2)}(x_1,x_2-0.5h_2,t) y(x_1,x_2-h_2),
\end{split}   
\]
\begin{equation}\label{76}
  C_0  = \sum_{\alpha =1}^{2} C_0^{(\alpha)} ,
 \quad  {\bm x} \in \omega .
\end{equation} 
For the convective transport operators $C_\alpha,~ \alpha =0,1,2$,
defined by (\ref{74})--(\ref{76}), Lemma~\ref{l-1} holds.

In the multidimensional case, the inequality
\begin{equation}\label{77}
 |(C_\alpha y,y) | \le M_1 \|y\|^2,
 \quad \alpha =1,2
\end{equation} 
with a constant $M_1$, independent of the grid steps, is also valid for the convective transport
operators under consideration. For operators (\ref{70}), (\ref{71}), we obtain a constant $M_1$, which
depends on the first derivatives of the convective transport coefficients, whereas in the case
(\ref{74}), (\ref{75}) it depends on the divergence, as in the continuous case. We formulate the corresponding
statement using the following notation for the grids in separate directions:
\[
   \omega_\alpha = \{x_\alpha~ |~ x_\alpha = i_\alpha h_\alpha,
   \quad i_\alpha = 1,2, \, \ldots, \, N_\alpha-1,
   \quad N_\alpha h_\alpha = l_\alpha\} ,
\]
\[
   \omega^+_\alpha = \{x_\alpha~ |~ x_\alpha = i_\alpha h_\alpha,
   \quad i_\alpha = 1,2, \, \ldots, \, N_\alpha,
   \quad N_\alpha h_\alpha = l_\alpha\} ,
   \quad \alpha = 1,2 .  
\]

\begin{lemma}\label{l-2}
For the difference convective transport operators $C_\alpha, \ \alpha =1,2,$, 
defined according to (\ref{70}), (\ref{71}),
the estimates (\ref{77}) hold with the constant
\begin{equation}\label{78}
  M_1 = \frac{1}{2} \max_{x_1 \in \omega^+_1}
  \max_{x_2 \in \omega_2} |b^{(1)}_{\bar{x}_1}| +
  \frac{1}{2} \max_{x_1 \in \omega_1}
  \max_{x_2 \in \omega^+_2} |b^{(2)}_{\bar{x}_2}| ,
\end{equation} 
whereas for (\ref{74}), (\ref{75})  --- with the constant
\begin{equation}\label{79}
  M_1 = \frac{1}{2}  \max_{{\bm x} \in \omega}
  \left | b^{(1)} _{\stackrel{\bullet}{x_1}} +
  b^{(2)} _{\stackrel{\bullet}{x_2}} \right | .
\end{equation} 
\end{lemma}
\proof
Taking into account (\ref{74}), (\ref{75}),  we get
\[
\begin{split}
(C_1 y,y) & = -(C_2 y,y) =
   \frac{1}{2}  ((C_1 y,y) + (y,C_1^*y)) \\
  & = - \frac{1}{2}  \left ( \left (b^{(1)} _{\stackrel{\bullet}{x_1}} +
  b^{(2)} _{\stackrel{\bullet}{x_2}} \right ) y,y \right ) .
\end{split}
\]
This implies the estimate (\ref{77}), (\ref{79}).

For the operators (\ref{70}), (\ref{71}), we  use the corresponding estimates for the 1D operators, too.
For instance, for the 2D convective transport operator in the nondivergent form, we have
\[
\begin{split}
  | (C_1 y,y) | & = \left | ((C_1^{(1)} + C_1^{(2)})y,y) \right | \\
  & = \left | \sum_{x_1 \in \omega_1} \sum_{x_2 \in \omega_2}
  (C_1^{(1)} + C_1^{(2)})y~ y h_1 h_2 \right | \\
  & \le \frac{1}{2} \max_{x_1 \in \omega^+_1}
  \max_{x_2 \in \omega_2} |b^{(1)}_{\bar{x}_1}| (y,y) +
  \frac{1}{2} \max_{x_1 \in \omega_1}
  \max_{x_2 \in \omega^+_2} |b^{(2)}_{\bar{x}_2}| (y,y) ,
\end{split}
\]
i. e., we arrive at the desired estimate (\ref{77}), (\ref{78}).
\qed

Finally, let us consider the subordination property of the convective transport operator
to the diffusive transport operator under the standard restrictions
$k({\bm x}) \ge \kappa_1 > 0$.

\begin{lemma}\label{l-3}
For the 2D convective transport difference operators $C=C_\alpha, \ \alpha=0,1,2$,
the following estimates hold:
 \begin{equation}\label{80}
  \|Cy\|^2 \le M_2 (Dy,y),
\end{equation} 
where the constant $M_2$ for the operators $C_1$, defined according to (\ref{70}), (\ref{74}), 
is, respectively,
\[
M_2 = \frac {2}{\kappa_1} \max_{\alpha}
\max_{{\bm x} \in \omega} (b^{(\alpha)}({\bm x},t))^2,
\]
\[
\begin{split}
M_2 = \frac {2}{\kappa_1} \max & \left \{
\max_{{\bm x} \in \omega^+_1\times\omega_2}
(b^{(1)}(x_1-0.5h_1,x_2,t))^2, \right. \\
& \left.
\max_{{\bm x} \in \omega_1\times\omega_2^+}
(b^{(2)}(x_1,x_2-0.5h_2,t))^2 \right \},
\end{split}
\]
for the operators (\ref{71}), (\ref{75}) ---
\[
  M_2 = \frac {2}{\kappa_1} \left ( 2 \max_{\alpha}
  \max_{{\bm x} \in \bar {\omega}} \left (b^{(\alpha)}({\bm x},t)\right )^2 +
  M_0 \max_{{\bm x} \in \omega}
  \left (b^{(1)}_{\stackrel{\circ}{x}_1} +
  b^{(2)}_{\stackrel{\circ}{x}_2} \right )^2 \right ) ,
\]
\[
  M_2 = \frac {2}{\kappa_1} \left (2
  \max \left \{
  \max_{{\bm x} \in \omega^+_1\times\omega_2}
  (b^{(1)}(x_1-0.5h_1,x_2,t))^2, \right. \right.
\]
\[
\left.
\max_{{\bm x} \in \omega_1\times\omega_2^+ }
(b^{(2)}(x_1,x_2-0.5h_2,t))^2 \right \} 
\]
\[
  \left.
  + M_0 \max_{{\bm x} \in \omega}
  \left (b^{(1)}_{\stackrel{\bullet}{x}_1} +
  b^{(2)}_{\stackrel{\bullet}{x}_2} \right )^2 \right ) ,
\]
and for the operators (\ref{72}), (\ref{76}) ---
\[
  M_2 = \frac {1}{\kappa_1} \left ( 3 \max_{\alpha}
  \max_{{\bm x} \in \bar {\omega}} \left (b^{(\alpha)}({\bm x},t) \right )^2 +
  M_0 \max_{{\bm x} \in \omega}
  \left (b^{(1)}_{\stackrel{\circ}{x}_1} +
  b^{(2)}_{\stackrel{\circ}{x}_2} \right )^2 \right ) ,
\]
\[
M_2 =   \frac {1}{\kappa_1} \left (3
\max \left \{
\max_{{\bm x} \in \omega^+_1\times\omega_2}
(b^{(1)}(x_1-0.5h_1,x_2,t))^2, \right. \right. 
\]
\[
 \left.
\max_{{\bm x} \in \omega_1\times\omega_2^+}
(b^{(2)}(x_1,x_2-0.5h_2,t))^2 \right \} 
\]
\[
  \left.
  + M_0 \max_{{\bm x} \in \omega}
  \left (b^{(1)}_{\stackrel{\bullet}{x}_1} +
  b^{(2)}_{\stackrel{\bullet}{x}_2} \right )^2 \right ) ,
\]
where $M_0$ --- is the constant from ( (\ref{68}).
\end{lemma}
\proof
Taking into account the inequality\[
  \left (\sum_{i=1}^{p} a_i \right )^2 \le p \sum_{i=1}^{p} a_i^2 ,
\]
for operator (\ref{70}), we have
\[
\begin{split}
  \|C_1 y\|^2 & = \sum_{{\bm x} \in \omega}
  \frac {1}{4} \left (b^{(1)}({\bm x},t) (y^{\bar{x}_1} + y^{x_1}) +
  b^{(2)}({\bm x},t) (y^{\bar{x}_2} + y^{x_2}) \right )^2 h_1 h_2 \\
  & \le 2 \max_{\alpha}
  \max_{{\bm x} \in \omega} (b^{(\alpha)}({\bm x},t))^2
  \|\nabla y\|^2 \le M_2 (Dy,y) .
\end{split}  
\]
The estimate for the operator (\ref{74}) is obtained in the same manner.

For the discrete operator (\ref{71}), we employ the representation
\[
\begin{split}
  C_2 y & = \frac{1}{2} \left (b^{(1)} (x_1-h_1,x_2,t) y^{\bar {x}_1} +
  b^{(1)} (x_1+h_1,x_2,t) y^{x_1}  \right. \\
  & \left.
  + \  b^{(2)} (x_1,x_2-h_2,t) y^{\bar {x}_2} +
  b^{(2)} (x_1,x_2+h_2,t) y^{x_2} \right ) \\
  &  + \left (b^{(1)}_{\stackrel{\circ}{x}_1} +
   b^{(2)}_{\stackrel{\circ}{x}_2} \right ) y .
\end{split}     
\]
Thus,
\[
\begin{split}
 \|C_2 y\|^2 & \le
  2 \sum_{{\bm x} \in \omega} (b^{(1)} (x_1-h_1,x_2))^2
  (y^{\bar {x}_1})^2 h_1 h_2 \\
  & +
  2 \sum_{{\bm x} \in \omega} (b^{(1)} (x_1+h_1,x_2))^2
  (y^{x_1})^2 h_1 h_2 \\
  & +
  2 \sum_{{\bm x} \in \omega} (b^{(2)} (x_1,x_2-h_2))^2
  (y^{\bar {x}_2})^2 h_1 h_2 \\
  & +
  2 \sum_{{\bm x} \in \omega} (b^{(2)} (x_1,x_2+h_2))^2
  (y^{x_2})^2 h_1 h_2 \\
  & + 2 \max_{{\bm x} \in \omega}
  (b^{(1)}_{\stackrel{\circ}{x}_1} +
  b^{(2)}_{\stackrel{\circ}{x}_2} )^2 (y,y) .
\end{split}    
\]
In view of the Friedrichs inequality, we obtain the estimate (\ref{80}) with the constant
$M_2$ given in the lemma.

For the grid operator in the divergent form with coefficients on staggered grids, on the
basis of (\ref{75}), we have
\[
  \|C_2 y\|^2 = 2\|C_1 y\|^2 +
  2 \left \| \left (b^{(1)} _{\stackrel{\bullet}{x_1}} +
  b^{(2)} _{\stackrel{\bullet}{x_2}} \right ) y \right \| ^2.
\]
For the first term in the right-hand side, we use the already derived estimate for $C_1$, 
whereas for the second one we apply the Friedrichs inequality.

Subordination estimates for the difference 2D convective transport operators in the skew-symmetric
form $C_0$ are derived from the estimates for the operators$C_\alpha,~ \alpha =1,2$.
\qed

The above values for the subordination constant $M_2$, in spite of their awkwardness,
demonstrate the fundamental independence of this constant of the computational grid. The
constant $M_2$ depends on the values of the convective transport coefficients 
$v_\alpha({\bm x},t),~ \alpha = 1,2$
(the velocity) and on $\div {\bm v}$, to be more precise, on their difference approximations.

In numerical solving the problem (\ref{41}), (\ref{59}),
using the above discretization in space, we obtain the operator-differential
equation
\begin{equation}\label{81}
  \frac{d y}{d t} + A y = \varphi(\bm{x},t),
  \quad A = C + D, 
  \quad  \bm{x} \in \omega,  
  \quad 0 < t \leq T,  
\end{equation} 
\begin{equation}\label{82}
  y(\bm{x},0) = u^0(\bm{x}),
  \quad  \bm{x} \in \omega .
\end{equation}
To investigate this semi-discrete problem, we employ the above properties
of the difference operators of convective and diffusive transport.
In particular, we can obtain analogues of a priori estimates of Theorem~\ref{t-2}.
 
\subsection{Monotone schemes for 2D convection-diffusion problems}

For 2D difference convection-diffusion problems in nondivergent and
divergent forms, the maximum principle is formulated. Unconditionally
monotone schemes are constructed on the basis
of the regularization principle  for difference schemes.

To simplify the material presentation,
we will consider difference schemes for
stationary 2D convection-diffusion equation on
uniform rectangular grids. The corresponding discrete analogues
on the standard five-point stencil \textit{cross} are written in the following form:
\begin{equation}\label{83}
\begin{split}
  \gamma({\bm x}) y({\bm x}) &- \alpha_1({\bm x}) y(x_1-h_1,x_2) -
  \beta_1 ({\bm x}) y(x_1+h_1,x_2) \\
  &- \alpha_2({\bm x})  y(x_1,x_2-h_2) -
  \beta_2 ({\bm x}) y(x_1,x_2+h_2) =
  \varphi({\bm x}),
  \quad {\bm x} \in \omega .
\end{split}
\end{equation}
This grid equations are considered with the boundary conditions
\begin{equation}\label{84}
   y({\bm x}) = 0,
  \quad {\bm x} \in \partial\omega .
\end{equation}
Assume that the coefficients of the difference scheme (\ref{83}) satisfy the conditions
\begin{equation}\label{85}
  \alpha_j ({\bm x}) > 0,
  \quad \beta_j ({\bm x}) > 0,
  \quad j = 1,2,
  \quad \gamma ({\bm x}) > 0,
  \quad {\bm x} \in \omega
\end{equation}
and suppose that
\[
  \alpha_j ({\bm x}) > 0,
  \quad \beta_j ({\bm x}) > 0,
  \quad j = 1,2,
  \quad {\bm x} \in \partial\omega .
\]

We highlight two classes of the monotone difference schemes (\ref{83}), (\ref{84}),
i.e., the schemes that satisfy the difference maximum principle.

\begin{theorem}\label{t-5}
Assume that in the scheme (\ref{83})--(\ref{85})
$\varphi ({\bm x}) \ge 0$ for all ${\bm x} \in \omega$
(or $\varphi ({\bm x}) \le 0$ for ${\bm x} \in \omega$).
Then for
\begin{equation}\label{86}
  \gamma ({\bm x}) \ge
  \alpha_1 ({\bm x}) + \alpha_2 ({\bm x}) +
  \beta_1 ({\bm x}) + \beta_2 ({\bm x}),
  \quad {\bm x} \in \omega
\end{equation} 
or for
\begin{equation}\label{87}
\begin{split}
  \gamma ({\bm x}) &\ge
  \alpha_1 (x_1+h_1,x_2) + \beta_1 (x_1-h_1,x_2) \\
  & +
  \alpha_2 (x_1,x_2+h_2) + \beta_2 (x_1,x_2-h_2),
  \quad {\bm x} \in \omega
\end{split}
\end{equation} 
we have that
$y ({\bm x}) \ge 0$ for all ${\bm x} \in \omega$
($y({\bm x}) \le 0$ for all ${\bm x} \in \omega$).
\end{theorem}
\proof
As usual, the argument is by reductio ad absurdum. Suppose, e.g.,
that the conditions (\ref{86}) are valid and the solution of equation (\ref{83})
with the non-negative right-hand side is not non-negative at all grid points.
Let $x^*$ be an interior grid point, where
the solution has the minimal negative value. If this value 
is achieved at several points, then we choose
the grid point, where $y(x_1^*-h_1,x^*_2) > y(x^*)$.
We write equation (\ref{83}) at this point as
\[
\begin{split}
 \gamma({\bm x}^*) y({\bm x}^*)  &-
  \alpha_1({\bm x}^*) y(x_1^*-h_1,x_2^*) -
  \beta_1 ({\bm x}^*) y(x_1^*+h_1,x_2^*) \\
  &  - \alpha_2({\bm x}^*)  y(x_1^*,x_2^*-h_2) -
  \beta_2 ({\bm x}^*) y(x_1^*,x_2^*+h_2) \\
  & = \varphi({\bm x}^*),
  \quad {\bm x}^* \in \omega .
\end{split}
\]
Under the theorem conditions, the right-hand side
is non-negative, whereas the left-hand side, in view of (\ref{85}), (\ref{86}),
\[
  \begin{split}
  (\gamma({\bm x}^*) & -
  \alpha_1({\bm x}^*) - \beta_1 ({\bm x}^*) -
  \alpha_2({\bm x}^*) - \beta_2 ({\bm x}^*)) y({\bm x}^*) \\
  & + \alpha_1({\bm x}^*) (y({\bm x}^*) - y(x_1^*-h_1,x_2^*)) \\
  & + \beta_1 ({\bm x}^*) (y({\bm x}^*) - y(x_1^*+h_1,x_2^*)) \\
  &+ \alpha_2({\bm x}^*)  (y({\bm x}^*) - y(x_1^*,x_2^*-h_2) ) \\
  & + \beta_2 ({\bm x}^*) (y({\bm x}^*) - y(x_1^*,x_2^*+h_2)) > 0 .
\end{split}
\]
We arrive at a contradiction, and therefore
$y ({\bm x}) \ge 0$ for all  ${\bm x} \in \omega$.

Now we consider a more complicated case of the difference scheme
(\ref{83}), (\ref{84}) with the conditions ((\ref{85}), (\ref{87}) satisfied. 
Suppose that for the non-negative right-hand side of equation
(\ref{83}),  there exists a subset of the interior grid points $\omega^*$, where
\[
  y({\bm x}) < 0,
  \quad {\bm x} \in \omega^*.
\]
Summarize equations (\ref{83}) over all these nodes:
\[
\begin{split}
  \sum_{{\bm x} \in \omega^*}
  \Big (\gamma({\bm x}) y({\bm x})
  &   - \alpha_1({\bm x}) y(x_1-h_1,x_2) -
  \beta_1 ({\bm x}) y(x_1+h_1,x_2) \\
  & - \alpha_2({\bm x})  y(x_1,x_2-h_2) -
  \beta_2 ({\bm x}) y(x_1,x_2+h_2) \Big ) \\
  &=
  \sum_{{\bm x} \in \omega^*}
  \varphi({\bm x}) \ge 0.
\end{split}
\]
For the left-hand side of this equality, we have
\[
\begin{split}
  \sum_{{\bm x} \in \omega^*}
  \Big (\gamma ({\bm x}) & -
  \alpha_1 (x_1+h_1,x_2) - \beta_1 (x_1-h_1,x_2) \\
  & - 
  \alpha_2 (x_1,x_2+h_2) - \beta_2 (x_1,x_2-h_2) \Big )
  y ({\bm x}) \\
  & +
  \sum_{{\bm x} \in \omega^*}
  (\alpha_1 (x_1+h_1,x_2) y({\bm x}) -
  \alpha_1({\bm x}) y(x_1-h_1,x_2)) \\
  & +
  \sum_{{\bm x} \in \omega^*}
  (\beta_1 (x_1-h_1,x_2) y({\bm x}) -
  \beta_1({\bm x}) y(x_1+h_1,x_2)) \\
  & +
  \sum_{{\bm x} \in \omega^*}
  (\alpha_2 (x_1,x_2+h_2) y({\bm x}) -
  \alpha_2({\bm x}) y(x_1,x_2-h_2)) \\
  & +
  \sum_{{\bm x} \in \omega^*}
  (\beta_2 (x_1,x_2-h_2) y({\bm x}) -
  \beta_2({\bm x}) y(x_1,x_2+h_2)) .
\end{split}  
\]

We see immediately that each of these terms
is non-negative and they cannot be equal to zero simultaneously. 
Thus, we again obtain a contradiction.
\qed

The maximum principle for multidimensional difference equations, where the
sufficient conditions of type (\ref{86}) are satisfied, is well-known in the theory
of difference schemes \cite{2001Samarskii}. For elliptic
difference problems, the maximum principle in the form (\ref{87}) is presented in
work \cite{vabmon}. Similarly to the 1D case, the conditions (\ref{86})
may be associated with the condition of diagonal dominance by rows 
if we use the natural order for the points of the numerical solution, whereas
the conditions (\ref{87}) correspond to diagonal dominance by columns.

Consider the 2D boundary value problem for the convection-diffusion equation
in nondivergent form:
\begin{equation}\label{88}
   \sum_{\alpha =1}^{2}
   v_\alpha ({\bm x}) \frac{\partial u}{\partial x_\alpha}
   - \sum_{\alpha =1}^{2}
   \frac{\partial } {\partial x_\alpha} \left ( k({\bm x}) \frac{\partial u}{\partial x_\alpha} \right ) = f({\bm x}),
   \quad {\bm x} \in \Omega,
\end{equation} 
\begin{equation}\label{89}
   u({\bm x}) = 0,
   \quad {\bm x} \in \partial\Omega .
\end{equation}
For numerical solving the problem (\ref{88}), (\ref{89}), we employ the scheme
\begin{equation}\label{90}
  C_1 y + D y = \varphi ({\bm x}),
  \quad {\bm x} \in \omega .
\end{equation} 
We restrict ourselves (see \cite{SamVabConv} for details) to the 2D operator
of convective transport in the form
\[
  C_1^{(1)} y = \frac {v_1({\bm x})}{2k({\bm x})}
  (k(x_1-0.5h_1,x_2) y^{\bar {x}_1} +
  k(x_1+0.5h_1,x_2) y^{x_1}),
\]
\[
  C_1^{(2)} y = \frac {v_2({\bm x})}{2k({\bm x})}
  (k(x_1,x_2-0.5h_2) y^{\bar {x}_2} +
  k(x_1,x_2+0.5h_2) y^{x_2}),
\]
\begin{equation}\label{91}
  C_1  = \sum_{\alpha =1}^{2} C_1^{(\alpha)} ,
 \quad  {\bm x} \in \omega .
\end{equation} 

We formulate the conditions of monotonicity for the difference scheme (\ref{90}), (\ref{91}). 
We write it in the form (\ref{83}), (\ref{84}) with
\[
  \alpha_1 ({\bm x}) = \left (\frac {v_1({\bm x})}{2h_1k({\bm x})} +
  \frac {1}{h_1^2} \right ) k(x_1-0.5h_1,x_2),
\]
\[
  \beta_1 ({\bm x}) = \left (- \frac {v_1({\bm x})}{2h_1k({\bm x})} +
  \frac {1}{h_1^2} \right ) k(x_1+0.5h_1,x_2),
\]
\[
  \alpha_2 ({\bm x}) = \left (\frac {v_2({\bm x})}{2h_2k({\bm x})} +
  \frac {1}{h_2^2} \right ) k(x_1,x_2-0.5h_2),
\]
\[
  \beta_2 ({\bm x}) = \left (- \frac {v_2({\bm x})}{2h_2k({\bm x})} +
  \frac {1}{h_2^2} \right ) k(x_1,x_2+0.5h_2),
\]
\[
  \gamma ({\bm x}) =
  \alpha_1 ({\bm x}) + \alpha_2 ({\bm x}) +
  \beta_1 ({\bm x}) + \beta_2 ({\bm x}),
  \quad {\bm x} \in \omega .
\]
The monotonicity condition (\ref{87}) is obviously satisfied, and
the positiveness of the coefficients of the difference scheme (\ref{83}) (the condition
(\ref{85})) leads to the natural restrictions
\[
  |\theta_\alpha ({\bm x})| \le 1,
  \quad \alpha =1,2,
  \quad {\bm x} \in \omega ,
\]
where
\[
   \theta_\alpha ({\bm x}) =
   \frac {v_\alpha ({\bm x}) h_\alpha}{2 k({\bm x})},
  \quad \alpha =1,2.
\]

Unconditionally monotone difference schemes for the 2D 
convection-diffusion equation (\ref{88}), (\ref{89}) are constructed by means of
regularization (disturbance) of the diffusion coefficient.
Instead of (\ref{90}), we consider the difference scheme
\begin{equation}\label{92}
  C_1 y + \sum_{\alpha=1}^{2} (1+ \varrho_\alpha) D^{(\alpha)} y =
  \varphi ({\bm x}),
  \quad {\bm x} \in \omega .
\end{equation}
The scheme (\ref{92}) is monotone under the condition
\[
  1 + \varrho_\alpha ({\bm x}) > |\theta_\alpha ({\bm x})|,
  \quad {\bm x} \in \omega ,
  \quad \alpha =1,2.
\]
We present some variants of regularizing grid functions
$\varrho_\alpha({\bm x}),~ \alpha =1,2$,, which lead to
unconditionally monotone difference schemes (\ref{92}).

For example, the choice
\[
  1 + \varrho_\alpha ({\bm x}) =
  \theta_\alpha ({\bm x}) {\rm cth} \theta_\alpha ({\bm x}),
  \quad {\bm x} \in \omega ,
  \quad \alpha =1,2
\]
corresponds to the use of  exponential schemes \cite{allsou,schgum} for each individual direction.
It is also possible to select the scheme with
\[
  \varrho_\alpha ({\bm x}) =
  \eta \theta_\alpha^2 ({\bm x}),
  \quad {\bm x} \in \omega ,
  \quad \alpha =1,2,
\]
which are monotone if $\eta > 0.25$.
Among unconditionally monotone difference schemes, we highlight
the regularized scheme (\ref{92}), where
\[
  \varrho_\alpha ({\bm x}) =
  \frac {\theta_\alpha^2 ({\bm x})}
  {1 + |\theta_\alpha ({\bm x})|},
  \quad {\bm x} \in \omega ,
  \quad \alpha =1,2 .
\]

The scheme with the upwind differences corresponds to
\[
  \varrho_\alpha ({\bm x}) =
  \eta |\theta_\alpha ({\bm x})|,
  \quad {\bm x} \in \omega ,
  \quad \alpha =1,2
\]
with $\eta = 1$. In this case, the difference operator of convective
transport seems like this:
\[
  C_1^{(1)} y = \frac {k(x_1-0.5h_1,x_2)}{k({\bm x})}
  v_1^+ ({\bm x}) y^{\bar {x}_1} +
  \frac {k(x_1+0.5h_1,x_2)}{k({\bm x})}
  v_1^- ({\bm x}) y^{x_1},
\]
\[
  C_2^{(1)} y = \frac {k(x_1,x_2-0.5h_2)}{k({\bm x})}
  v_2^+ ({\bm x}) y^{\bar {x}_2} +
  \frac {k(x_1,x_2+0.5h_2)}{k({\bm x})}
  v_2^- ({\bm x}) y^{x_2},
\]
\begin{equation}\label{93}
  C_1  = \sum_{\alpha =1}^{2} C_1^{(\alpha)} ,
 \quad  {\bm x} \in \omega .
\end{equation}
Obviously, the corresponding scheme has only the first-order approximation.

Now consider the convection-diffusion equation in the divergent form:
\begin{equation}\label{94}
   \sum_{\alpha =1}^{2}
   \frac{\partial } {\partial x_\alpha}
   (v_\alpha ({\bm x}) u)
   - \sum_{\alpha =1}^{2}
   \frac{\partial } {\partial x_\alpha} \left ( k({\bm x}) \frac{\partial u} {\partial x_\alpha} \right ) = f({\bm x}),
   \quad {\bm x} \in \Omega,
\end{equation}
which is supplemented with the boundary conditions (\ref{89}).
For this equation, we consider the difference scheme
\begin{equation}\label{95}
  C_2 y + D y = \varphi ({\bm x}),
  \quad {\bm x} \in \omega .
\end{equation}
Define the difference operator of the convective transport as
\[
\begin{split}
  C_2^{(1)} y & = \frac {1}{2h_1}
  v_1(x_1+0.5h_1,x_2) (y(x_1+h_1,x_2)+ y({\bm x})) \\
  &- \frac {1}{2h_1}
  v_1(x_1-0.5h_1,x_2) (y(x_1-h_1,x_2)+ y({\bm x})) ,
\end{split}
\]
\[
\begin{split}
  C_2^{(2)} y & = \frac {1}{2h_2}
  v_2(x_1,x_2+0.5h_2) (y(x_1,x_2+h_2)+ y({\bm x})) \\
  & - \frac {1}{2h_1}
  v_2(x_1,x_2-0.5h_2) (y(x_1,x_2-h_2)+ y({\bm x})) ,
\end{split}  
\]
\begin{equation}\label{96}
  C_2  = \sum_{\alpha =1}^{2} C_2^{(\alpha)} ,
 \quad  {\bm x} \in \omega .
\end{equation}
This is typical for coefficients of convective transport defined
on staggered grids.

The difference scheme (\ref{95}), (\ref{96}) may be written in the form
(\ref{83}) with the coefficients
\[
  \alpha_1 ({\bm x}) = \frac {v_1(x_1-0.5h_1,x_2)}{2h_1} +
  \frac {k(x_1-0.5h_1,x_2)}{h_1^2},
\]
\[
  \beta_1 ({\bm x}) = - \frac {v_1(x_1+0.5h_1,x_2)}{2h_1} +
  \frac {k(x_1+0.5h_1,x_2)}{h_1^2},
\]
\[
  \alpha_2 ({\bm x}) = \frac {v_2(x_1,x_2-0.5h_2)}{2h_2} +
  \frac {k(x_1,x_2-0.5h_2)}{h_2^2},
\]
\[
  \beta_2 ({\bm x}) = - \frac {v_2(x_1,x_2+0.5h_2)}{2h_2} +
  \frac {k(x_1,x_2+0.5h_2)}{h_2^2},
\]
\[
\begin{split}
  \gamma ({\bm x}) & =
  \alpha_1 (x_1+h_1,x_2) + \beta_1 (x_1-h_1,x_2) \\
  & +
  \alpha_2 (x_1,x_2+h_2) + \beta_2 (x_1,x_2-h_2),
  \quad {\bm x} \in \omega.
\end{split}    
\]
Therefore, the condition (\ref{87}) is valid, and from (\ref{85}), we get
\[
  |\theta_1 (x_1-0.5h_1,x_2)| \le 1,
  \quad {\bm x} \in \omega_1^+\times \omega_2,
\]
\[
  |\theta_2 (x_1,x_2-0.5h_2)| \le 1,
  \quad {\bm x} \in \omega_1\times \omega_2^+ .
\]

A class of regularized difference schemes for
the convection-diffusion equation in the divergent form is defined as
\begin{equation}\label{97}
\begin{split}
  C_2 y & - ((1 + \varrho_1(x_1-0.5h_1,x_2)) k(x_1-0.5h_1,x_2)
  y^{\bar{x}_1} )_{x_1} \\
  &-
  ((1 + \varrho_2(x_1,x_2-0.5h_2)) k(x_1,x_2-0.5h_2)
  y^{\bar{x}_2} )_{x_2} \\
  & = \varphi({\bm x}),
  \quad {\bm x} \in \omega.
\end{split}    
\end{equation}
Sufficient conditions for the monotonicity of the scheme (\ref{97})
are written as
\[
  1 + \varrho_1(x_1-0.5h_1,x_2) > |\theta_1 (x_1-0.5h_1,x_2)|,
  \quad {\bm x} \in \omega_1^+\times \omega_2,
\]
\[
  1 + \varrho_2(x_1,x_2-0.5h_2) > |\theta_2 (x_1,x_2-0.5h_2)|,
  \quad {\bm x} \in \omega_1\times \omega_2^+ .
\]

Some approaches to select regularizing grid functions
$\varrho_\alpha ({\bm x}),~ \alpha =1,2$ in order to obtain unconditionally
monotone difference schemes were considered above for the
convection-diffusion equations in the nondivergent form.

\subsection{Triangular grids}

The possibilities of solving boundary value problems for PDEs
on irregular grids are discussed here. The focus is on constructing
difference schemes on triangular grids (as the most common
unstructured grids). We emphasize approximations on Delaunay grids
(triangulations) that demonstrate optimal properties.
The problem of discretization in space is illustrated considering steady-state problems.

The basis for the construction of discrete analogs is the balance method
(the integro-interpolation method) \cite{balansametod,2001Samarskii}, which nowadays
(in the English literature) is referred to as
the finite volume method \cite{finitevolumebook,FVM_China}.
The efficiency of this approach becomes evident
in designing difference schemes on irregular grids.
As a control volume in Delaunay triangulations appears Voronoi diagrams.

Among general irregular grids we distinguish structured grids that
are topologically equivalent to regular grids. A typical example
of unstructured meshes are triangular grids. There are discussed
general issues of designing grids and discretization on them.

Numerical solving boundary value problems of
mathematical physics in complicated domains is carried out
using irregular grids. A computational domain
$\Omega$ is assumed to be irregular (nonrectangular and not
composed of rectangles). Because of this, we have to
use nonrectangular grids. Among irregular grids,
we emphasize two main classes.

\begin{description}
\item[Structured grids.] 
The most important example of such a type of grids is irregular quadrangular grids 
that inherit, in many senses, properties of standard rectangular
grids (they are topologically equivalent to them).

\item[Unstructured grids.]
In this case, a stencil of a difference scheme has no fixed
structure. It is impossible to connect topologically such a computational grid to a
regular rectangular grid. In particular, schemes have a different number of neighbors
at each grid point.
\end{description}

Approximations on structured grids can be performed on the basis of the above-mentioned 
closeness of these grids to standard rectangular grids. The simplest realization of
this situation is to use new independent variables. In this case, a grid that is irregular in the original
coordinates is transformed into a regular one in new independent coordinates.

The second possibility is not associated with a formal introduction of new
coordinates; it is implemented using an approximation of the original problem on
an irregular grid. It is clear that the use of simple
approaches for the construction of difference schemes on the basis of uncertain
coefficients for irregular grids is possible, but it seems not so reasonable.

Advantages of structured grids results from the conservation of the canonical structure
of neighbors for each grid node, i.e., the conservation of the stencil. This simplifies, in
particular, the process of programming and solving difference problems. But the problems of
constructing difference schemes on such grids are not much less difficult than for the general
unstructured grids.

Among structured grids, it is necessary to distinguish an important class of orthogonal
grids. In this case, the advantages of structured grids over unstructured ones become
evident because a lot of problems connected with the development of difference schemes,
and the solution of grid equations is radically simplified. If it is necessary to use the
advantages of structured irregular grids over unstructured ones, it is better to restrict ourself
to orthogonal curvilinear grids. Problems of grid generation are not necessarily more
difficult than the problems being solved. Moreover, this situation is the most typical one.
Therefore, it is better to make efforts (that are comparable to the solution of the original problem) to
optimize the computational grid. In complicated computational domains, it is reasonable to
use the multiblock technology of generation of orthogonal grids that is based on modern CAD systems.

An arbitrary grid is generated from a set of nodes.
The most simple and natural approach is to define a triangulation, i.e., to construct a triangular
grid. There is no need to use more complicated structures of unstructured grids.

For the given points, a triangulation can be performed in different ways. Note also
that for a given set of nodes, we obtain the same number of triangles by any triangulation
method. 

Thus, we need to optimize the triangulation by some criteria. The main optimization
criterion consists in the following: the obtained triangles should be close to equilateral ones
(they should be without too sharp angles). This is a local criterion governing to an individual triangle.
The second (global) criterion declares that adjacent triangles must not differ too widely in
an area --- the criterion of grid uniformity.

There is a special triangulation --- the Delaunay triangulation \cite{george1998delaunay,VoronoiDiagr}, 
which has a number of optimal properties. One of them is the tendency of obtained triangles to be equiangular
ones. The above-mentioned property can be formulated more exactly in the following way:
in the Delaunay triangulation, the minimal value of inner angles of triangles is maximized.
The formal definition of the Delaunay triangulation 
is associated with the property that for each triangle all the other nodes
are located outside of the circumcircle. For our further presentation, the relation between the
Delaunay triangulation and the Voronoi diagram (\textup{tesselation}) is very important.

A Voronoi polygon for a separate node is a set of points lying closer to this node than
to all the other nodes. For two points, the sets are defined by the half-plane bounded by
a perpendicular to the middle of the segment connecting these two points. The Voronoy
polygon thereby will be the intersection of such half-planes for all pairs of nodes created by
this node and all the other nodes. Note that this polygon is always convex. 

Each vertex of a Voronoi polygon is a point of contact of three Voronoi polygons. The
triangle constructed by the corresponding nodes of contacting Voronoi polygons is associated
with each of these vertices. This is exactly the Delaunay triangulation. Thus, between the
Voronoi diagram and the Delaunay triangulation a unique correspondence is established.

In the case of the Delaunay triangulation, we obtain the optimal decomposition of a
computational domain according to the given set of nodes. The decomposition is optimal
in terms of maximization of minimal angles of triangles. For the Delaunay triangulation,
there does exist the corresponding Voronoi diagram that uniquely determines a set of points of the
domain for each node. This separation of the set points is made by the clear geometrical
criterion of optimal closeness to the node. Thus, the Delaunay triangulation and the Voronoi
diagram determine completely (optimally and uniquely) a computational triangular grid
and a control volume. 

The Delaunay triangulation is widely used in numerical practice
for constructing finite element procedures. There also exist a lot of
developed numerical methods for generating such triangular grids, the appropriate software
is also available.

\subsection{Difference schemes on triangular grids}

We start with a discussion of some possible
general approaches that may be applied to (and find) practical applications.

The simplest (from the methodological viewpoint)
approach to construct discrete analogs on triangular
grids consists in using the finite element procedures \cite{Strang:1973:AFE}. 
However, it is not always possible to employ standard variants
of finite elements.

In computational practice, there are widely used piecewise linear
finite elements, which correspond to the approximation
of the numerical solution on each triangle via linear functions.
In the convection-diffusion problems, we obtain analogues of schemes
with central difference approximation of the convective terms.

In constructing finite element approximations, there do exist some problems to obtain monotone
procedures, i.e., the schemes that satisfy the maximum principle. In
the theory of finite elements, the problem is resolved
not using the Galerkin method, but using its generalization --- the projection 
Petrov-Galerkin method. In this case, probe functions that are used to construct the solution, and
test functions that generate a system of equations, are distinct. In
this sufficient artificial way, we obtain schemes, which are,
e.g., very similar to the schemes with the upwind differences \cite{MortonKellogg1996}.

In the method of support operators \cite{samkol}, the original problem is formulated in terms
of differential operators from the vector analysis: the divergence,
gradient, and rotor. Next, only one of them is freely approximated on a selected grid. 
Other operators are defined by some prescribed relations of integral nature that exist between the differential
operators. This ensures consistent approximations of the operators that provide the fulfillment of
such essential properties as conservatism, adjointness and so on.

The method of support operators has been developed by many researchers just for triangular grids.
The main peculiarities are associated with a selection of the set of grid functions. 
For example, the solution can be approximated at vortices of the triangular grid, whereas
fluxes can be refereed to cells (the cell-vortex arrangement of variables) or to the the midpoints of cell faces
(staggered grids).

One of the basic approaches to the construction of difference schemes on irregular
grids is the classical integro-interpolation method \cite{2001Samarskii}.
This balance method is based on the following main points. First,
we must specify a grid (determine a set of nodes and a set of grid functions). 
Secondly, for each node, we define a neighbor domain (control volume) --- a part
of the computational domain adjoining to a given computational node. And finally,
a difference scheme is obtained by integrating the original equation over the control
volume using some assumptions about the solution behavior.
A set of these three components specify a particular variant of the control volume method.

In constructing difference schemes on triangular grid via the control volume method,
it seems natural to set grid functions at the grid nodes. 
This is the standard, but not the only variant. As an alternative,
we can reffer grid functions to some points being connected
with a triangular cell. The specification of grid functions
at the vertices of the Voronoi polygons provide an example.

The second problem is connected with selecting a control volume. During
a triangulation procedure, in many approaches, a part of the
triangle appearing from the intersection of medians is separated as as a control volume.
In this case, each node obtains a part of the triangle with an equal area.

An interesting variant of the control volume selection is associated with the
Voronoi diagrams. In this case, each individual node has a part of the whole computational
domain that is closest to it. In this case, there is no division into triangles with equal areas.
It is important that in the both above approaches, a criterion for selecting
a control volume is clear and is connected with geometric requirements: in the first case ---
the triangle is decomposed into three parts of equal area, in the second ---
we obtain geometric proximity of points of the computational domain.

Among the merits of the division by medians, we emphasize that it may be conducted
for an arbitrary partitioning into triangles, i.e., not only for the Delaunay triangulation.
Advantages of the Voronoi tesselation are more essential and
associated with the orthogonality of the triangle sides to the faces of Voronoi polygons.

Heuristic arguments in favor of Voronoi polygons are associated with the idea of globalization
(optimization) of grids and control volumes~--- optimization for all
nodes, rather than for a single triangle.

The balance method for triangular grids need to be implemented using the Delaunay triangulation
with Voronoi polygons as control volumes. This is the most natural way that allows
to construct difference schemes with optimized triangular grids \cite{vago}.

\subsection{Diffusive transport operator}

Assume that a computational domain is a convex polygon $\Omega$ with the boundary $\partial\Omega$.
The points of the domain are denoted by $\bm{x} = (x^{(1)},x^{(2)})$.

In the domain $\overline{\Omega} = \Omega \bigcup \partial\Omega$,
we consider the grid $\overline{\omega}$, which consists of nodes 
$\bm{x}_{i}, \ i = 1,2, \, \ldots, \, M$,
and the angles of the polyhedron $\Omega$ are nodes. Let $\omega$ be a set of interior nodes 
and $\partial \omega$ is a set of boundary nodes, i.e., 
$\omega=\overline{\omega} \bigcap \Omega$, 
$\partial \omega = \overline{\omega} \bigcap \partial \Omega$.

Each node $\bm{x}_i, \ i = 1,2, \, \ldots, \, M,$ is associated with a certain part of 
the computational domain $\Omega_i$ treated as a control volume.
A Voronoi polygon or its part belonging to $\Omega$ are selected as the control volumes. 
A Voronoi polygon for an individual node is a set of points lying 
closer to this node than to all the other ones.
For two nodes, the sets are defined by the half-plane bounded by
the perpendicular to the midpoint of the segment connecting these two nodes. The Voronoi
polygons thereby will be the intersection of such half-planes for all pairs of nodes created by
this node and all the other nodes. 
Each vertex of a Voronoi polygon is a point of contact of three Voronoi polygons. The
triangle constructed by the corresponding nodes of contacting Voronoi polygons is associated
with each of these vertices. This is exactly the Delaunay triangulation. 

Control volumes cover the whole computational domain, so that
\[
  \overline{\Omega} = \bigcup_{i=1}^{M} \overline{\Omega}_{i},
  \quad \overline{\Omega}_{i} = \Omega_{i} \bigcup \partial \Omega_{i},
  \quad \Omega_{i} \bigcap \Omega_{j} = \emptyset,
\]
\[
  \quad i \neq j,
  \quad i, \ j = 1,2, \, \ldots, \, M.
\]
For the common faces of control volumes, we use notation
\[
 \partial \Omega_{i} \bigcap \partial \Omega_{j} = \Gamma_{ij},
  \quad i \neq j,
  \quad i, \ j = 1,2, \, \ldots, \, M.
\]
For the node $i$, we define a set of neighboring nodes $\mathcal{W}(i)$
that have the control volumes with common faces with the control volume for the node $i$, i.e.,
\[
  \mathcal{W}(i) = \{ j \ | \ \partial \Omega_{i} \bigcap \partial \Omega_{j} \neq \emptyset,
  \ j =1,2, \, \ldots, \, M\},
  \quad i = 1,2, \, \ldots, \, M.
\]
Introduce notation
\[
  V_{i} = \int\limits_{\Omega_{i}} d{\bm x},
  \quad l_{ij} = \int\limits_{\Gamma_{ij}} d{\bm x},
  \quad i, \ j =1,2, \, \ldots, \, M ,
\]
for the area of the control volume and the length of the edge of the  Voronoi polyhedron, respectively.

For the grid functions $y({\bm x}), \ w({\bm x})$ that are specified at the nodes
${\bm x} \in \bar \omega$ and vanish at the boundary nodes
${\bm x} \in \partial \omega$, in $H=L_2(\omega)$, we define the scalar product
and norm
\[
  (y,w) =  \sum_{\bm{x}_i \in \omega}
  V_i y({\bm x}_i) w({\bm x}_i),
  \quad \|y\| = (y,y)^{1/2} .
\]

Define the distance between the nodes $\bm{x}_i$, $\bm{x}_j$ as
\[
  d(\bm{x}_i,\bm{x}_j) =
  \left [  \sum_{\alpha=1}^{2} (x_i^{(\alpha)} - x_j^{(\alpha
  )})\right]^{1/2} ,
\]
and the midpoint of the segment connecting these nodes as follows:
\[
  \bm{x}_{ij} = (x_{ij}^{(1)},x_{ij}^{(2)}),
  \quad x_{ij}^{(\alpha)} = \frac{1}{2} (x_i^{(\alpha)} + x_j^{(\alpha)}),
  \quad \alpha = 1,2.
\]

For simplicity, we assume that the coefficients of the convection-diffusion equation and
its solution itself are sufficiently smooth.
The discrete operator of diffusive transport corresponding to
the interior node of the computational grid $\bm{x}_i \in \omega$ is defined
according to the integro-interpolation method by means of the integration
over the control volume $\Omega_i$:
\begin{equation}\label{98}
  Du = (Du)_{i} \approx \frac{1}{V_{i}} \int\limits_{\Omega_i}
  \mathcal{D}u d \bm{x}.
\end{equation}

For the diffusive flux vector, we use the expression
\[
  \bm{q} = - k(\bm{x}) \grad u,
\]
so that
\[
  \mathcal{D} u = \div \bm{q} .
\]
From this, we obtain for the right-hand side of (\ref{98}) that
\begin{equation}\label{99}
  \int\limits_{\Omega_i} \mathcal{D}u dx =
  \int\limits_{\partial \Omega_i} (\bm{q},\bm{n}) d \bm{x},
\end{equation}
where $\bm{n}$ is the outer normal.

A difference approximation for the normal component of the diffusive flow through the face
$\gamma_{ij}$ is denoted by $\bm{q}^h_{ij}$, and therefore, from
(\ref{98}), (\ref{99}), we get for the difference diffusive transport operator
the following representation:
\begin{equation}\label{100}
  (D y)_{i} = \frac{1}{V_{i}}
  \sum_{j \in \mathcal{W}(i)} l_{ij} q^h_{ij},
  \quad \bm{x}_i \in \omega .
\end{equation}

In the case of smooth enough coefficients and the solution itself,
it is natural to employ elementary approximations for the flux along
the normal at the point $\bm{x}_{ij}$:
\begin{equation}\label{101}
  q^h_{ij} =  - k(\bm{x}_{ij})
  \frac{y^{j}- y^{i}}{d(\bm{x}_i,\bm{x}_j)} .
\end{equation}

From (\ref{100}), (\ref{101}), we obtain the difference operator
of the diffusive transport:
\begin{equation}\label{102}
  (D y)_{i} = - \frac{1}{V_{i}}
  \sum_{j \in \mathcal{W}(i)} l_{ij}
  k(\bm{x}_{ij})
  \frac{y^{j}- y^{i}}{d(\bm{x}_i,\bm{x}_j)},
  \quad \bm{x}_i \in \omega .
\end{equation}

As in the case of difference schemes on regular grids,
it is possible to introduce various approximations for the diffusive flows. This
issue is very important, in particular, for problems with piecewise-smooth
coefficients of the equation.

For the grid functions $y_i \equiv y(\bm{x}_i) = 0, \ w_i=0$,
$\bm{x}_i \in \partial \omega$, we have
\[
  (D y, w) =
  \sum_{\bm{x}_i \in \omega}
  \sum_{j \in \mathcal{W}(i)} l_{ij} q^h_{ij} w_i =
  - \sum_{\bm{x}_i \in \omega}
  \sum_{j \in \mathcal{W}(i)} l_{ij}
  k(\bm{x}_{ij})
  \frac{y^{j}- y^{i}}{d(\bm{x}_i,\bm{x}_j)} w_i .
\]
This summation over all faces of the Voronoi polygons for
all interior nodes may be rewritten in the more convenient form:
\[
\begin{split}
  (D y, w) & =
  \frac{1}{2}  \sum_{\bm{x}_i \in \omega}
  \sum_{j \in \mathcal{W}(i)}
  \frac {l_{ij}} {d(\bm{x}_i,\bm{x}_j)}
  k(\bm{x}_{ij})
  ((y_j-y_i)w_j - (y^{j}- y^{i}) w_i) \\
  & =
  \frac{1}{2}  \sum_{\bm{x}_i \in \omega}
  \sum_{j \in \mathcal{W}(i)}
  l_{ij} d(\bm{x}_i,\bm{x}_j) k(\bm{x}_{ij})
  \frac{y_j-y_i}{d(\bm{x}_i,\bm{x}_j)}
  \frac{z_j-z_i}{d(\bm{x}_i,\bm{x}_j)} .
\end{split}  
\]
Thus, $(Dy, w) = (y,D^* w)$, i.e, the difference operator
(\ref{102}) is self-adjoint in $H$. In view of
\begin{equation}\label{103}
  (D y, y) =
  \frac{1}{2}  \sum_{\bm{x}_i \in \omega}
  \sum_{j \in \mathcal{W}(i)}
  l_{ij} d(\bm{x}_i,\bm{x}_j) k(\bm{x}_{ij})
  \left (\frac{y_j-y_i}{d(\bm{x}_i,\bm{x}_j)} \right)^2 ,
\end{equation}
it is also  positive ($D = D^* > 0$).

Now we establish a discrete analog of the Friedrichs lemma. In various
formulations, it is proved in the theory of finite elements
for considering difference schemes on irregular structured and unstructured grids. Difference
schemes that are based on the Delaunay triangulation and the Voronoi diagram demonstrate
some peculiarities, and therefore the discrete analog of the Friedrichs lemma
must be proved for them separately.

\begin{lemma}\label{l-4}
For the grid functions
$y_i \equiv y(\bm{x}_i) = 0$, $\bm{x}_i \in \partial \omega$, 
the following inequality is valid:
\begin{equation}\label{104}
  \|y\|^{2} \le \frac{M_0}{2}
  \sum_{\bm{x}_i \in \omega}
  \sum_{j \in \mathcal{W}(i)}
  l_{ij} d(\bm{x}_i,\bm{x}_j)
  \left (\frac{y_j-y_i}{d(\bm{x}_i,\bm{x}_j)} \right)^2
\end{equation}
with the constant
\[
  M_0 = \frac{l_1^2}{16} + \frac{l_2^2}{16},
\]
where $l_{\alpha}, \ \alpha =1,2$ are the side lengths of the rectangle with
sides parallel to the coordinate axes that contains completely the whole polygon $\Omega$.
\end{lemma}
\proof
On the set of the grid functions
$y_i \equiv y(\bm{x}_i) = 0$, $\bm{x}_i \in \partial \omega$,
in view of (\ref{102}), we define the discrete Laplace operator $\Lambda$ via the expression
\begin{equation}\label{105}
  (\Lambda  y)_{i} = - \frac{1}{V_{i}}
  \sum_{j \in \mathcal{W}(i)} l_{ij}
  \frac{y^{j}- y^{i}}{d(\bm{x}_i,\bm{x}_j)},
  \quad \bm{x}_i \in \omega .
\end{equation}
Using this notation, the inequality (\ref{104}) may be rewritten
in the more compact form:
\begin{equation}\label{106}
  \|y\|^{2} \le M_0
  (\Lambda y,y) .
\end{equation}

To estimate the lower bound of the discrete Laplace operator, we employ
the solution of an auxiliary boundary value problem. We will show
that in the inequality (\ref{106}) under consideration, we can put
\begin{equation}\label{107}
  M_0 = \max_{{\bm x} \in \omega} w({\bm x}),
\end{equation}
where $w({\bm x})$ is the solution of the problem
\begin{equation}\label{108}
  \Lambda w = 1,
  \quad {\bm x} \in \omega .
\end{equation}

We consider in $H$  the eigenvalue problem
\begin{equation}\label{109}
  \Lambda y = \lambda y,
  \quad {\bm x} \in \omega
\end{equation}
for the grid operator $\Lambda = \Lambda^* > 0.$ For the problem (\ref{109}),
the following inequality holds:
\begin{equation}\label{110}
  (\Lambda y,y) \ge \lambda_{\min} \|y\|^2
\end{equation}
for any $y({\bm x})$. The equality in (\ref{110}) is achieved only for the eigenfunctions $v({\bm x})$
that correspond to the minimal eigenvalue $\lambda_{\min}.$ Thus,
the estimate (\ref{106}) will be established if we will show that
$M_0^{-1} \le \lambda_{\min}$.

First of all, let us explain that $v({\bm x})$ is a constant-sign
function. Suppose that this is not true and $v({\bm x})$ changes its sign
on the grid $\omega$. Now consider the function $|v({\bm x})|$. Then taking into account
\[
  (\Lambda y, y) =
  \frac{1}{2}  \sum_{\bm{x}_i \in \omega}
  \sum_{j \in \mathcal{W}(i)}
  l_{ij} d(\bm{x}_i,\bm{x}_j)
  \left (\frac{y_j-y_i}{d(\bm{x}_i,\bm{x}_j)} \right)^2 ,
\]
we obtain
\[
  \frac{(\Lambda |v({\bm x})|, |v({\bm x})|)}{(|v({\bm x})|, |v({\bm x})|)}
  < \frac{(\Lambda v({\bm x}), v({\bm x}))}{(v({\bm x}), v({\bm x}))} .
\]
This contradicts the fact that the minimal Rayleigh ratio
\[
(\Lambda y({\bm x}), y({\bm x})) \ (y({\bm x}), y({\bm x}))^{-1}
\]
is achieved for $y({\bm x}) = v({\bm x}).$

From $\Lambda v = \lambda_{\min} v$, we have
\begin{equation}\label{111}
  \lambda_{\min} =
  \frac{(\Lambda v({\bm x}), 1)}{(v({\bm x}), 1)}.
\end{equation}
In view of (\ref{108}), we obtain for the denominator
\[
  (v({\bm x}), 1) = (v({\bm x}), \Lambda w ({\bm x})) =
  (w({\bm x}), \Lambda v({\bm x})) \le
  \max_{{\bm x} \in \omega} w({\bm x}) \ (\Lambda v ({\bm x}), 1),
\]
since $\Lambda v ({\bm x}) \ge 0, \ {\bm x} \in \omega$.
Thus, from (\ref{111}), it follows immediately that $M_0^{-1} \le \lambda_{\min}$, and
for the constant $M_0$, we can use the expression (\ref{107}).

Next, we apply the maximum principle to discrete elliptic equations. 
We place the polygon $\Omega$ into the rectangle
\[
  \Omega_0 = \{ {\bm x} = (x^{(1)}, x^{(2)}) \ | \
  a_{\alpha} \le  x^{(\alpha)} \le b_{\alpha}, \
  \alpha =1,2 \},
\]
where $l_{\alpha} = b_{\alpha} - a_{\alpha}, \ \alpha=1,2.$
Consider the function
\[
  W({\bm x}) =  - \mu (
  (x^{(1)} - a_1) (x^{(1)} - b_1) +
  (x^{(2)} - a_2) (x^{(2)} - b_2))
\]
with some positive constant $\mu$. We write this majorant function as
\[
  W({\bm x}) = - \mu (
  (x^{(1)} - x^{(1)}_i)^2 +
  (x^{(2)} - x^{(2)}_i)^2) +
  g_i({\bm x}),
\]
where $g_i({\bm x})$ is a linear function.

For linear functions, we have
\[
  \Lambda g({\bm x}) = 0 .
\]
To show this, it is sufficient to consider the function $g({\bm x})=x^{(1)}$. 
By virtu of (\ref{105}), we obtain
\[
  \Lambda x^{(1)} =
  - \frac{1}{V_{i}}
  \sum_{j \in \mathcal{W}(i)} l_{ij}
  \frac{x^{(1)}_{j}- x^{(1)}_{i}}{d(\bm{x}_i,\bm{x}_j)} .
\]
We have
\[
  \frac{x^{(1)}_{j}- x^{(1)}_{i}}{d(\bm{x}_i,\bm{x}_j)} =
  \cos (\varphi_{ji}) ,
\]
where $\varphi_{ji}$ is the angle between the segment
$[{\bm x}_i, {\bm x}_j]$ and the axis $Ox^{(1)}$. 
In this case, $l_{ij} \cos (\varphi_{ji})$ is the projection of the $j$-th face 
of the Voronoi polygon  $\Omega_i$ on the axis $Ox^{(1)}$. For a closed polygon
\[
  \Lambda x^{(1)} =
  - \frac{1}{V_{i}}
  \sum_{j \in \mathcal{W}(i)} l_{ij}\cos (\varphi_{ji}) = 0.
\]

Taking this into account, for the function $W({\bm x})$, we obtain directly
\[
  (\Lambda W({\bm x}))_i =
  \mu (\Lambda d^2({\bm x}, {\bm x}_i))_i =
  \frac{\mu }{V_{i}}
  \sum_{j \in \mathcal{W}(i)} l_{ij}
  d(\bm{x}_i,\bm{x}_j) = 4 .
\]
That is why for $\mu = 0.25$, we have
\[
  \Lambda W({\bm x}) = 1,
  \quad {\bm x} \in \omega .
\]
At the boundary nodes
\[
  W({\bm x})  \ge 0,
  \quad {\bm x} \in \partial \omega ,
\]
and therefore the function $W({\bm x})$ is a majorant for the problem
(\ref{108}). For the constant $M_0$ in the inequality (\ref{104}),
the following estimate holds:
\[
  M_0 \le \max_{{\bm x} \in \omega} W({\bm x}) =
  \frac{l_1^2}{16} + \frac{l_2^2}{16} .
\]
This completes the proof of the lemma.
\qed

This makes possible to formulate the main result concerning the properties
of the difference operator of diffusive transport.

\begin{theorem}\label{t-6}
For the grid operator of diffusive transport determined from (\ref{44}), the inequality
\begin{equation}\label{112}
D = D^* \ge \frac{\kappa}{{M}_0} {E}
\end{equation}
is valid for the grid functions from the space $H = L_2(\omega)$. 
\end{theorem}

It is important that the constant $M_0$ in the Friedrichs inequality
(\ref{104}) is independent of nodes of the computational domain, and the estimate itself
is quite similar to the estimate for the differential operator.

\subsection{Convective transport operators}

In the construction of difference operators of convective transport,
we start with the operator in the divergent form (\ref{44}). Assume that
\begin{equation}\label{113}
  C_2 u = (C_2 u)_{i} \approx \frac{1}{V_{i}} \int\limits_{\Omega_i}
  \mathcal{C}_2 u d \bm{x}
\end{equation}
for the interior nodes of the grid. For the right-hand side, we have
\[
  \int\limits_{\Omega_i} \mathcal{C}_2 u d \bm{x} =
  \int\limits_{\partial \Omega_i} (\bm{v},\bm{n}) u d \bm{x}.
\]
Similarly to the case of the diffusive transport approximation,
the normal component of the velocity is referred to the midpoint of the segment connecting
grid nodes. Introducing notation
\[
  b_{ij} = (\bm{v},\bm{n})(\bm{x}_{ij}) ,
\]
from (\ref{113}), the difference operator of convective transport  is written as
\begin{equation}\label{114}
  (C_2 y)_{i} = \frac{1}{V_{i}}
  \sum_{j \in \mathcal{W}(i)} l_{ij} b_{ij}
  \frac{y_j+y_i}{2},
  \quad \bm{x}_i \in \omega .
\end{equation}

Now we discuss approximations of the convective transport operator in
the nondivergent form (\ref{43}). There do exist several opportunities. 
The first one is connected with the use of a sufficiently
complicated (not so evident) structure for the difference
convective transport operator in the nondivergent form based on
special formulas of approximate integration. But we have
another way. Using the idea  of the method of support operators, first, we design
a simple approximation of the convective transport operator in
the divergent form (\ref{114}). To obtain a difference
approximation of the operator (\ref{43}), we search a difference operator
that is adjoint to (\ref{114}). In doing so, we get
a difference operator of the convective transport in the nondivergent form.
Such an opportunity is essential for developing approximations on irregular grids.

Straightforward calculations yield
\[
\begin{split}
  (C_2 y, w) & =
  \frac{1}{2}  \sum_{\bm{x}_i \in \omega}
  \sum_{j \in \mathcal{W}(i)}
  l_{ij} b_{ij} y_i w_i +
  \frac{1}{2}  \sum_{\bm{x}_i \in \omega}
  \sum_{j \in \mathcal{W}(i)}
  l_{ij} b_{ij} y_j w_i \\
  & =
  \frac{1}{2}  \sum_{\bm{x}_i \in \omega}
  \sum_{j \in \mathcal{W}(i)}
  l_{ij} b_{ij} y_i w_i -
  \frac{1}{2}  \sum_{\bm{x}_i \in \omega}
  \sum_{j \in \mathcal{W}(i)}
  l_{ij} b_{ij} y_i w_j \\
  & =
  - (y, C_1 w),
\end{split}
\]
where we take into account that $b_{ij} = - b_{ji}$. This allows to
define the difference operator
\begin{equation}\label{115}
  (C_1 y)_{i} = \frac{1}{V_{i}}
  \sum_{j \in \mathcal{W}(i)} l_{ij} b_{ij}
  \frac{y_j-y_i}{2},
  \quad \bm{x}_i \in \omega .
\end{equation}
By construction, the difference operators of convective transport in
the divergent and nondivergent forms determined in accordance with
(\ref{114}), (\ref{115}) are adjoint to each other with within the sign, i.e.,
\begin{equation}\label{116}
  C_1^* = - C_2 .
\end{equation}

As for the operator of convective transport in the skew-symmetric form (\ref{45}),
using the representation
\[
  C_0 = \frac{1}{2} (C_1 + C_2) ,
\]
from (\ref{114}) and (\ref{115}), we get the most compact approximation
\begin{equation}\label{117}
  (C_0 y)_{i} = \frac{1}{2 V_{i}}
  \sum_{j \in \mathcal{W}(i)} l_{ij} b_{ij} y_j,
  \quad \bm{x}_i \in \omega .
\end{equation}
The primary feature of this difference operator is
\begin{equation}\label{118}
  C_0^* = - C_0 ,
\end{equation}
and moreover, this skew-symmetric property is true for any velocity field --- it is valid
for arbitrary vectors $\bm{v}$, not necessarily satisfying some difference analogue of
the incompressibility constraint
\begin{equation}\label{119}
 \div {\bm v} \equiv
 \sum_{\alpha=1}^{2} 
 \frac{\partial v_\alpha}{\partial x^{(\alpha)}}  = 0, \quad {\bm x} \in
\Omega .
\end{equation}

To study difference analogs of the boundedness of the convective transport operator
and its subordination to the diffusive transport operator,
we discuss the important features of the approximations (\ref{114}) and (\ref{115}) in detail.

For numerical solving continuum mechanics problems, it is essential
to have consistent approximations of the convective transport operator 
in the divergent and nondivergent forms. The consistency
is treated in the sense that one difference operator coincides with
other operator if the corresponding difference incompressibility constraint holds.
This issue is very important due to the fact that this equivalence takes place
for the differential equations, and just it ensures the fulfillment of several conservation laws.
In particular, for the incompressible Navier-Stokes equations,
the convective transport operator in the momentum equation
makes no contribution neither to the kinetic energy, nor to the individual momentum components 
(we speak of energy neutrality and neutrality with respect to the momentum). 
Unfortunately, elementary approximations ensure only one of these properties --- either for the kinetic
energy or for the momentum.

For an incompressible fluid (the constraint (\ref{119})),
at the differential level, we can use any of the above-mentioned three
equivalent forms of the convective terms (\ref{43})--(\ref{45}).
It follows from the formula of vector analysis:
\[
   {\rm div} ({\bm v} u) =
   ({\bm v}\cdot {\rm grad}u) +
   u~ {\rm div} {\bm v},
\]
which, in turn, is based on the differentiation formula
of the product of two functions. For the operators of convective transport, we have
\begin{equation}\label{120}
  \mathcal{C}_2 u = \mathcal{C}_1 u + \div \bm{v}~ u .
\end{equation}

In constructing difference approximations for the convective transport operator,
it seems natural to develop such difference operators that satisfies the property
(\ref{120}) of differential operators, i.e., the property of equivalence of various difference
approximations.

From (\ref{114}), (\ref{115}), we get
\[
  C_2 y = C_1 y + \frac{1}{V_{i}}
  \sum_{j \in \mathcal{W}(i)} l_{ij} b_{ij} y_i,
  \quad \bm{x}_i \in \omega .
\]
Transform this equality to the form that is similar to (\ref{120}):
\begin{equation}\label{121}
  C_2 y = C_1 y + \div_h\bm{v}~ y ,
\end{equation}
where the difference operator of divergence is
\begin{equation}\label{122}
  \div_h\bm{v} =  \frac{1}{V_{i}}
  \sum_{j \in \mathcal{W}(i)} l_{ij} b_{ij},
  \quad \bm{x}_i \in \omega .
\end{equation}

An approximation of the divergence for a vector by means of (\ref{122}) is obtained
using the definition
\[
  \div \bm{v} =
  \lim_{\delta \rightarrow 0} \frac{\int\limits_{\partial V} {\bm f} \ d {\bm S}}{\int\limits_V
  dV},
\]
where $V$ denotes the area, $\partial V$ stands for the boundary of the domain, and $\delta$
is the domain diameter. The expression (\ref{122}) may be treated as
the corresponding quadrature formula for the right-hand side in
the integration over the control volume for the node $\bm{x}_i \in \omega$.

In view of (\ref{121}), we write
\begin{equation}\label{123}
\begin{split}
  & C_1 y = C_0 y - \frac{1}{2}  \div_h\bm{v}~ y, \\ 
  & C_2 y = C_0 y + \frac{1}{2} \div_h\bm{v}~ y.
\end{split}  
\end{equation}
In view of (\ref{118}), from (\ref{123}), it follows immediately that
\begin{equation}\label{124}
  |(C_{\alpha} y,y)| \le M_1 \|y\|^2,
  \quad \alpha =1,2,
\end{equation}
with a constant $M_1$ that depends only on the compressibility of a medium
(at the discrete level), i.e.,
\begin{equation}\label{125}
  M_1 = \frac 12 \|\div_h {\bm v} \|_{\infty}.
\end{equation}
Here we use notation
\[
  \| y \|_{\infty} = \max_{\bm{x}_i \in \omega}
  |y(\bm{x}_i)|
\]
for the norm in $L_{\infty}(\omega)$. To prove (\ref{124}),
multiply equation (\ref{123}) scalarly by $y$ and
apply the skew-symmetric property of the operator $C_0$.

To obtain a difference analogue for the inequality representing
the subordination of the difference convective transport operator
to the difference diffusive transport operator, we start with the upper bound for the expression
\begin{equation}\label{126}
  \|C_1 y\|^{2} =
  \sum_{\bm{x}_i \in \omega} \frac{1}{V_{i}}
  \left (
  \sum_{j \in \mathcal{W}(i)} l_{ij} b_{ij}
  \frac{y_j-y_i}{2} \right )^2 .
\end{equation}
For the right-hand side, we have
\[
  \sum_{\bm{x}_i \in \omega} \frac{1}{V_{i}}
  \left (
  \sum_{j \in \mathcal{W}(i)} l_{ij} b_{ij}
  \frac{y_j-y_i}{2} \right )^2
  \le
\]
\[
  \quad
  \le \max_{\bm{x}_i \in \omega} \max_{j \in \mathcal{W}(i)}
  |b_{ij}|^{2}
  \sum_{\bm{x}_i \in \omega} \frac{1}{4 V_{i}}
  \left (
  \sum_{j \in \mathcal{W}(i)}
  l_{ij} d(\bm{x}_i,\bm{x}_j)
  \frac{y_j-y_i}{d(\bm{x}_i,\bm{x}_j)} \right )^2 .
\]
Taking into account the inequality
\[
  \left (\sum_{j \in \mathcal{W}(i)}
  \xi_j \theta_j \right )^{2} \le
  \sum_{j \in \mathcal{W}(i)} \xi_j^2
  \sum_{j \in \mathcal{W}(i)} \theta_j^2 ,
\]
and assuming
\[
  \xi_j = (l_{ij} d(\bm{x}_i,\bm{x}_j))^{1/2},
  \quad \theta_j = (l_{ij} d(\bm{x}_i,\bm{x}_j))^{1/2}
  \frac{y_j-y_i}{d(\bm{x}_i,\bm{x}_j)},
\]
we obtain
\[
  \sum_{\bm{x}_i \in \omega} \frac{1}{4 V_{i}}
  \left (
  \sum_{j \in \mathcal{W}(i)}
  l_{ij} d(\bm{x}_i,\bm{x}_j)
  \frac{y_j-y_i}{d(\bm{x}_i,\bm{x}_j)} \right )^2 \le
\]
\[
  \le
  \sum_{\bm{x}_i \in \omega} \frac{1}{4 V_{i}}
  \sum_{j \in \mathcal{W}(i)}
  l_{ij} d(\bm{x}_i,\bm{x}_j)
  \sum_{j \in \mathcal{W}(i)}
  l_{ij} d(\bm{x}_i,\bm{x}_j)
  \left ( \frac{y_j-y_i}{d(\bm{x}_i,\bm{x}_j)} \right )^2 .
\]
In view of
\[
  \sum_{j \in \mathcal{W}(i)}
  l_{ij} d(\bm{x}_i,\bm{x}_j)
  = 4V_i ,
\]
substitution into (\ref{126}) yields
\begin{equation}\label{127}
  \|C_1 y\|^{2}
  \le \max_{\bm{x}_i \in \omega} \max_{j \in \mathcal{W}(i)}
  |b_{ij}|^{2}
  \sum_{\bm{x}_i \in \omega} \frac{1}{V_{i}}
  \sum_{j \in \mathcal{W}(i)}
  l_{ij} d(\bm{x}_i,\bm{x}_j)
  \left ( \frac{y_j-y_i}{d(\bm{x}_i,\bm{x}_j)} \right )^2 .
\end{equation}

Comparing (\ref{127}) with (\ref{103}), we obtain
\begin{equation}\label{128}
  \|C y\|^{2} \le M_2 (Dy,y),
\end{equation}
where $C= C_1$ and the constant
\[
  M_2 = \frac{2}{\kappa_1}
  \max_{\bm{x}_i \in \omega} \max_{j \in \mathcal{W}(i)}
  |b_{ij}|^{2} ,
\]
with $k(\bm{x}) \geq  \kappa_1 > 0$.

For the difference convective transport operator in
the divergent form (\ref{114}), we use the difference analogue
of the Friedrichs inequality in the form (\ref{104}). 
By the representation (\ref{121}), we obtain immediately:
\[
  \|C_2y\|^{2} = \|C_1 y + \div_{h}\bm{v} y\|^{2} \le
  2 \|C_1 y\|^{2} + 2 \|\div_{h}\bm{v} y\|_{2} .
\]
From (\ref{103}) and (\ref{104}), for the last term in the right-hand side, we have
\[
  \|\div_{h} \bm{v} y\|_{2} \le \|\div_{h} \bm{v}\|_{\infty}^{2}
  \frac{M_{0}}{2}
  \sum_{\bm{x}_i \in \omega}
  \sum_{j \in \mathcal{W}(i)}
  l_{ij} d(\bm{x}_i,\bm{x}_j)
  \left (\frac{y_i-y_i}{d(\bm{x}_i,\bm{x}_j)} \right)^2 .
\]
Taking into account (\ref{103}) (with $C= C_1$), we arrive at the estimate (\ref{128})
for the operator $C= C_2$, where
\[
  M_2 = \frac{2}{\kappa_1} \left (
  2 \max_{\bm{x}_i \in \omega} \max_{j \in \mathcal{W}(i)}
  |b_{ij}|^{2} +
  M_{0} \|\div_{h} \bm{v}\|_{\infty}^{2} \right ).
\]

For the difference operator of convective transport in the skew-symmetric
form (\ref{117}), in a similar way, we establish the inequality (\ref{128})
with the constant
\[
  M_2 = \frac{1}{\kappa_1} \left (
  3 \max_{\bm{x}_i \in \omega} \max_{j \in \mathcal{W}(i)}
  |b_{ij}|^{2} +
  M_{0} \|\div_{h} \bm{v}\|_{\infty}^{2} \right ).
\]

\begin{theorem}\label{t-7}
For the grid operators of convective transport defined in accordance with (\ref{114}), (\ref{115}) and
(\ref{117}), in the space of grid functions $H = L_2(\omega)$, the properties
(\ref{116}) and (\ref{118}) are valid along with
the estimate of the operator energy boundedness (\ref{124}), and 
the estimate of subordination (\ref{128})
to the difference operator of diffusive transport (\ref{102}) hold.
\end{theorem}

The above estimates (\ref{124}) and (\ref{128}) for difference operators
of convective transport are fully consistent with the continuous case.
They serve us as the basis for the study of difference convection-diffusion problems.

\subsection{Monotone approximations on triangular grids}

Monotone approximations on triangular grids are constructed \cite{vabmon1} similarly
to other grids. We separate the positive and negative parts of the normal
velocity component putting
\[
  b_{ij} = b_{ij}^{+} + b_{ij}^{-},
\]
where
\[
  b_{ij}^{+} = \frac{1}{2} (b_{ij} + | b_{ij} | ),
\]
\[
  b_{ij}^{-} = \frac{1}{2} (b_{ij} - | b_{ij} | ) .
\]

To approximate the right-hand side of (\ref{113}), we will use
the value of the grid function either in the central or in the peripheral
node depending on the sign of the velocity. This leads us to
the difference operator of convective transport in the form 
\begin{equation}\label{129}
  (C_2 y)_{i} = \frac{1}{V_{i}}
  \sum_{j \in \mathcal{W}(i)} l_{ij}
  (b_{ij}^{-} y_j + b_{ij}^{+} y_i),
  \quad \bm{x}_i \in \omega .
\end{equation}

If we apply the difference divergence operator, then,
for the difference analog of (\ref{43}), we have the expression
\begin{equation}\label{130}
  (C_1 y)_{i} = \frac{1}{V_{i}}
  \sum_{j \in \mathcal{W}(i)} l_{ij} b_{ij}^{-}
  (y_j-y_i),
  \quad \bm{x}_i \in \omega .
\end{equation}

Thus, we have developed the approximations (\ref{129}) and (\ref{130})
for the convective transport operators in the divergent (\ref{44}) and
nondivergent (\ref{43}) forms using the upwind differences.

For the boundary value problems (\ref{88}), (\ref{89}) and (\ref{89}), (\ref{94}),
we put into the correspondence the difference problems
\begin{equation}\label{131}
  Cy + Dy = \varphi(\bm{x}),
  \quad \bm{x} \in \omega ,
\end{equation}
for the grid functions 
$y(\bm{x}) = 0, \ \bm{x} \in \partial \omega$. 
For the right-hand side of (\ref{131}), suppose, e.g.,
\[
  \varphi(\bm{x}) =
  \frac{1}{V_i} \int\limits_{\Omega_i} f(\bm{x}) \ d \bm{x},
  \quad \bm{x} \in \omega.
\]

To employ the fulfillment of the  maximum principle at the discrete level, 
we rewrite the difference problem (\ref{131}) in the form
\begin{equation}\label{132}
  \alpha_i y_i -
  \sum_{j \in \mathcal{W}(i)} \beta_{ij} y_j = \phi_{i},
  \quad \bm{x}_i \in \omega,
\end{equation}
\begin{equation}\label{133}
  y_i = 0,
  \quad \bm{x} \in \partial \omega .
\end{equation}
Assume that $\overline{\omega}$  is a connected grid.

For the difference problem (\ref{132}), (\ref{133}), the maximum principle is valid,
i.e., the difference scheme is monotone \cite{2001Samarskii} 
under the following restrictions:
\begin{equation}\label{134}
  \alpha_i  > 0,
  \quad \beta_{ij} > 0,
  \quad   j \in \mathcal{W}(i),
\end{equation}
\begin{equation}\label{135}
  \delta_i \equiv \alpha_i  -
  \sum_{j \in \mathcal{W}(i)} \beta_{ij} \geq 0,
  \quad \bm{x}_i \in \omega.
\end{equation}

The difference scheme (\ref{131}) for the problem (\ref{88}), (\ref{89}) 
based on the approximations (\ref{102}) and (\ref{130}) may be written in
the form (\ref{132}), (\ref{133}) with the coefficients
\[
   \alpha_i =  - \frac{1}{V_{i}}
  \sum_{j \in \mathcal{W}(i)} l_{ij} b_{ij}^{-} +
  \frac{1}{V_{i}}
  \sum_{j \in \mathcal{W}(i)} l_{ij}
  k(\bm{x}_{ij})
  \frac{1}{d(\bm{x}_i,\bm{x}_j)},
\]
\[
  \beta_{ij} =
  - \frac{1}{V_{i}}
  l_{ij} b_{ij}^{-} +
  \frac{1}{V_{i}}
  k(\bm{x}_{ij})
  \frac{l_{ij}}{d(\bm{x}_i,\bm{x}_j)},
  \quad   j \in \mathcal{W}(i),
\]
\[
  \delta_i = 0,
  \quad \bm{x}_i \in \omega .
\]
The monotonicity conditions (\ref{134}), (\ref{135}) are unconditionally valid.

Now consider the scheme (\ref{131}) with
the difference operator of convective transport  $C = C_1$ defined
according to (\ref{115}). The scheme (\ref{102}), (\ref{115}), (\ref{131}) may be
represented in the canonical form (\ref{132}), (\ref{133}) with
\[
   \alpha_i =  - \frac{1}{2 V_{i}}
  \sum_{j \in \mathcal{W}(i)} l_{ij} b_{ij} +
  \frac{1}{V_{i}}
  \sum_{j \in \mathcal{W}(i)} l_{ij}
  k(\bm{x}_{ij})
  \frac{1}{d(\bm{x}_i,\bm{x}_j)},
\]
\[
  \beta_{ij} =
  - \frac{1}{2 V_{i}}
  l_{ij} b_{ij} +
  \frac{1}{V_{i}}
  k(\bm{x}_{ij})
  \frac{l_{ij}}{d(\bm{x}_i,\bm{x}_j)},
  \quad   j \in \mathcal{W}(i),
\]
\[
  \delta_i = 0,
  \quad \bm{x}_i \in \omega .
\]

Define a local grid Peclet number as follows:
\[
  {\rm Pe}_{ij} =
   \frac{|b_{ij}| d(\bm{x}_i,\bm{x}_j)}
   {k(\bm{x}_{ij})}
  \quad   j \in \mathcal{W}(i),
  \quad \bm{x}_i \in \omega .
\]
The monotonicity condition (\ref{134}) leads to the restrictions
\begin{equation}\label{136}
  {\rm Pe}_{ij} < 2,
  \quad   j \in \mathcal{W}(i),
  \quad \bm{x}_i \in \omega .
\end{equation}
Such restrictions are typical if we apply the standard central-difference
approximations on regular grids.

For the convection-diffusion equation with the divergent convective
terms (\ref{89}), (\ref{94}), the use of the upwind
approximations (\ref{102}) and (\ref{129}) yields
\[
   \alpha_i =   \frac{1}{V_{i}}
  \sum_{j \in \mathcal{W}(i)} l_{ij} b_{ij}^{+} +
  \frac{1}{V_{i}}
  \sum_{j \in \mathcal{W}(i)} l_{ij}
  k(\bm{x}_{ij})
  \frac{1}{d(\bm{x}_i,\bm{x}_j)},
\]
\[
  \beta_{ij} =
  - \frac{1}{V_{i}}
  l_{ij} b_{ij}^{-} +
  \frac{1}{V_{i}}
  k(\bm{x}_{ij})
  \frac{l_{ij}}{d(\bm{x}_i,\bm{x}_j)},
  \quad   j \in \mathcal{W}(i),
\]
\[
  \delta_i = {\rm div}_h\bm{v},
  \quad \bm{x}_i \in \omega .
\]
Thus, the standard monotonicity conditions (\ref{134}), (\ref{135})
are valid only if ${\rm div}_h\bm{v} \ge 0$.

A similar situation occurs in the consideration of difference schemes
on rectangular grids. In this case, the unconditional fulfillment
of the maximum principle for schemes with the upwind differences
designed for the difference equation (\ref{131}) may be associated with
diagonal dominance by columns rather than by rows (as the conditions (\ref{134}), (\ref{135})).
The second possibility, which seems more promising for schemes on unstructured grids, 
involves the establishment of the maximum principle in the standard formulation for the conjugate problem.

Consider the operator that is adjoint to $C_2$ and is defined according to (\ref{129}).
Taking into account  that
$l_{ij}=l_{ji}, \ b_{ij}^{+} = - b_{ji}^{-}$, 
we obtain
\[
  (C_2 y,v) =
  \sum_{\bm{x}_i \in \omega}
  \sum_{j \in \mathcal{W}(i)} l_{ij}
  (b_{ij}^{-} y_j + b_{ij}^{+} y_i) v_i =
  \sum_{\bm{x}_i \in \omega}
  \sum_{j \in \mathcal{W}(i)} l_{ij}
  y_i b_{ij}^{+} (v_i - v_j) .
\]
Therefore
\begin{equation}\label{137}
  C_2^{*} v =
  \frac{1}{V_{i}}
  \sum_{j \in \mathcal{W}(i)} l_{ij} b_{ij}^{+}
  (v_i - v_j),
  \quad \bm{x}_i \in \omega .
\end{equation}

As for the adjoint problem
\begin{equation}\label{138}
  C_2^{*} v + D v = \varphi(\bm{x}),
  \quad \bm{x} \in \omega ,
\end{equation}
the unconditional fulfillment of the maximum principle
written in the standard formulation is established in a usual fashion. 
Recall that we speak of the formulation for the maximum principle in the following form ---
if the conditions (\ref{134}), (\ref{135}) are true, 
then the solution of the problem \ref{132}), (\ref{133}) is non-negative (nonpositive)
for the non-negative (nonpositive) right-hand side (\ref{132}).

Now we show that from the fulfillment of the maximum principle for the adjoint
problem, it follows that the maximum principle is satisfied for the original problem.
For each $\bm{x}_i \in \omega$, we define the grid function
\[
  \delta_h (\bm{x} - \bm{x}_i) =
  \left \{
  \begin{array}{ll}
  {\displaystyle \frac{1}{V_i},} & \bm{x} = \bm{x}_i,
  \\[10pt]
  0, & \bm{x} \neq \bm{x}_i .
  \end{array}
  \right .
\]
Suppose that the grid function $G(\bm{x}, \bm{x}_i)$ for
the given $\bm{x}_i \in \omega$ is the solution of the problem
\begin{equation}\label{139}
  C_2^* G + D G = \delta_h (\bm{x} - \bm{x}_i) ,
  \quad \bm{x} \in \omega .
\end{equation}
Due to the fact that the maximum principle holds for the adjoint difference problem (\ref{129}),
we have $G(\bm{x}, \bm{x}_i) \ge 0$.

Multiply equation (\ref{139}) scalarly by the solution of the original boundary value problem
\[
  C_2 y + D y = \varphi(\bm{x}),
  \quad \bm{x} \in \omega .
\]
By virtu of (\ref{139}), the solution is represented as
\[
  y(\bm{x}_i) = (G(\bm{x},
  \bm{x}_i),\varphi(\bm{x})) .
\]
Thus, we get $y(\bm{x}_i) \ge 0, \ \bm{x} \in \omega$.
Therefore, the maximum principle holds also for the original problem
(\ref{102}), (\ref{129}), (\ref{131}).

The study of the difference scheme (\ref{102}), (\ref{114}), (\ref{131}) is conducted in as similar way. 
The monotonicity of the above difference scheme (\ref{102}), (\ref{114}), (\ref{131}) 
is established under the restrictions (\ref{136}). Our investigation results in
the following statement.

\begin{theorem}\label{t-8}
The upwind difference schemes 
(\ref{102}), (\ref{130}), (\ref{131}) and (\ref{102}), (\ref{129}), (\ref{131})
for the convection-diffusion equations
(\ref{102}), (\ref{130}), (\ref{131}) and (\ref{102}), (\ref{129}), (\ref{131})
are unconditionally monotone, 
whereas the schemes
(\ref{102}), (\ref{115}), (\ref{131}) and (\ref{102}), (\ref{114}), (\ref{131})
satisfy the maximum principle under the restrictions (\ref{136}). 
\end{theorem}

The above approximations for elliptic operators of convection-diffusion
are used for discretization in space on irregular grids for
numerical solving time-dependent problems.

\section{Discretization in time} 

Discretization in space results in the Cauchy problem for systems of ODEs
treated as an operator-differential equation in the appropriate spaces.
Two- or three-level difference schemes are used for numerical solving these equations. 
This part of the work discusses issues of constructing unconditionally stable schemes for
the approximate solution of unsteady convection-diffusion problems.
The investigation is based on the general theory of stability (well-posedness)
for operator-difference schemes.

\subsection{Two-level operator-difference schemes}

We start with the key concepts of the stability theory for operator-difference schemes 
considered in finite-dimensional Hilbert spaces.
Next, for two-level difference schemes, we formulate criteria of stability 
with respect to the initial data. And finally, typical estimates for stability 
with respect to the initial data and the right-hand side are presented.

For simplicity, we define a uniform grid in time as follows:
\[
  \bar\omega_\tau =
  \omega_\tau\cup \{T\} =
  \{t^n=n\tau,
  \quad n=0,1,...,N_0,
  \quad \tau N_0=T\} .
\]
Denote by $A,B:H\to H$ linear operators in $H$ depending, in general, on $\tau$, $t^n$.
Consider the Cauchy problem for an operator-difference equation
\begin{equation}\label{140}
  B(t^n)\frac {y^{n+1}-y^n}{\tau}+A(t^n)y^n=\varphi^n,
  \quad t^n\in\omega_\tau,
\end{equation}
\begin{equation}\label{141}
  y^0=u^0,
\end{equation}
where $y^n=y(t^n)\in H$ is a desired function and $\varphi^n,u^0\in H$ are given. We
use the index-free notation of the theory of difference schemes:
\[
  y=y^n,
  \quad \hat y=y^{n+1},
  \quad \check y=y^{n-1},
\]
\[
  y^{\bar t}=\frac {y-\check y}{\tau},
  \quad y_t=\frac {\hat y-y}{\tau}.
\]
Then equation (\ref{140}) may be written as
\begin{equation}\label{142}
  By_t+Ay=\varphi,
  \quad t\in\omega_\tau.
\end{equation}

We define a two-level difference scheme as a set of the Cauchy problems
(\ref{140}), (\ref{141}) that depend on the parameter $\tau$. The formulation
(\ref{140}), (\ref{141}) (as well as (\ref{141}), (\ref{142})) 
is called the canonical form of two-level schemes.

For solvability of the Cauchy problem at a new time level, it is assumed
that $B^{-1}$ does exist. Then equation (\ref{142}) may be written as
\begin{equation}\label{143}
  \hat y = Sy + \tau\tilde\varphi,
  \quad S=E-\tau B^{-1}A,
  \quad \tilde\varphi = B^{-1}\varphi,
\end{equation}
where, as usual, $E$ is the identity operator. The operator $S$ is called the
transition operator of the two-level scheme (the transition from a current
time level to the next one).

A two-level scheme is called stable if there exist positive constants $m_1$ and $m_2$, 
independent of $\tau$, $u^0$,  and $\varphi$, such that for any 
$u^0\in H$, $\varphi\in H$, $t\in\bar\omega_\tau$,
for the solution of (\ref{140}), (\ref{141}), the following estimate is valid:
\begin{equation}\label{144}
  \|y^{n+1}\|\le m_1\|u^0\|+
  m_2\max_{0\le \theta \le t^n}\|
  \varphi(\theta)\|_*,
  \quad t^n \in \omega_{\tau},
\end{equation}
where $\|\cdot\|$ and $\|\cdot\|_*$ are some norms. The inequality (\ref{144}) 
reflects the continuous dependence of the solution of (\ref{140}), (\ref{141})
on the input data

The difference scheme
\begin{equation}\label{145}
  B(t^n)\frac {y^{n+1}-y^n}{\tau}+A(t^n)y^n=0,
  \quad t^n\in\omega_\tau,
\end{equation}
\begin{equation}\label{146}
  y^0=u^0
\end{equation}
is called stable with respect to the initial data if for the solution of
(\ref{145}), (\ref{146}), the following estimate holds:
\begin{equation}\label{147}
  \|y^{n+1}\|\le m_1\|u^0\|,
  \quad t^n \in \omega_{\tau} .
\end{equation}
The two-level difference scheme
\begin{equation}\label{148}
  B(t^n)\frac {y^{n+1}-y^n}{\tau}+A(t^n)y^n=\varphi^n,
  \quad t^n\in\omega_\tau,
\end{equation}
\begin{equation}\label{149}
  y^0=0
\end{equation}
is called stable with respect to the right-hand side if the solution satisfies the inequality
\begin{equation}\label{150}
  \|y^{n+1}\|\le
  m_2\max_{0\le \theta \le t^n}\|
  \varphi(\theta)\|_*,
  \quad t^n \in \omega_{\tau} .
\end{equation}

The difference scheme (\ref{145}), (\ref{146}) is said to be $\rho$-stable (uniformly
stable) with respect to the initial data in $H_D$ if there exist constants $\rho>0$
and $m_1$, independent of $\tau$ and $n$, such that for any $n$ and all $y^n\in H$, the solution
$y^{n+1}$ of the difference equation (\ref{145}) satisfies the estimate
\begin{equation}\label{151}
  \|y^{n+1}\|_D \le\rho \|y^n\|_D,
  \quad t^n \in \omega_{\tau} ,
\end{equation}
and $\rho^n\le m_1$.

In the theory of difference schemes, one of the following quantities is selected as $\rho$:
\[
\begin{array}{lll}
  \rho=1,\\[5pt]
  \rho=1+c\tau,\quad c>0,\\[5pt]
  \rho=\exp{(c\tau)},
\end{array}
\]
where a constant $c$ is independent of $\tau$, $n$.

In view of (\ref{145}), rewrite equation (\ref{143}) in the form
\begin{equation}\label{152}
  y^{n+1}=Sy^n.
\end{equation}
The requirement of $\rho$-stability is equivalent to the fulfillment of the bilateral
operator inequality
\begin{equation}\label{153}
  -\rho D\le DS \le \rho D,
\end{equation}
if $DS$ is self-adjoint ($DS = S^*D$). For an arbitrary operator of transition in
(\ref{152}), the condition of $\rho$-stability is given by
\begin{equation}\label{154}
S^* D S \le \rho^2 D.
\end{equation}

Let us formulate the discrete analog of Gronwall's lemma.
\begin{lemma}\label{l-5}
From the estimate for the difference solution at the $n+1$-st time level
\begin{equation}\label{155}
  \|y^{n+1}\| \le \rho \|y^n\|+\tau\|\varphi^n\|_* ,
\end{equation}
it follows that the a priori estimate
\begin{equation}\label{156}
  \|y^{n+1}\| \le \rho^{n+1} \|y^0\| +
  \sum_{k=0}^{n}\tau\rho^{n-k} \|\varphi^k\|_*
\end{equation}
holds.
\end{lemma}

Thus, from the levelwise estimate, we obtain an a priori estimate for the
difference solution at any time level.

Let us formulate the basic criteria for stability of two-level schemes with respect
to the initial data \cite{2001Samarskii,SamMatVab}. The most important is the following 
theorem, proved by Samarskii, on the exact (coinciding necessary and sufficient) condition for stability
in $H_A$.

\begin{theorem}\label{t-9}
Assume that in equation \ref{145}), the operator $A$ is a positive and self-adjoint
operator independing of $n$. The condition
\begin{equation}\label{157}
  B\ge \frac {\tau}{2} A,
  \quad t\in\omega_\tau
\end{equation}
is necessary and sufficient for stability in $H_A$, i.e., for the fulfillment of the
estimate
\begin{equation}\label{158}
  \|y^{n+1}\|_A \le \|u^0\|_A,
  \quad t\in\omega_\tau .
\end{equation}
\end{theorem}

\begin{proof}
Multiplying equation (\ref{145}) scalarly by $y_t$, we get
\begin{equation}\label{159}
  (By_t,y_t)+(Ay,y_t)=0.
\end{equation}
Using the representation
\[
  y= \frac{1}{2}(y+\hat{y})- \frac{1}{2} \tau y_t,
\]
rewrite (\ref{159}) as
\begin{equation}\label{160}
  \left ( \left (B-\frac \tau2 A \right )y_t,y_t \right )+
  \frac {1}{2\tau} (A(\hat{y}+y),\hat{y}-y)=0.
\end{equation}
For the self-adjoint operator $A$, we have $(Ay,\hat{y})=(y,A\hat{y})$ and
\[
  (A(\hat{y}+y),\hat{y}-y)=(A\hat{y},\hat{y})-(Ay,y).
\]
Substituting these relations into (\ref{160}) and using the condition (\ref{157}), we
obtain the inequality
\begin{equation}\label{161}
  \|y^{n+1}\|_A\le\|y^n\|_A,
\end{equation}
which ensures the desired estimate (\ref{158}).

To prove the necessity of the inequality (\ref{158}), assume that the scheme is
stable in  $H_A$, i.e., the inequality (\ref{158}) holds. We prove that this implies the
operator inequality (\ref{157}). Consider (\ref{160}) at the initial time level $n=0$:
\[
  2\tau \left ( \left (B-\frac{\tau}{2} A \right )w,w \right )+
  (Ay_1,y_1)=(Ay_0,y_0),
  \quad w = \frac {y_1 - y_0}{\tau} .
\]
In view of (\ref{158}), this identity holds only if
\[
  \left ( \left (B-\frac {\tau}{2} A \right )w,w \right )\ge 0.
\]
Let $y_0=u_0\in H$ be an arbitrary element, then the element $w=-B^{-1}Au_0\in H$ 
is arbitrary, too. Indeed, for any element $w\in H$, we obtain $u_0=-A^{-1}Bw\in H$
since $A^{-1}$ exists. Thus, the inequality holds for any $w\in H$,
i.e., we have the operator inequality (\ref{157}).
\end{proof}

The condition (\ref{157}) is necessary and sufficient for stability not only in
$H_A$, but also in other norms. We now formulate (without proof) the stability
result for $H_B$.

\begin{theorem}\label{t-10}
Assume that in (\ref{145}), (\ref{146}), the operators
$A$ and $B$ are constant and
\begin{equation}\label{162}
  B=B^* > 0,
  \quad A=A^* > 0.
\end{equation}
Then the condition (\ref{157}) is necessary and sufficient for stability of the
scheme (\ref{145}), (\ref{146}) with respect to the initial data in $H_B$ with $\rho=1$.
\end{theorem}

The consideration of general time-dependent problems is based on using $\rho$-stability.

\begin{theorem}\label{t-11}
Let $A$ and $B$ be constant operators and
\[
  A=A^*,
  \quad B=B^*>0.
\]
Then the condition
\begin{equation}\label{163}
  \frac{1-\rho}{\tau}B \le A \le \frac{1+\rho}{\tau}B
\end{equation}
is necessary and sufficient for the $\rho$-stability of the scheme (\ref{145}), (\ref{146}) in $H_B$,
i.e., for the fulfilment of
\[
\|y^{n+1}\|_B\le\rho\|y^n\|_B.
\]
\end{theorem}

\begin{proof}
Writing (\ref{145}) in the form of (\ref{152}), we get from (\ref{153}) the following
condition for stability in $H_B$:
\[
  -\rho B\le B - \tau A \le \rho B.
\]
This bilateral operator inequality can be formulated in a more traditional
representation using inequalities in the form of (\ref{163}) for the scheme operators.
\end{proof}

We emphasize that in this theorem there is no assumption that the operator
$A$ is positive (or at least non-negative). Under the additional assumption
on the positiveness of $A$, we get that the condition (\ref{163}) is necessary and
sufficient for the $\rho$-stability of the scheme (\ref{145}), (\ref{146}) in $H_A$.

If $\rho \ge 1$, then stability, as in Theorem~\ref{t-9}, is established for two-level
difference schemes with the non-self-adjoint operator $B$.

\begin{theorem}\label{t-12}
Let $A$ be a self-ajoint, positive, and constant operator. Then
under the condition
\begin{equation}\label{164}
  B \ge \frac {\tau}{1 + \rho} A ,
\end{equation}
the scheme (\ref{145}), (\ref{146}) is $\rho$-stable in $H_A$.
\end{theorem}

\begin{proof}
Adding and subtracting from the basic energy identity (see (\ref{160}))
\begin{equation}\label{165}
  2 \tau \left ( \left (B-\frac{\tau}{2} A \right )y_t,y_t \right )+
  (A\hat{y},\hat{y})-(Ay,y) = 0
\end{equation}
the expression
\[
  2 \tau^2 \frac {1}{1 + \rho} (Ay_t,y_t) ,
\]
we get
\[
  2 \tau \left (\left (B-\frac {\tau}{1 + \rho} A \right )y_t,y_t \right )  +
  (A\hat{y},\hat{y})-(Ay,y) -
\]
\[
 - \frac {1-\rho}{1 + \rho} \tau^2 (Ay_t,y_t) = 0.
\]
In view of (\ref{164}) and the self-adjointness of $A$, we obtain immediately
\[
  (A\hat{y},\hat{y})-\rho (Ay,y) + (\rho -1)(A\hat {y},y) \le 0.
\]
The inequality
\[
  |(A\hat {y},y)| \le \|\hat {y}\|_A \|y\|_A
\]
with notation
\[
  \eta = \frac {\|\hat {y}\|_A} {\|y\|_A},
\]
yields the inequality
\[
  \eta^2 - (\rho-1) \eta + \rho \le 0.
\]
It holds for all $1 \le \eta \le \rho$, and so we go to the desired estimate
\[
  \|\hat {y}\|_A  \le \|y\|_A,
\]
which ensures stability in $H_A$.
\end{proof}

Now we consider a priori estimates that express stability with respect to
the right-hand side. Such estimates are employed to study convergence of
difference schemes for time-dependent problems.

First, we show that stability with respect to the initial data in $H_D,~ D = D^* >0 $ 
results in stability with respect to the right-hand side in the norm
$\|\varphi\|_*= \|B^{-1}\varphi\|_R$.

\begin{theorem}\label{t-13}
Assume that (\ref{140}), (\ref{141}) is $\rho$-stable in $H_R$ with respect to
the initial data, i.e., the estimate (\ref{151}) holds with $\varphi^n = 0$. Then the scheme
(\ref{140}), (\ref{141}) is stable with respect to the right-hand side and the following
a priori estimate is true:
\begin{equation}\label{166}
  \|y^{n+1}\|_R \le
  \rho^{n+1}\|u^0\|_R+\sum_{k=0}^{n}\tau\rho^{n-k}
  \|B^{-1}\varphi^k\|_R.
\end{equation}
\end{theorem}

\begin{proof}
Since $B^{-1}$ exists, we have that equation (\ref{140}) may be written as
\begin{equation}\label{167}
  y^{n+1} = Sy^n+\tau\tilde\varphi^n,
  \quad S=E-\tau B^{-1}A,
  \quad \tilde\varphi^n=B^{-1}\varphi^n.
\end{equation}
From (\ref{167}), we get
\begin{equation}\label{168}
  \|y^{n+1}\|_R \le \|Sy^n\|_R+\tau\|B^{-1}\varphi^n\|_R.
\end{equation}
The requirement of $\rho$-stability with respect to the initial data is equivalent
to the boundedness of the norm of the transition operator $S$:
\[
  \|Sy^n\|_R \le \rho \|y^n\|_R,
  \quad t\in\omega_\tau .
\]
Because of this, from  (\ref{168}), we obtain
\[
  \|y^{n+1}\|_R \le \rho\|y^n\|_R+\tau\|B^{-1}\varphi^n\|_R.
\]
Using the discrete analog of Gronwall's lemma, we obtain the desired estimate
(\ref{166}), which expresses the stability of the scheme with respect to the initial
data and the right-hand side.
\end{proof}

In particular, if $D=A$ or $D=B$ (under the condition $A=A^*>0$ or $B=B^*>0$), 
then, from (\ref{166}), we obtain elementary estimates for stability
in the energy space $H_A$ or $H_B$.

Some new estimates for the two-level difference scheme (\ref{140}), (\ref{141}) can
be obtained by coarsening the stability criterion (\ref{167}).

\begin{theorem}\label{t-14}
Let $A$ be a self-ajoint, positive, and constant operator and
assume that $B$ satisfies the condition
\begin{equation}\label{169}
  B \ge \frac{1+\varepsilon}{2} \tau A
\end{equation}
with a constant $\varepsilon >0$ independing of $\tau$. Then the scheme (\ref{140}), (\ref{141})
satisfies the a priori estimate
\begin{equation}\label{170}
  \|y^{n+1}\|_A^2 \le
  \|u^0\|_A^2 + \frac{1+\varepsilon}{2\varepsilon}
  \sum_{k=0}^{n} \tau \|\varphi^k\|_{B^{-1}}^2 .
\end{equation}
\end{theorem}

\begin{proof}
Multiplying equation (\ref{140}) scalarly by $2\tau y_t$, we obtain, similarly to
(\ref{165}), the energy identity
\begin{equation}\label{171}
 2\tau((B-\frac{\tau}{2} A)y_t,y_t) +
 (A\hat y,\hat y) =
 (Ay,y) + 2 \tau (\varphi,y_t).
\end{equation}
The right-hand side of the above expression can be estimated as
\[
  2\tau (\varphi,y_t)  \le
  2\tau \|\varphi\|_{B^{-1}} \|y_t\|_B \le
\]
\[
  \le 2\tau\varepsilon_1 \|y_t\|_B^2 +
  \frac{\tau}{2\varepsilon_1} \|\varphi\|_{B^{-1}}^2
\]
with a positive constant $\varepsilon_1$. Substituting this estimate into (\ref{171}), we get
\[
  2\tau \left ( \left ((1-\varepsilon_1)B-\frac{\tau}{2} A \right )y_t,y_t \right ) +
  (A\hat y,\hat y) \le (Ay,y) +
  \frac {\tau}{2\varepsilon_1} \|\varphi\|_{B^{-1}}^2.
\]
If the condition (\ref{169}) holds, then it is possible to select $\varepsilon_1$ such that
\[
  \frac {1}{1-\varepsilon_1} = 1+\varepsilon,
\]
and so
\[
  (1-\varepsilon_1)B - \frac {\tau}{2} A =
  (1-\varepsilon_1) (B-\frac{1+\varepsilon}{2}\tau A) \ge 0,
\]
\[
  (A\hat y,\hat y) \le (Ay,y) +
  \frac{1+\varepsilon}{2\varepsilon} \tau \|\varphi\|_{B^{-1}}^2.
\]
The last inequality implies the estimate (\ref{170}).
\end{proof}

\begin{theorem}\label{t-15}
Let $A$ be a self-ajoint, positive, and constant operator, and
assume that $B$ satisfies the condition
\begin{equation}\label{172}
  B \ge G + \frac{\tau}{2} A,
  \quad G = G^* > 0 .
\end{equation}
Then the solution of (\ref{140}), (\ref{141}) satisfies the a priori estimate
\begin{equation}\label{173}
  \|y^{n+1}\|_A^2 \le \|u^0\|_A^2 +
  \frac {1}{2}\sum_{k=0}^n\tau\|\varphi^k\|^2_{G^{-1}}.
\end{equation}
\end{theorem}

\begin{proof}
In the identity (\ref{171}), we employ the estimate
\[
  2\tau (\varphi,y_t) \le
  2\tau (G y_t,y_t) +
  \frac {\tau}{2} (G^{-1}\varphi, \varphi).
\]
Substituting this estimate into (\ref{171}) and taking into account (\ref{172}), we get
\[
  (A\hat y,\hat y) \le (Ay,y) +
  \frac{1}{2} \tau \|\varphi\|_{G^{-1}}^2
\]
that, by a discrete analog of Gronwall's lemma, gives (\ref{173}).
\end{proof}

The convergence study of difference schemes is conducted in various classes
of smoothness of the solution of the original differential problem, 
and therefore we must have a wide range of estimates. In particular, the right-hand side
should be estimated in different and simply calculated norms. Only typical a
priori estimates for solutions of operator-difference schemes are considered
here.

We now apply the above results to elementary schemes with weights for
an operator-differential equation of first order. The Cauchy problem
\begin{equation}\label{174}
\frac{d u}{d t} + Au = f(t), \quad t > 0,
\end{equation}
\begin{equation}\label{175}
  u(0) = u^0,
\end{equation}
with $A \geq 0$ is associated with the two-level scheme with weights
\begin{equation}\label{176}
  \frac {y^{n+1} - y^n}{\tau} +
  A (\sigma y^{n+1} + (1 - \sigma) y^n) =
  \varphi^n,
  \quad t^n \in \omega_\tau,
\end{equation}
\begin{equation}\label{177}
  y^0 = u^0 .
\end{equation}

The scheme (\ref{176}), (\ref{177}) may be written in the canonical form (\ref{142}) with
the operators
\begin{equation}\label{178}
  B = E + \sigma \tau A,
  \quad A  > 0.
\end{equation}

\begin{theorem}\label{t-16}
The scheme with weights ( \ref{176}), (\ref{177}) is stable in $H$ with respect
to the initial data if and only if the following operator inequality holds:
\begin{equation}\label{179}
  A^* + \left (\sigma - \frac{1}{2} \right ) \tau A^*A \ge 0.
\end{equation}
\end{theorem}

\begin{proof}
By $A> 0$, there exists $A^{-1}$. Multipying (\ref{176}) by $A^{-1}$, we go from
(\ref{142}), (\ref{178}) to the scheme
\[
  \tilde B \frac {y^{n+1} - y^n}{\tau} + \tilde A y^n = \tilde \varphi^n,
  \quad t^n \in \omega_\tau
\]
where
\[
  \tilde B = A^{-1} + \sigma \tau E,
  \quad \tilde A = E.
\]
The necessary and sufficient condition for stability of this scheme with respect
to the initial data in $H = H_{\tilde A}$ (Theorem~\ref{t-9}) is formulated as the inequality
\[
  A^{-1} + \left (\sigma - \frac{1}{2} \right ) \tau E \ge 0.
\]
Multiplying it from the left by $A^*$ and from the right by $A$, we obtain (\ref{179}).
This completes the proof of the theorem.
\end{proof}

If $\sigma \geq 0.5$, then the operator-difference scheme (\ref{176}), (\ref{177}) 
is unconditionally stable (stable for any $\tau > 0$).

\subsection{Difference schemes for convection-diffusion poblems}

Discretization in space of the Cauchy problem (\ref{37}), (\ref{41})  
yields the problem (see, e.g., (\ref{81}), (\ref{82})):
\begin{equation}\label{180}
  \frac{d y}{d t} + A y = \varphi(t),
  \quad A = C + D, 
  \quad 0 < t \leq T,  
\end{equation} 
\begin{equation}\label{181}
  y(0) = u^0 .
\end{equation}
With the above approximations, the grid operators of convective and diffusive
transport inherit the basic properties of differential operators
in the appropriate spaces of grid functions.
Among these properties, we recall the following as the major ones.
The constant (time-independent) grid diffusion operator is self-adjoint
and positive definite:
\begin{equation}\label{182}
  \frac{d}{d t} D = D \frac{d}{d t},
  \quad D = D^* \ge  \frac{1}{M_0} \kappa_1 E ,
  \quad \kappa_1 > 0,
  \quad M_0 > 0 .
\end{equation}
For the grid operator of convective transport in various forms
($C = C(t) = C_\alpha, \ \alpha = 0,1,2$), we have
 \begin{equation}\label{183}
  C_0 = - C_0^* ,
\end{equation} 
\begin{equation}\label{184}
 |(C_\alpha y,y) | \le M_1 \|y\|^2,
 \quad \alpha =1,2 ,
\end{equation} 
\begin{equation}\label{185}
  \|C_\alpha y\|^2 \le M_2 (Dy,y) ,
  \quad \alpha =0,1,2 , 
\end{equation}
with the corresponding positive constants $M_1$ and $M_2$.

To solve numerically the problem (\ref{180}), (\ref{181}), we consider the two-level
scheme with weights
\begin{equation}\label{186}
\begin{split}
  \frac {y^{n+1} - y^n}{\tau} &+
  C (\sigma_1 y^{n+1} + (1 - \sigma_1) y^n) \\
  & + D (\sigma_2 y^{n+1} + (1 - \sigma_2) y^n) = \varphi^n,
  \quad t^n \in \omega_\tau,  
\end{split}
\end{equation} 
\begin{equation}\label{187}
  y_0 = u_0 .
\end{equation} 
Here, e.g., we have
\[
  C = C(0.5 (t^{n+1} + t^n)),
  \quad \varphi^n = \varphi(0.5 (t^{n+1} + t^n)).
\]

Among the most important variants of the difference scheme with weights
(\ref{186}), (\ref{187}), we highlight the scheme with equal weights ($\sigma_1= \sigma_2$) and
the scheme, where convective transport is taken from a previous time level
($\sigma_1=0$).

We start with the convective transport operator in the skew-symmetric
form, i.e., $C = - C^* = C_0$. Problems with the convective transport 
operator in the nondivergent ($C=C_1$) and divergent ($C=C_2$) forms will be
considered later. Assume that in the difference scheme (\ref{186}), we have
\begin{equation}\label{188}
  \sigma_1 = \sigma_2 = \sigma .
\end{equation} 
In view of (\ref{188}), instead of (\ref{186}), we consider the difference scheme
\begin{equation}\label{189}
  \frac {y^{n+1} - y^n}{\tau}+
  (C_0 + D) (\sigma y^{n+1} + (1 - \sigma) y^n) = \varphi^n,
  \quad t^n \in \omega_\tau.
\end{equation} 

The scheme (\ref{187}), (\ref{189}) under investigation
may be written in the canonical form for the two-level difference scheme (\ref{140}), (\ref{141})
with the operators
\begin{equation}\label{190}
  B = E + \sigma \tau A,
  \quad A = C_0 + D > 0.
\end{equation}
The main peculiarity of difference schemes for the convection-diffusion equation 
is connected with non-self-adjointness of the operators $B$ and $A$. Therefore, 
it is impossible to use the above results on stability of operator-difference
schemes, which were formulated for constant self-adjoint operators.

The second important feature is associated with the fact that operators
of the difference scheme are variable in time. We consider the problems with
the time-dependent difference operator of convective transport. To obtain a
priori estimates for such problems, it is often necessary to require additionally
Lipschitz continuity of the difference operators with respect to time.

Conditions for stability of the scheme (\ref{140}), (\ref{141}),  (\ref{190})
have been presented above in the form of Theorem~\ref{t-16}.
Let us supplement this result with the corresponding stability estimate 
of the difference solution with respect to the right-hand side and the initial data. 

\begin{theorem}\label{t-17}
The difference scheme (\ref{140}), (\ref{141}),  (\ref{190}) is unconditionally stable for
$\sigma \ge 0.5$, and the difference solution satisfies the a priori estimate
\begin{equation}\label{191}
  \|y^{n+1}\|^2 \le \|u^0\|^2 + 
  \frac{1}{2} \sum_{k=0}^{n} \tau \|\varphi^k\|_{D^{-1}}^2.
\end{equation} 
\end{theorem}

\begin{proof}
Rewrite the scheme (\ref{140}), (\ref{190}) (see (\ref{189})) as follows:
\begin{equation}\label{192}
  \frac{y^{n+1} - y^n}{\tau} + A v^{n+1} = \varphi^n,
  \quad n = 0,1,..., N_0-1 ,
\end{equation} 
where
\[
  v^{n+1} = \sigma y^{n+1} + (1-\sigma y^n) =
  \left (\sigma - \frac 12 \right ) y_t + \frac 12 (y^{n+1} + y^n),
\]
\[
  y_t = \frac {y^{n+1} - y^n}{\tau} .
\]
Multiplying equation (\ref{192}) scalarly by $v^{n+1}$, we obtain
\begin{equation}\label{193}
\begin{split}
 \left (\sigma - \frac 12 \right ) \tau (y_t,y_t)  & +
 (Av^{n+1},v^{n+1}) \\
  & +
 \frac {1}{2 \tau} ((y^{n+1},y^{n+1}) - (y^n,y^n)) =  ( \varphi^n, v^{n+1}) .
\end{split}
\end{equation}
In the condition (\ref{190}), we have
$(Ay,y)  = (Dy,y)$. 
For the right-hand side, we use the estimate
\[
  ( \varphi^n, v^{n+1}) = (D v^{n+1}, v^{n+1}) + \frac{1}{4} (D^{-1} \varphi^n, \varphi^n).
\]
With this in mind, from (\ref{193}), under the conditions of the theorem we get the estimate
\[
 \|y^{n+1}\|^2 \le \|y^n\|^2 + \frac{\tau }{2} (D^{-1} \varphi^n, \varphi^n).
\]
Thus, we come to the desired estimate (\ref{191}).
\end{proof}

The a priori estimate (\ref{191}) obtained above for the
difference solution is a grid analog of the a priori estimate (\ref{60})
for the solution of the differential problem (\ref{41}), (\ref{59}), because
the convective transport operator in the skew-symmetric form under consideration
corresponds to the constant $\mathcal{M}_1 = 0$ in (\ref{60}).

Now we consider the case, where the skew-symmetry of the difference
operator of convective transport is not valid. We will study the problem with
the convective transport written in the nondivergent form, i.e., $C=C_1$. The
case the convective transport in the divergence form ($C = C_2$) is investigated
in a similar way.

Let us examine the scheme (\ref{140}), (\ref{141}), where
\begin{equation}\label{194}
  B = E + \sigma \tau A,
  \quad A = C_1 + D.
\end{equation}
It is important to distinguish two classes of problems. The simplest case is
associated with the assumption that the operator $A$ is non-negative. Such
a situation takes place, e.g., if $M_1 M_0 - \kappa_1 \le 0$  --- convective transport has
only an insignificant effect.
Indeed, in view of (\ref{182}), (\ref{184}), in the case (\ref{194}), we have
\[
\begin{split}
  (Ay,y) & = (C_1y,y) + (Dy,y) \geq - M_1 \|y\|^2 +  \frac{1}{M_0} \kappa_1 \|y\|^2 \\
  & =
  \frac{1}{M_0} (\kappa_1 - M_1 M_0) \|y\|^2 .
\end{split}
\]
Because of this, for the operator $A$, we have the following lower bound:
\begin{equation}\label{195}
  A \geq  \frac{1}{M_0} (\kappa_1 - M_1 M_0)  E .
\end{equation} 

Another case deals with slightly compressible
flows, where $A \ge 0$ under the condition $M_2 M_0 - \kappa_1 \le 0$. 
In this situation, in view of (\ref{182}) and (\ref{185}), we obtain
\[
\begin{split}
  (Ay,y) & = (C_1y,y) + (Dy,y) \geq - \|C_1y \| \|y\| +  (Dy,y) \\
  & \geq \left (1 - \left (\frac{M_2 M_0}{\kappa_1 } \right)^{1/2} \right )  (Dy,y) .
\end{split}
\]
Under these restrictions on parameters of the problem,
we can apply the results on the unconditional stability (Theorem~\ref{16}) for
the difference scheme (\ref{140}), (\ref{141}), (\ref{194})
in $H$ for $\sigma \ge 0.5$.

In the general case, we cannot rely on the non-negativity of the operator
$A$. This leads to the fact that the conventional schemes with weights are not unconditionally stable 
under the standard restrictions $\sigma \ge 0.5$. Let us consider
the difference scheme (\ref{140}), (\ref{141}), (\ref{194}) as an example.

The solvability of the scheme (\ref{140}), (\ref{141}), (\ref{194}) ($B > 0$), in view of the fact
that the operator $A$ is not non-negative, takes place under the constraint of
an appropriately small time step --- we speak of conditional solvability. 
Taking into account (\ref{194}), (\ref{195}) with
$M_1 M_0 - \kappa_1 > 0$, we get the following restriction on a time step:
\begin{equation}\label{196}
  \tau \le \tau_1 = \frac {M_0}{\sigma (M_1 M_0 - \kappa)} .
\end{equation}
In this case (see Theorem~\ref{2}, the estimate (\ref{60}) for the solution of the differential problem), 
it is necessary to be oriented to obtaining an appropriate estimate that expresses conditions for $\varrho$-stability.
 
We have already formulated the necessary and sufficient condition for $\varrho$-stability
in the case with the constant self-adjoint operators $B$ and $A$. Therefore,
our study will be based on the schemes with weights of type (\ref{140}), (\ref{141}), (\ref{194})
considered above.

Let us define new grid functions $v^n$:
\begin{equation}\label{197}
  y^n = \varrho^n v^n,
  \quad n = 0,1,..., N_0,
  \quad \varrho > 0 .
\end{equation}
A condition for $\varrho$-stability for $y^n$ is evidently equivalent to stability ($\varrho=1$)
for $v^n$. Substitution of (\ref{197}) into (\ref{145}) yields the difference scheme
\begin{equation}\label{198}
  \tilde B \frac {v^{n+1} - v^n}{\tau} + \tilde A v^n = 0,
  \quad t_n \in \omega_\tau,
\end{equation} 
where
\begin{equation}\label{199}
  \tilde B = \varrho E + \sigma \varrho \tau A,
  \quad
  \tilde A = \frac {\varrho -1}{\tau} E + (1 + \sigma (\varrho -1)) A.
\end{equation}
It is possible to use the following representation for the operators of the
difference scheme (\ref{198}):
\begin{equation}\label{200}
  \tilde B = G + \tilde \sigma \tau \tilde A,
\end{equation}
which treat the scheme (\ref{198}) as a scheme with weights. 
In view of the representation (\ref{199}), we obtain in (\ref{200}):
\begin{equation}\label{201}
  G = \frac {\varrho} {1 + \sigma (\varrho -1)} E,
  \quad \tilde \sigma = \frac {\sigma \varrho} {1 + \sigma (\varrho -1)} .
\end{equation} 

Similarly to Theorem~\ref{t-16}, we prove the stability of the scheme (\ref{198}), (\ref{200})
in $H_G$, i.e., in $H$ with $\tilde \sigma \ge 0.5$ under the constraint $\tilde A \ge 0$. Taking
into account (\ref{201}), we get the desired condition on a weight of the difference
scheme (\ref{198}), (\ref{199}):
\begin{equation}\label{202}
  \sigma \ge \frac {1} {1 + \varrho} .
\end{equation} 

The non-negativity of the operator $\tilde A$ is connected with an appropriate
choice of $\varrho$. In view of the stability estimate for the differential problem (see
Theorem~\ref{t-2}, the estimate (\ref{60})), it is natural to set
\begin{equation}\label{203}
  \varrho = 1 + M_1 \tau .
\end{equation}
Taking into account the estimate (\ref{195}), 
the condition $\tilde A \ge 0$ (see (\ref{199})) is fullfiled for
\[
  M_1 M_0 - (1 + \sigma \tau M_1) (M_1 M_0 - \kappa_1) \ge 0.
\]
This inequality yields the following restriction on a permissible time step:
\begin{equation}\label{204}
  \tau \le \tau_2 = \frac {\kappa_1}{\sigma M_1 (M_1 M_0 - \kappa_1)} .
\end{equation}
A comparison with the estimate (\ref{196}) shows that the time step restriction
(\ref{204}) is slightly stronger ($\tau_2 < \tau_1$, we recall, $M_1 M_0 > \kappa_1$). Summarizing,
we obtain the following statement.

\begin{theorem}\label{t-18}
The scheme with weights (\ref{140}), (\ref{141}), (\ref{194}) under the constraint
$M_1 M_0 > \kappa_1$ is $\varrho$-stable in $H$, where $\varrho$ is defined according to (\ref{203}), if the
weight $\sigma$ satisfies the restriction (\ref{202}) and a time step meets the condition
(\ref{204}).
\end{theorem}

This statement complements Theorem~\ref{t-17}, which ensures the stability of
the scheme (\ref{140}), (\ref{141}), (\ref{194}) under the constraint
$M_1 M_0 \le \kappa_1$ in $H$ with $\sigma \ge 0.5$.
Possible non-negativity of the operator $A= C_1 +D$ leads to the situation,
where we must use $\varrho$-stability. In addition, we impose (see (\ref{204})) restrictions on a time step.

In solving convection-diffusion problems, it is reasonable to focus on difference
schemes, where a part of the operator $A$ (it is, of
course, the convective transport operator) is taken from the previous time
level \cite{SamVabConv}. Such explicit-implicit schemes from the above class of two-level
schemes with weights are considered in \cite{ExplicitImlicit}. Suppose now that in the
difference scheme (\ref{186}), we have
\begin{equation}\label{205}
  \sigma_1 = 0,
  \quad \sigma_2 = \sigma.
\end{equation}
The homogeneous ($\varphi_n = 0$) scheme (\ref{186}), (\ref{205}) is reduced to the canonical
form (\ref{140}) if we define
\begin{equation}\label{206}
  B = E + \sigma \tau D,
  \quad A = C + D.
\end{equation}
For any $\tau > 0$, we have $B > 0$, and therefore the discrete equation (\ref{186}), (\ref{205})
is solvable at every time level. Let us formulate a sufficient condition for
$\varrho$-stability of the difference scheme for the convection-diffusion equation in
$H_D$.

\begin{theorem}\label{t-19}
The solution of the explicit-implicit scheme (\ref{186}), (\ref{205})
with $\sigma \ge 0.5$ satisfies the estimate
\begin{equation}\label{207}
  \|y^{n+1}\|_D \le \varrho \|y^{n}\|_D
\end{equation} 
where
\begin{equation}\label{208}
  \varrho = 1 + \frac {M_2}{4} \tau ,
\end{equation} 
and $M_2$ is the constant from the inequality (\ref{185}).
\end{theorem}
\begin{proof}
Multiply (\ref{186}) scalarly by $2\tau y_t = 2(y^{n+1}-y^n)$ and, in view of (\ref{206}),
obtain the energy identity
\begin{equation}\label{209} 
  \tau ((2B - \tau D)y_t,y_t) + (Dy^{n+1},y^{n+1}) -
  (Dy^n,y^n) + 2\tau (C y^n,y_t) = 0.
\end{equation}
Taking into account the representation (\ref{206}) and the constraint $\sigma \ge 0.5$,
from (\ref{209}), it follows the inequality
\begin{equation}\label{210}
  2\tau (y_t,y_t) + (D y^{n+1},y^{n+1}) -
  (D y^n,y^n) \le 2\tau |(C y^n,y_t)|.
\end{equation}
In view of (\ref{185}), the right-hand side is evaluated as follows:
\[
  |(C y^n,y_t)| \le \|y_t\|^2 +
  \frac {1}{4} \|C y^n\|^2 \le
  \|y_t\|^2 + \frac {M_2}{4} (D y^n,y^n) .
\]
Substitution into (\ref{210}) yields
\[
  (D y^{n+1},y^{n+1}) \le \left (1 + \frac {M_2}{2}\tau \right ) (D y^n,y^n) .
\]
Therefore, in view of inequality
\[
  1 + \frac {M_2}{2}\tau \le
  \left ( 1 + \frac {M_2}{4}\tau \right )^2 ,
\]
we obtain the desired stability estimate  (\ref{207}), (\ref{208}).
\end{proof}

The $\varrho$-stability estimate (\ref{207}), (\ref{208}), derived here, is fully consistent with
the corresponding estimate for the differential problem
(see, e.g., the estimate (\ref{62}) and the proof of Theorem~\ref{t-2}). 
An important point is that, in contrast to Theorem~\ref{t-18}, we obtained stability with the standard
restrictions on a weight $\sigma$ in a stronger norm. Moreover, the implementation
of the explicit-implicit scheme is much simpler from the computational point
of view -- we must invert a self-adjoint elliptic grid operator.

Considering two-level difference schemes, we have highlighted two main
classes of difference schemes for unsteady convection-diffusion problems.
The first class is based on the use of the simplest schemes with
equal weights for the convective and diffusive transport. The second and 
the most promising class of difference schemes (explicit-implicit schemes) is associated with
the explicit treatment of the convective transport. Here we do not
analyze the complete set of three-level difference schemes. We focus on
the study of explicit-implicit schemes. Using three-level
difference schemes, we can obtain the second-order approximation in time.

To solve numerically the problem (\ref{180}), (\ref{181}),
we employ the three-level explicit-implicit scheme with weights
\begin{equation}\label{211}
\begin{split}
  \frac {y^{n+1} - y^{n-1}}{2\tau} & +
  D (\sigma y^{n+1} + (1 - 2 \sigma) y^n + \sigma y^{n-1}) \\
  & +
  C y^{n} =
  \varphi^n,
  \quad n=1,2,..., N_0-1
\end{split}
\end{equation}
with
\begin{equation}\label{212}
  y^0 = u^0 ,
  \quad y^1 = u^1.
\end{equation}
In (\ref{211}), we put, e.g., $C = C(t^n),~ \varphi^n = \varphi(t^n)$.
To specify the second initial condition ($u^1$ in (\ref{212})) with
the second order, in the simplest case, we involve a two-level scheme, so that
\[
\frac {y^{1} - y^0}{\tau} +
  (C + D) \frac {y^{1} + y^0}{2} = \varphi^0 .
\]
The difference scheme (\ref{211}), (\ref{212}) approximates (\ref{180}), (\ref{181})
with the second order in time.

The explicit-implicit scheme (\ref{211}) is written in the canonical form
\begin{equation}\label{213}
\begin{split}
  B(t^n) \frac {y^{n+1} - y^{n-1}}{2\tau}  & +
  R (t^n) (y^{n+1} - 2 y^n + y^{n-1})  \\ 
  & + A (t^n) y^n = \varphi^n,
  \quad n = 1,2,..., N_0-1
\end{split}
\end{equation}
with
\begin{equation}\label{214}
  B = E,
  \quad R = \sigma D,
  \quad A = C + D.
\end{equation}
To evaluate the difference solution, we introduce the norm associated only with
the diffusive transport, i.e.,
\begin{equation}\label{215}
\begin{split}
 {\cal E}^{n+1} & = \frac 14 (D(y^{n+1}+y^n),y^{n+1} + y^n) \\
  & + \left (\sigma - \frac 14 \right ) (D(y^{n+1} - y^n),y^{n+1} - y^n).
\end{split}
\end{equation}
Stability is established taking into account the subordination of the 
convective transport operator to the diffusive transport operator --- 
we say about the estimate (\ref{185}).

\begin{theorem}\label{t-20}
If $\sigma > 0.25$, then the difference scheme (\ref{211}), (\ref{212})
is $\varrho$-stable with
\begin{equation}\label{216}
  \varrho = 1 + M_2 \frac {4\sigma}{4 \sigma - 1} \tau,
\end{equation}
and the solution satisfies the a priori estimate
\begin{equation}\label{217}
  {\cal E}^{n+1} \le \varrho {\cal E}^{n} + \tau \|\varphi^n\|^2 .
\end{equation}
\end{theorem}
\begin{proof}
For the scheme (\ref{213}), (\ref{214}), we have
\[
\begin{split}
 \frac 1{2\tau} \|w^{n+1}+w^n\|^2 & + {\cal E}^{n+1} =
- (Cy^n,w^{n+1} + w^n) \\
  & +
  (\varphi^n, w^{n+1}+w^n) + {\cal E}^{n},
\end{split}
\]
where
\[
  w^{n+1} =  y^{n+1}-y^n .
\]
For the first two terms in the right-hand side, we obtain
\[
| (Cy^n,w^{n+1} + w^n) | \le \frac {1}{4 \tau} \|w^{n+1} + w^n\|^2 +
\tau \|C y^n\|^2 ,
\]
\[
| (\varphi^n, w^{n+1}+w^n) | \le \frac {1}{4 \tau} \|w^{n+1} + w^n\|^2
+ \tau \|\varphi^n\|^2 .
\]
Thus, in view of (\ref{215}), we arrive at the inequality
\begin{equation}\label{218} 
{\cal E}^{n+1} \le {\cal E}^{n} +
\tau M_2 \|y^n\|^2_D + \tau \|\varphi^n\|^2 .
\end{equation}
Next, we use the estimate
\begin{equation}\label{219}
  \|y^n\|^2_D \le \frac{4\sigma}{4\sigma -1} {\cal E}^n .
\end{equation}
According to (\ref{212}), we get
\[
\begin{split}
 {\cal E}^n & = (Dy^n,y^n) - \tau (Dy^n,y_{\bar t}) +
  \sigma \tau^2 (D y_{\bar t}, y_{\bar t}) \\
  & \ge \|y^n\|^2_D - \tau \|y^n\| \ \|y_{\bar t}\| +
  \sigma \tau^2 \|y_{\bar t}\|^2_D  \\
  & \ge (1 - \beta) \|y^n\|^2_D +
  (\sigma - \frac {1}{4\beta}) \tau^2 \|y_{\bar t}\|^2_D .
\end{split}
\]
For $\sigma > 0.25$, we select the parameter $\beta = 1/(4\sigma)$
and obtain the estimate (\ref{219}). Substitution of (\ref{219}) into
(\ref{218}) yields the levelwise estimate (\ref{216}), (\ref{218}).
\end{proof}

\subsection{Unconditionally stable schemes} 

For convection-diffusion problems with convective transport in the divergent and nondivergent
forms, we have constructed (Theorem~\ref{t-18}) conditionally stable schemes with weights.
Restrictions on a time step (see (\ref{204})) are governed by features of the problem and
do not related, in general, with parameters of discretization in space.
Conditionally stable schemes with weights are developed only for problems with the convective
transport in the skew-symmetric form (Theorem ~\ref{t-17}).

Different nature of convective and diffusive transport as well as reaction processes
appear, in particular, in significantly distinct representative rates of these phenomena.
Such heterogeneity can be taken into account when choosing discretization in time.
The most pronounced occurrence of the heterogeneity of discretization in time is expressed in
explicit-implicit schemes. In this case, for numerical solving the unsteady problem,
a part of the problem operator terms is approximated by explicit relationships, 
whereas the other part is treated implicitly.

Explicit-implicit schemes are widely used for the numerical solution of convection-diffusion problems.
Various variants of inhomogeneous discretization in time are given in \cite{ascher1995implicit}.
One or another explicit approximations are applied to the convective transport operator,
whereas the diffusive transport operator is approximated implicitly.
Thus, the most severe restrictions on a time step due to diffusion are removed. 
In view of the subordination of the convective transport operator  to 
the diffusive transport operator, we have already proved unconditional stability
of the above-considered explicit-implicit schemes for time-dependent convection-diffusion problems.

Similar techniques are used in the analysis of diffusion-reaction problems.
In this case (see, e.g., \cite{ruuth1995implicit}), the diffusive transport is treated
implicitly, whereas for reactions (source terms), explicit approximations are used.
Such explicit approximations demonstrate obvious advantages for problems with nonlinear
terms describing reaction processes.
 
In convection-diffusion-reaction problems, the problem operator may be sign-indefinite.
This means that the system may be nondissipative, i.e., the solution norm for the homogeneous
problem does not decrease during the time evolution. Thus, the exponential growth of the solution 
may be observed, and such behavior of the solution must be reflected at the discrete level.
Unconditionally stable schemes for such problems are constructed in the work \cite{VabVas}. 
They are based on the splitting of the problem operator into two terms, where one of the terms
has explicit approximations in time, whereas the other is approximated implicitly.
Implicit approximations are applied to the part of the problem operator that causes the dissipative
properties of the problem. In the case of the skew-symmetric operator of convective transport,
such a splitting is used for the operator of reaction.

The standard schemes, which are used in computational practice, should be corrected
even for solving dissipative problems. For example, both
the standard fully implicit scheme (backward Euler)
and symmetric scheme (Crank-Nicholson)
does not produce the exact solution for the test problem ($\lambda > 0$):
\[
  \frac{d u}{dt} + \lambda u = 0,
  \quad u(0) = u^0 .  
\]
In \cite{Mickens2002}, there is discussed a modification of standard schemes that is based
on the use of $(\exp(\lambda\tau) - 1) / \lambda$
instead of the original time step $\tau$ in the application to the fully implicit scheme.
More recent results concerned with constructing and employing
such nonstandard discretizations in time can be found, e.g., in \cite{Mickens2005}.
Here we mention new possibilities in designing unconditionally stable schemes for solving
unsteady convection-diffusion problems that involve the introduction of new variables.

Time-dependent convection-diffusion problems with the convective transport in the divergent
(\ref{37})--(\ref{39}) and the nondivergent (\ref{36})--(\ref{38}) forms
may be written as the Cauchy problem for the operator equation (compare with (\ref{41})):
\begin{equation}\label{220}
 \frac{d u}{d t} + {\cal A} u = f(t), \quad {\cal A} =
{\cal C}_0 +  {\cal R} + {\cal D}.
\end{equation}
Here we introduce the reaction operator
\[
   {\cal R} u = r({\bm x},t) u .
\] 
In the case (\ref{37})--(\ref{39}), we have
\[
  r({\bm x},t)  = - \frac{1}{2} \div {\bm v} .
\]
Similarly, for equation (\ref{36})--(\ref{38}), we obtain
\[
  r({\bm x},t)  = \frac{1}{2} \div {\bm v} .
\]
For the reaction operator, we get
\begin{equation}\label{221}
  {\cal R} = {\cal R}^* ,
  \quad \mathsf{m}   {\cal E} \leq {\cal R} \leq \mathcal{M} {\cal E}.
\end{equation}
Using roughened estimates fr the reaction operator, we can put
\[
 \mathsf{m} = - \mathcal{M}_1,
  \quad  \mathcal{M} = \mathcal{M}_1 .
\] 

After discretization in space, from  (\ref{220}), we obtain the equation
\begin{equation}\label{222}
  \frac{d y}{d t} + A y = \varphi(t),
  \quad A = C_0 + R + D, 
  \quad 0 < t \leq T,  
\end{equation}
supplemented by the initial condition (\ref{181})..

For the operator $R$, we have 
\begin{equation}\label{223}
  R y = r({\bm x},t) y,
  \quad  {\bm x} \in \omega . 
\end{equation} 
In this case, we get
\begin{equation}\label{224}
  R = R^*,
  \quad m E \leq R \leq M E .   
\end{equation}
For instance, the convective transport operator of in the nondivergent form seems like this:
\[
  r({\bm x},t)  = - \frac{1}{2} \div_h {\bm v} 
\]
for an appropriate approximation of the divergence operator.
 
To construct unconditionally stable schemes for solving the problem
(\ref{181}), (\ref{222}) without the assumption of non-negativity of the problem operator,
we apply explicit-implicit approximations.
The bottleneck is connected with the reaction operator, and therefore for $m < 0$, 
we split it into two terms:
\begin{equation}\label{225}
  R = R_+ + R_-,
  \quad R_+ = R_+^*,
  \quad R_- = R_-^*,
  \quad 0 \leq R_+ \leq M E,
  \quad m E \leq R_- < 0 .   
\end{equation}
By (\ref{223}), it is sufficient to put
\[
  R_+ y = r_+({\bm x},t) y,
  \quad  R_- y = r_-({\bm x},t) y,
  \quad  {\bm x} \in \omega ,
\] 
where
\[
  r_+ = \max (0, r),
  \quad r = r_+ + r_- .
\] 

Using two-level explicit-implicit schemes, we  may rely only on the first-order accuracy 
with respect to time. Therefore, we focus on the fully implicit approximations of the main
operator terms. We employ the difference scheme
\begin{equation}\label{226}
  \frac {y^{n+1} - y^n}{\tau} +
  (C^n + D + R_+^n) y^{n+1} + R_-^n y^n = 0,
  \quad n = 0,1,...,N_0-1 .
\end{equation}

\begin{theorem}\label{t-21}
The explicit-implicit scheme (\ref{141}), (\ref{223})--(\ref{226}) with $m < 0$
is unconditionally $\varrho$-stable in $H$ for
\begin{equation}\label{227}
  \varrho = 1 - m \tau ,
\end{equation}
and the difference solution satisfies the estimate
\begin{equation}\label{228}
  \|y^{n+1}\| \le \varrho \|y^n\|,
  \quad n = 0,1,...,N_0-1 .  
\end{equation} 
\end{theorem}

\begin{proof}
Multiplying equation (\ref{226}) scalarly in $H$ by $y^{n+1}$,
and taking into account the skew-symmetry of the operator $C_0$,  positive definiteness of
the operator $D$, and relation (\ref{225}), we obtain
\begin{equation}\label{229}
  \|y^{n+1}\|^2 \leq
  (y^{n+1}, y^{n}) - (R_-^n y^{n}, y^{n}) .
\end{equation}
In view of
\[
 (y^{n+1}, y^{n}) \leq \frac{1}{2} ( \|y^{n+1}\|^2 + \|y^{n}\|^2 ),
\] 
\[
  |- (R_-^n y^{n}, y^{n}) | \leq
  m  \|y^{n+1}\| \|y^{n}\|,
\] 
from (\ref{229}), we arrive at the inequality
\[
 \|y^{n+1}\|^2 \leq (1 - 2 m \tau) \|y^{n}\|^2 .
\] 
By virtu of
\[
 (1 - 2 m \tau) \leq (1 - m \tau)^2 ,
\]
this yields the inequality of $\varrho$-stability (\ref{228}) with
$\varrho$ defined by (\ref{227}).
\end{proof}

Among possible generalizations of explicit-implicit scheme (\ref{141}), (\ref{226}),
special attention should be given to schemes of the second-order accuracy 
with respect to time. The symmetric scheme provides an example of such a scheme:
\begin{equation}\label{230}
\begin{split}
  \frac {y^{n+1} - y^{n-1}}{2\tau} & +
  (C^n + D + R_+^n) \frac{y^{n+1} + 2 y^{n} + y^{n+1}}{4} \\
  & + R_-^n y^n = 0,
  \quad n = 1,2,...,N_0-1 ,
\end{split}
\end{equation}
where now $C^n = C(t^n), \ R^n = R(t^n)$.
To start calculations with the second order in time, we put, e.g.,
\[
  \frac {y^{1} - y^0}{\tau} +
  \frac{1}{2} ((C^1 + D + R^1) y^{1} + (C^0 + D + R^0) y^{0} )  = 0, 
\]
Because of this and taking into account the initial condition (\ref{141}),
the difference equation (\ref{230}) is considered for a given $y^{0}$ and $y^{1}$.

In addition to (\ref{230}), special mention should be given to the scheme
\[
\begin{split}
  \frac {3 y^{n+1} - 4 y^{n} + y^{n-1}}{2\tau} & +
  (C^n + D + R_+^n) y^{n+1} \\
  & + R_-^n (2 y^n - y^{n-1}) = 0,
  \quad n = 1,2,...,N_0-1 . 
\end{split}  
\]
Preserving the second-order approximation in time, for this scheme,
the implicity of the main part of the problem operator is expressed more essentially. 

In equation (\ref{222}), for the operator $A$, by (\ref{182}) and (\ref{183}), we have
\[
  A \geq m E +   \frac{1}{M_0} \kappa_1 E .
\]
In our study, we use a more rough estimate 
\begin{equation}\label{231}
  A \geq m E ,
\end{equation}
and consider the most interesting case $m < 0$.

To construct unconditionally stable schemes for the differential problem
(\ref{141}), (\ref{222}), (\ref{224}) under the condition (\ref{231}),
we define a new function $w$:
\begin{equation} \label{232}
  y = \exp(- m t) w.
\end{equation}
Substitution of (\ref{232}) into (\ref{141}), (\ref{222})
for the homogeneous right-hand side gives
the following problem for $w$:
\begin{equation} \label{233}
  \frac{d w}{dt} + \tilde{A} w = 0, \quad \tilde{A} = A - m E, 
  \quad 0 < t \leq T,
\end{equation}
\begin{equation} \label{234}
  w\left(0\right) = u^0.
\end{equation}
For this transformation, the problem operator $\tilde{A}$ is non-negative ($\tilde{A} \geq 0$).

To solve the problem (\ref{233}), (\ref{234}), we apply the two-level scheme with weights:
\begin{equation} \label{235}
  \frac{w^{n+1} - w^n}{\tau} + \tilde{A}^n \left( \sigma w^{n+1} + \left(1 -
  \sigma\right) w^n \right) = 0, \quad t^n \in \omega_\tau,
\end{equation}
\begin{equation}\label{236}
  w^0 = u^0.
\end{equation}
This scheme under the standard constraints $\sigma \geq 0.5$ 
is unconditionally stable (Theorem~\ref{t-17}).

Let us write the difference equation (\ref{235}) for the desired grid function
$y^n$. Taking into account $t^{n+1} = t^n + \tau$, we put
\[
  y^n = \exp(- m t^n) w^n, 
  \quad y^{n+1} = \exp(- m t^n) \exp(- m \tau) w^{n+1} .
\]
Because of this, from (\ref{235}), (\ref{236}), we obtain the following difference scheme for $y^n$:
\begin{equation} \label{237}
  \frac{\exp(m \tau) y^{n+1} - y^n}{\tau} + \left( A - m E\right) \left(
  \sigma \exp(m \tau) y^{n+1} + \left(1 - \sigma\right) y^n \right) = 0, 
\end{equation}
\begin{equation} \label{238}
  y^0 = u^0.
\end{equation}

In contrast to the nonstandard schemes discussed in \cite{Mickens2002,Mickens2005},
a positive effect is achieved not only through the use of new approximations in time,
but also by correcting the problem operator.

\begin{theorem}\label{t-22}
The difference scheme (\ref{237}), (\ref{238}) for $\sigma \geq 0.5$ 
is unconditionally $\varrho$-stable in the $H$ with
\begin{equation} \label{239}
  \varrho = \exp(- m \tau) ,
\end{equation}
and the solution satisfies the a priori estimate
\begin{equation} \label{240}
  \|y^{n+1}\| \leq \varrho \|y^n\| .
\end{equation}
\end{theorem}
\begin{proof}
The above proofs were based on the transition to the problem with a non-negative operator and
the use of the previous Theorem~\ref{t-17}.
It is possible to conduct a direct proof of stability for the scheme (\ref{237}), (\ref{238}).
Rewrite the scheme under consideration in the form 
\begin{equation} \label{241}
  \frac{\exp(m \tau) y^{n+1} - y^n}{\tau} + \tilde{A} p^{n+1}= 0, \quad t^n \in
  \omega_\tau,
\end{equation}
where
\[
\begin{split}
  p^{n+1} &= \sigma \exp(m \tau) y^{n+1} + \left(1 - \sigma\right) y^n \\
  & =
  \tau \left( \sigma - \frac{1}{2} \right) r^{n+1} +
  \frac{1}{2} \left (\exp(m \tau) y^{n+1} - y^n \right ),
\end{split}
\]
\[
  r^{n+1} = \frac{\exp(m \tau) y^{n+1} - y^n}{\tau}.
\]
Multiplying equation (\ref{241}) scalarly by $p^{n+1}$, we obtain
\[
\begin{split}
\tau \left( \sigma - \frac{1}{2} \right) \left(r^{n+1}, r^{n+1} \right)
 & + \tilde{A} \left(p^{n+1}, p^{n+1}\right)  \\
 & + \frac{1}{2 \tau}\left( \left(\exp(m \tau) y^{n+1}, \exp(m \tau) y^{n+1}\right) 
 - \left(y^n, y^n\right) \right)   = 0 .
\end{split} 
\]
From this equation, under the conditions $\sigma \geq 0.5$ and $\tilde{A} \geq 0$,
it follows that the stability estimate (\ref{239}), (\ref{240}) holds.
\end{proof}

It is important to note that, in contrast to the conventional scheme with weights (see Theorem~\ref{t-18}), here stability is obtained with no restriction on a time step. The value of $\varrho$ defined by (\ref{239})
is fully consistent with the corresponding constant for the solution
of the differential problem. The transition to a new time level involves the solution of the grid problem
\begin{equation}\label{242}
  (E +  \sigma \tau (A - m E ))  y^{n+1} =  \chi^n.
\end{equation}
The equation (\ref{242}) is a system of linear algebraic
equations with a positive definite and non-self-adjoint matrix;
it can be solved using standard iterative methods.

\section{Stability in Banach spaces}

The main results on stability of difference schemes for the unsteady
convection-diffusion equation were obtained above considering the problem in
Hilbert spaces of grid functions. Here we study difference schemes in Banach spaces,
where stability of difference schemes is established in the uniform and integral norms.

In our study we can employ the maximum principle for difference schemes
as it was done in investigating monotone approximations.
The second and more promising approach presented below is to use
the concept of the logarithmic norm. In this section, monotone schemes of the second-order accuracy in space
are constructed for the time-dependent convection-diffusion.

\subsection{One-dimensional problems} 

To simplify the material presented here, we start with the 1D convection-diffusion problems.
Consider the time-dependent convection-diffusion equation with convective terms in the
nondivergent form:
\begin{equation}\label{243}
  \frac{\partial u}{\partial t} +
   v (x,t) \frac{\partial u}{\partial x} -
  \frac{\partial }{\partial x} \left ( k(x) \frac{\partial u}{\partial x} \right ) =
  f(x,t) 
\end{equation} 
for
\[
   \quad 0 < x < l,
   \quad 0 < t \leq T . 
\]
This equation is supplemented with homogeneous Dirichlet boundary conditions:
\begin{equation}\label{244}
  u(0, t) = 0,
  \quad u(l,t) = 0 , 
   \quad 0 < t \leq T . 
\end{equation}
In addition, the initial condition is given:
\begin{equation}\label{245}
  u(x,0) = u^0(x),
  \quad 0 < x < l .
\end{equation} 

The second important example is the unsteady equation
of convection-diffusion in the divergent form:
\begin{equation}\label{246}
  \frac{\partial u}{\partial t} +
  \frac{\partial }{\partial x} (v (x,t) u) -
  \frac{\partial }{\partial x} \left ( k(x) \frac{\partial u}{\partial x} \right ) =
  f(x,t) .
\end{equation} 

Consider the set of functions $u(x,t)$ satisfying the boundary conditions (\ref{244}). 
The transient problem of convection-diffusion is written in the form of 
the operator-differential equation
\begin{equation}\label{247}
\frac {du}{dt} + {\cal A} u = f(t),
\quad {\cal A} = {\cal A}(t) = {\cal C}(t) + {\cal D},
\end{equation}
where ${\cal C}(t)$ denotes the convective transport operator, and
${\cal D}$ stands for the operator of diffusive transport.
The Cauchy problem for the evolutionary equation (\ref{247})
is supplemented with the initial condition
\begin{equation}\label{248}
  u(0) = u^0 .
\end{equation} 

We recall some a priori estimates for the convection-diffusion problems (\ref{243})--(\ref{245}) and
(\ref{244})--(\ref{246}), which are derived from the maximum principle.
The corresponding a priori estimates are derived in the spaces
$L_\infty(0,l)$ and $L_1(0,l)$, where the norms are, respectively,
\[
  \|v\|_\infty = \max_{0<x<l} | v(x) |,
  \quad \|v\|_1 = \int_{0}^{l}  | v(x) | d x .
\]
The solution of the time-dependent convection-diffusion problem (\ref{243})--(\ref{245})
(the convective transport in the nondivergent form)
satisfies the a priori estimate in $ L_\infty(0,l)$:
\begin{equation}\label{249}
  \|u(x,t)\|_\infty \le
  \|u^0(x)\|_\infty +
  \int_0^t \|f(x,\theta)\|_\infty d\theta .
\end{equation}
We present also the estimate for the convection-diffusion equation with convective
terms in the divergent form. The solution of the problem (\ref{244})--(\ref{246})
satisfies the a priori estimate in $L_1(0,l)$:
\begin{equation}\label{250}
  \|u(x,t)\|_1 \le
  \|u^0(x)\|_1 +
  \int_0^t \|f(x,\theta)\|_1 d\theta .
\end{equation}
The a priori estimates (\ref{249}), (\ref{250}) serve us as a guide in considering discrete problems.

\subsection{Stability of two-level schemes} 

Let us obtain sufficient conditions for the stability of two-level difference schemes
for the Cauchy problem for a system of ODEs.
Further, these general conditions will be applied to particular cases of model convection-diffusion equations
with the convective terms in the nondivergent and divergent forms.

Consider a system of linear ODEs of first order:
\begin{equation}\label{251}
  \frac{dw_i}{dt} + \sum_{j=1}^{m} a_{ij}(t) w_j = \phi_i(t),
  \quad i = 1,2,\ldots,m .
\end{equation}
Assume that $w = w(t) = \{w_1,w_2,\ldots,w_m\},~  A =  [a_{ij}]$, then we can write
(\ref{251}) in matrix (operator) form as
\begin{equation}\label{252}
  \frac{dw}{dt} + A(t) w = \phi(t) .
\end{equation}
We will construct difference schemes for numerical solving 
the Cauchy problem (\ref{252}) for $t > 0$ and the initial condition
\begin{equation}\label{253}
   w(0) = u^0.
\end{equation} 

We will investigate the stability of the difference solution of the problem (\ref{252}), (\ref{253})
in $L_\infty$ and $L_1$. 
For a norm of a vector and a norm of a matrix, consistent with it in $L_\infty$, we have
\begin{equation}\label{254}
  \|w\|_\infty = \max_{1\le i \le m} |w_i|,
  \quad \|A\|_\infty = \max_{1\le i \le m} \sum_{j=1}^{m} |a_{ij}|.
\end{equation} 
Similarly, in $L_1$, we obtain
\begin{equation}\label{255}
  \|w\|_1 = \sum_{i=1}^{m} |w_i|,
  \quad \|A\|_1 = \max_{1\le j \le m} \sum_{i=1}^{m} |a_{ij}|.
\end{equation} 

The problem (\ref{252}), (\ref{253})  will be considered under the following constraints. 
Assume that the diagonal elements of the matrix $A$ are non-negative,
and there is row-wise or column-wise diagonal dominance, i.e., we have
\begin{equation}\label{256}
  a_{ii} \ge \sum_{i\not=j=1}^{m} |a_{ij}|,
  \quad i = 1,2,\ldots,m
\end{equation}
(weak diagonal dominance by rows) or
\begin{equation}\label{257}
  a_{ii} \ge \sum_{i\not=j=1}^{m} |a_{ji}|,
  \quad i = 1,2,\ldots,m
\end{equation}
(weak diagonal dominance by columns).

The logarithmic norm of the matrix $A$ is defined \cite{DekkerVerwer1984,HairerNorsettWanner1987}
by the number
\[
  \mu[A] = \lim_{\delta \to 0+} \frac {\|E + \delta A\| - 1}{\delta} .
\]
For the logarithmic norm of a matrix in $L_\infty$ (consistent with (\ref{254})) and in
$L_1$ (consistent with (\ref{255})), we have the expressions
\[
  \mu_\infty [A] = \max_{1\le i \le m} \Big ( a_{ii} +
  \sum_{i\not=j=1}^{m} |a_{ij}| \Big ),
\]
\[
  \mu_1 [A] = \max_{1\le j \le m} \Big ( a_{jj} +
  \sum_{j\not=i=1}^{m} |a_{ij}| \Big ) .
\]
In view of the restrictions (\ref{256}), (\ref{257}), we have that the logarithmic norm
of the matrix $-A$ in the Cauchy problem (\ref{252}), (\ref{253}) satisfies the inequality
\begin{equation}\label{258}
  \mu[-A] \le 0
\end{equation}
in the corresponding space (in $L_\infty$ for (\ref{256}) and in $L_1$ for (\ref{257})).

Among the properties of the logarithmic norm (see \cite{DekkerVerwer1984,DesoerHaneda1972}), we highlight
the following:
\begin{enumerate}
\item
$\mu [cA] = c \mu[A], \quad c = \const \ge 0;$
\item
$\mu [cE + A] = c + \mu[A], \quad c = \const;$
\item
$\|Aw\| \ge \max \{ - \mu[-A],~ - \mu[A] \}~ \|w\| $.
\end{enumerate}
The emphasis is placed on the property 3, which allows to get easily the lower
bound of the norm $Aw$. This bound can be combined with the standard upper
bound of $Aw$: $\|Aw\| \le \|A\|~ \|w\|$.

Let us study the stability of difference schemes for the problem (\ref{252}), (\ref{253}). 
We denote the approximate solution at the time level $t^n = n \tau$
(where $\tau$ is a time step) as $y^n$, and write the two-level difference scheme with weights
\begin{equation}\label{259}
  \frac {y^{n+1} - y^n} {\tau} +
  A (\sigma y^{n+1} + (1-\sigma) y^n) = \varphi^n,
\end{equation}
where, e.g., $A = A(\sigma t^{n+1} + (1- \sigma) t^n)$, with the initial data
\begin{equation}\label{260}
  y^0 = u^0 .
\end{equation}
A sufficient condition for stability of the scheme (\ref{259}), (\ref{260}) is formulated
as the following statement.

\begin{theorem}\label{t-23}
Assume that in the Cauchy problem (\ref{252}), (\ref{253}) the
matrix $A$ satisfies the restriction (\ref{256}) (or (\ref{257})). Then the difference
scheme with weights (\ref{259}), (\ref{260}) is unconditionally stable for $\sigma =1$,
and it is conditionally stable for $\sigma < 1$ in $L_\infty$ (in $L_1$) if and only if
\begin{equation}\label{261}
  \tau \le \frac {1}{1-\sigma} \Big (\max_{1\le i\le m} a_{ii} \Big )^{-1} .
\end{equation}
In this case, the difference solution satisfies the a priori estimate
\begin{equation}\label{262}
  \|y^{n+1}\| \le \|u^0\| + \sum_{k=0}^{n} \tau \|\varphi^k\| .
\end{equation} 
\end{theorem}

\begin{proof}
From (\ref{259}), it follows that
\[
  (E +  \sigma \tau A) y^{n+1} =
  (E - (1- \sigma) \tau A )y^n + \tau \varphi^n ,
\]
and therefore
\begin{equation}\label{263}
  \|(E +  \sigma \tau A) y^{n+1}\| \le
  \|(E - (1- \sigma) \tau A )y^n\| + \tau \|\varphi^n\| .
\end{equation}
For the left-hand side of (\ref{263}), by the above-mentioned properties of the
logarithmic norm and in view of (\ref{258}), we have
\[
\begin{split}
  \|(E +  \sigma \tau A) y^{n+1}\|  &\ge
  - \mu [- E -  \sigma \tau A] \ \|y^{n+1}\| \\
  & = (1 + \sigma \mu [-A]) \|y^{n+1}\| \ge  \|y^{n+1}\| .
\end{split}  
\]
For the first term in the right-hand side of (\ref{263}), we obtain
\[
  \|(E - (1- \sigma) \tau A )y^n\| \le
  \|E - (1- \sigma) \tau A \| \ \|y^n\| .
\]
We investigate this estimate in more detail for $L_\infty$. The case $L_1$
is studied in a similar manner. Considering (\ref{254}) and taking into account the
condition of diagonal dominance (\ref{256})), we have
\[
\begin{split}
  \|E - (1- \sigma) \tau A \|  & =
  \max_{1\le i \le m} \Big | 1 - (1-\sigma) \tau
  \Big (a_{ii} + \sum_{i\not=j=1}^{m} a_{ij} \Big ) \Big | \\
  & \le
  \max_{1\le i \le m} \Big (|1 - (1-\sigma) \tau a_{ii}| +
  (1-\sigma) \tau \sum_{i\not=j=1}^{m} |a_{ij}| \Big) \\
  & \le
  \max_{1\le i \le m} (|1 - (1-\sigma) \tau a_{ii}| +
  (1-\sigma) \tau a_{ii}) \le 1
\end{split}   
\]
with $0\le \sigma \le 1$ and under the restriction (\ref{261}) on the time step.

The substitution into (\ref{263}) yields the inequality
\[
  \|y^{n+1}\| \le \|y^n\| + \tau \|\varphi^n\| ,
\]
which immediately implies the desired estimate (\ref{262}) for stability with
respect to the right-hand side and the initial data.
\end{proof}

The above estimates for stability (\ref{262}) in $L_\infty$ and $L_1$
are directly associated with the monotonicity of the difference solution
of the  problem \ref{259}), (\ref{260}) under the assumption that the off-diagonal
elements of the matrix $A$ are non-positive. Let us prove the following statement.

\begin{theorem}\label{t-24}
Assume that in the schemes (\ref{259}), (\ref{260}), the conditions of diagonal
dominance (\ref{256}) (or (\ref{257})) are fulfilled for
\begin{equation}\label{264}
  a_{ij} \leq  0, 
  \quad i \neq j,
  \quad i,j = 1,2,\ldots,m 
\end{equation} 
and let
\[
  u^0  \ge 0,
  \quad \varphi^n \ge 0,
  \quad n = 0,1,\ldots,
\]
then
\[
  y^{n+1} \ge 0,
  \quad n = 1,2,\ldots,
\]
for any $\tau >0$ if $\sigma =1$, and if $0 \le \sigma < 1$,
this is true under the constraints on a time step (\ref{261}) .
\end{theorem}
\begin{proof}
For the transition from the current time level to the next one, we have
\begin{equation}\label{265}
  y^{n+1} + \sigma \tau A y^{n+1} = g^n,
  \quad n = 0,1,\ldots ,
\end{equation} 
where
\begin{equation}\label{266}
  g^n = y^n - (1-\sigma) \tau A y^n + \tau \varphi^n .
\end{equation}
Suppose that $y^{n} \ge 0$ (for $n=0$ this is true from the assumptions of the theorem).
We show that from this it follows also the non-negativity of $y^{n+1}$ ($y^{n+1} \ge 0$).

We prove that under the assumptions of the diagonal dominance
(\ref{256}) (or (\ref{257}))
and under the restrictions on a time step (\ref{261}),
for a non-negative $y^{n}$ and $\varphi^n$, we get $g^n \ge 0$. .
In view of (\ref{266}), we obtain
\[
\begin{split}
  g^n_i &= (1 - (1-\sigma)\tau a_{ii}) y^n_i - (1-\sigma)\tau \sum_{j\neq i, j = 1 }^{m} a_{ij} y^n_j  \\
  & \geq  (1 - (1-\sigma)\tau a_{ii}) y^n_i  \geq 0.
\end{split} 
\] 

In the conditions of the theorem, the matrix of the system of linear algebraic equations
(\ref{265}) is an M-matrix, i.e., we have: strong diagonal dominance,
positive diagonal elements, and non-positive off-diagonal elements of the matrix. 
Because of this, from $g^n \ge 0$, it follows that $y^{n+1} \ge 0$.
\end{proof}

Apply the derived results to studying stability and monotonicity of
difference schemes for time-dependent problems of convection-diffusion 
in the nondivergent and divergent forms.

\subsection{Difference schemes for convection-diffusion equations} 

For simplicity, we restrict ourselves to uniform grids. On the interval $[0,l]$, we introduce
a grid
\[
   \bar{\omega} \equiv \omega \cup  \partial \omega =   \{x~ |~ x = x_i  = ih,
   \quad i = 0,1,\ldots,N,  \quad Nh = l\} ,
\]
where $\omega$ is the set of interior nodes:
\[
   \omega = \{x~ |~ x = x_i  = ih,
   \quad i = 1,2,\ldots,N-1,  \quad Nh = l\} .
\]

After discretization in space of the model convection-diffusion problems with homogeneous
boundary conditions (\ref{243})--(\ref{245}) and (\ref{244})--(\ref{246}),
we arrive at the problem (\ref{252}), (\ref{253}), where $m = N-1$
and the approximate solution $w_i(t) = w(x,t), \ x \in \omega$.
The difference diffusion operator is specified, e.g., as follows:
\begin{equation}\label{267}
\begin{split}
  D w = & - \frac{1}{h^2} k (x+0.5 h) (w(x+h, t) - w(x,t)) \\
  &+ \frac{1}{h^2} k (x-0.5 h) (w(x, t) - w(x-h,t)),
  \quad x \in  \omega    
\end{split}
\end{equation} 
with
\begin{equation}\label{268}
  w(x,t) = 0,
  \quad x \in \partial\omega .  
\end{equation} 

Approximation of convective transport is conducted in such a way
that $v(x,t)$ are defined at the half-integer grid points $\bar{\omega}$.
For operators of convective transport in the nondivergent form (equation (\ref{243})),
in view of (\ref{244}), we put
\begin{equation}\label{269}
\begin{split}
  C w & =  \frac{1}{2 h} v (x+0.5 h,t) (w(x+h, t) - w(x,t)) \\
  &+ \frac{1}{2h} v (x-0.5 h,t) (w(x, t) - w(x-h,t)),
  \quad x \in  \omega .  
\end{split} 
\end{equation}
A similar approximation of the second order with respect to $h$ for
the convective transport operator in the nondivergent form (equation (\ref{246})) leads to
\begin{equation}\label{270}
\begin{split}
  C w & =  \frac{1}{2 h} v (x+0.5 h,t) (w(x+h, t) + w(x,t)) \\
  &- \frac{1}{2h} v (x-0.5 h,t) (w(x, t) + w(x-h,t)),
  \quad x \in  \omega .  
\end{split} 
\end{equation}   

Let us formulate the condition for stability and monotonicity
of the schemes with weights (\ref{259}), (\ref{260}) attributed to the problem
(\ref{252}), (\ref{253}), where
\begin{equation}\label{271}
  A = C + D 
\end{equation} 
and $D, C$ are specified according to (\ref{267})--(\ref{269}) or
(\ref{267}), (\ref{268}), (\ref{270}). 

\begin{theorem}\label{t-25}
The difference scheme (\ref{259}), (\ref{260})
with (\ref{267})--(\ref{269}), (\ref{271})
(or (\ref{267}), (\ref{268}), (\ref{270}), (\ref{271}))
is monotone, and the difference solution
satisfies the a priori estimate (\ref{262}) in $L_\infty$ (or in $L_1$)
under the restriction
\begin{equation}\label{272}
  \frac{h |v(x\pm 0.5h,t)| }{k(x\pm 0.5h)} \leq 2,
  \quad x \in \omega   
\end{equation}
for any $\tau >0$ if $\sigma =1$, and if $0 \le \sigma < 1$,
then this is true under the constraint on a time step
\begin{equation}\label{273}
  \tau \le \frac {1}{(1-\sigma)\gamma} ,
\end{equation} 
with
\[
  \gamma = \max_{x \in \omega} \Big (
  \frac{1}{h^2} (k (x+0.5 h) + k (x-0.5 h))
  - \frac{1}{2 h} (v (x+0.5 h,t) - v (x-0.5 h,t)) \Big )   
\] 
for (\ref{270}), and with
\[
  \gamma = \max_{x \in \omega} \Big (
  \frac{1}{h^2} (k (x+0.5 h) + k (x-0.5 h))
  + \frac{1}{2 h} (v (x+0.5 h,t) - v (x-0.5 h,t)) \Big )   
\] 
in the case (\ref{271}).
\end{theorem}
\begin{proof}
Consider the case of the convection-diffusion equation
(\ref{243})--(\ref{245}) (approximations (\ref{267})--(\ref{269}), (\ref{271})) in detail.
The problem (\ref{244})--(\ref{246}) (approximations (\ref{267}), (\ref{268}), (\ref{270}))
are investigated is a similar way.

To apply Theorems~\ref{t-23} and \ref{t-24},
write explicitly the elements of $A$.
For (\ref{267})--(\ref{269}), (\ref{271})), we have
\[
  a_{ii} = \frac{1}{h^2}(k_{i+1/2} + k_{i-1/2})
  - \frac{1}{2h} v_{i+1/2} + \frac{1}{2h} v_{i-1/2} ,
\] 
\[
  a_{i,i-1} = - \frac{1}{h^2} k_{i-1/2}
  - \frac{1}{2h} v_{i-1/2} ,
\] 
\[
  a_{i,i+1} = - \frac{1}{h^2}k_{i+1/2} 
  + \frac{1}{2h} v_{i+1/2} ,
\] 
where $k_{i\pm 1/2} = k(x\pm 0.5h), \ x \in \omega$. 

The condition of nonpositivity of off-diagonal elements
(\ref{264}) holds for
\begin{equation}\label{274}
  \frac{1}{h^2} k_{i-1/2} + \frac{1}{2h} v_{i-1/2} \geq 0,
  \quad   \frac{1}{h^2}k_{i+1/2} - \frac{1}{2h} v_{i+1/2} \geq 0 .
\end{equation}
In this case, diagonal dominance is assured.
A spatial computational grid with the step from the 
conditions (\ref{272}) satisfies the inequalities (\ref{274}).
Restrictions on a time step (\ref{261}) are reduced to the particular condition (\ref{273}).
Thus, the conditions of Theorems~\ref{t-23} and \ref{t-24} hold.
This provides the stability and monotonicity
of the difference solution of the convection-diffusion problem
Under the above restrictions on the time step.
\end{proof}

To overcome restrictions on the spatial grid (\ref{272}), we apply
upwind approximations for the convective terms.
We introduce notation
\[
v(x,t) = v^+(x,t) + v^-(x,t),
\]
\[
v^+(x,t) = \frac 12 (v(x,t) + |v(x,t)|) \ge 0,
\]
\[
v^-(x,t) = \frac 12 (v(x,t) - |v(x,t)|) \le 0.
\]
Instead (\ref{269}), we put
\begin{equation}\label{275}
\begin{split}
  C w & =  \frac{1}{h} v^- (x+0.5 h,t) (w(x+h, t) - w(x,t)) \\
  &+ \frac{1}{h} v^+ (x-0.5 h,t) (w(x, t) - w(x-h,t)) .
\end{split} 
\end{equation}
For the convective transport in the divergent form, we get
\begin{equation}\label{276}
\begin{split}
  C w & =  \frac{1}{h} ( v^- (x+0.5 h,t) w(x+h, t) - v^- (x-0.5 h,t)w(x,t)) \\
  &+ \frac{1}{h} (v^+ (x+0.5 h,t) w(x, t) - v^+ (x-0.5 h,t)w(x-h,t)) .
\end{split} 
\end{equation}  

\begin{theorem}\label{t-26}
The difference scheme (\ref{259}), (\ref{260})
with (\ref{267}), (\ref{269}), (\ref{271}), (\ref{275})
(or (\ref{267}), (\ref{268}), (\ref{271}), (\ref{276}))
is monotone, and the difference solution
satisfies the a priori estimate (\ref{262}) in $L_\infty$ (or in $L_1$)
for any $\tau >0$ if $\sigma =1$, and  if $0 \le \sigma < 1$,
then this is true under the constraints on a time step (\ref{273})
with
\[
  \gamma = \max_{x \in \omega} \Big (
  \frac{1}{h^2} (k (x+0.5 h) + k (x-0.5 h))
  - \frac{1}{h} (v^- (x+0.5 h,t) - v^+ (x-0.5 h,t)) \Big )   
\] 
for (\ref{275}), and with
\[
  \gamma = \max_{x \in \omega} \Big (
  \frac{1}{h^2} (k (x+0.5 h) + k (x-0.5 h))
  + \frac{1}{h} (v^+ (x+0.5 h,t) - v^- (x-0.5 h,t)) \Big )   
\] 
in the case (\ref{276}).
\end{theorem}

In particular, the fully implicit scheme ($\sigma =1$) is
unconditionally stable and monotone. The principal shortcomings
of the above schemes are connected with the upwind approximations for convective
terms (\ref{275}), (\ref{276})) --- these schemes indicate the first-order approximation in space.
Schemes on the basis of the central difference approximations
(\ref{269}), (\ref{270})) are more accurate --- they have the second-order spatial approximation.

\subsection{Exponential schemes} 

It is convenient to construct monotone schemes by means of transforming
the original convection-diffusion equation, i.e., by eliminating the convective terns.
The equation (\ref{243}) may be written as
\begin{equation}\label{277}
  \frac{\partial u}{\partial t} 
   - \frac{1}{\chi (x,t)} 
  \frac{\partial }{\partial x} \left ( k(x) \chi (x,t) \frac{\partial u}{\partial x} \right ) =
  f(x,t) ,
\end{equation} 
where
\begin{equation}\label{278}
  \chi (x,t) = \exp \left  ( - \int\limits_0^x \frac{v(s,t)}{k(s)} ds \right ) .
\end{equation}
The equation (\ref{246}) is reduced to
\begin{equation}\label{279}
  \frac{\partial u}{\partial t} 
   -  \frac{\partial }{\partial x} \left (\frac{k(x)}{\chi (x,t)}   
    \frac{\partial (\chi (x,t) u) }{\partial x} \right ) =
  f(x,t) .
\end{equation}
Further, we can design discretizations in space, i.e., exponential schemes \cite{allsou,schgum}.
 
Similarly to (\ref{267}), for the grid functions satisfying (\ref{268}), 
it is possible to put in equation (\ref{277}):
\begin{equation}\label{280}
\begin{split}
  A w &=  - \frac{1}{h^2 \chi (x,t)} k (x+0.5 h) \chi (x+0.5 h,t) (w(x+h, t) - w(x,t)) \\
  &+ \frac{1}{h^2 \chi (x,t)} k (x-0.5 h) \chi (x-0.5 h,t)(w(x, t) - w(x-h,t)) ,
\end{split}
\end{equation} 
where
\[
  \chi (x-0.5h,t) =
  \exp \left (- \int\limits_0^{x-0.5h} \frac {v(s,t)}{k(s)} ds \right ) .
\]
Taking into account that
\[
  \chi (x-0.5h,t) =
  \chi (x)
  \exp \left (- \int\limits_x^{x-0.5h} \frac {v(s)}{k(s)}  ds \right ) ,
\]
with a precision of $O(h^2)$ we put
\[
  \chi (x-0.5h,t) = \chi (x) \exp ( \theta (x,t) h) 
\]
with notation
\[
  \theta (x,t) = \frac {v(x,t)}{2 k(x)} .
\]
Therefore, instead of (\ref{280}), we can use the following approximation:
\begin{equation}\label{281}
\begin{split}
  A w =  &- \frac{1}{h^2} k (x+0.5 h) \exp ( \theta (x,t) h) (w(x+h, t) - w(x,t)) \\
  &+ \frac{1}{h^2} k (x-0.5 h) \exp ( - \theta (x,t) h)(w(x, t) - w(x-h,t)) .
\end{split}
\end{equation} 

For equation (\ref{279}), similarly to (\ref{280}), we put
\[
\begin{split}
  A w =  &- \frac{k (x+0.5 h)}{h^2 \chi (x+0.5 h,t)} 
  ( \chi (x+h,t) w(x+h, t) - \chi (x,t)w(x,t)) \\
  &+ \frac{k (x-0.5 h)}{h^2 \chi (x-0.5 h,t)}  
  (\chi (x,t)w(x, t) - \chi (x-h,t)w(x-h,t)) .
\end{split}
\]
Simplifying this expression, we obtain
\begin{equation}\label{282}
\begin{split}
  A w =  &- \frac{1}{h^2} k (x+0.5 h) 
  \exp ( - \theta (x+ h,t) h) w(x+h, t) \\
  & + \frac{1}{h^2} k (x+0.5 h) \exp (\theta (x,t) h)w(x,t) \\
  &+ \frac{1}{h^2} k (x-0.5 h) 
  \exp ( - \theta (x,t) h) w(x, t) \\
  & - \frac{1}{h^2} k (x-0.5 h) \exp (\theta (x-h,t) h)w(x-h,t) .
\end{split}
\end{equation} 

Using the above-introduced approximations for the convection-diffusion operator,
we can construct monotone schemes. The primary statement is formulated as follows.

\begin{theorem}\label{t-27}
If on the set of grid functions (\ref{268})
the operator $A$ is defined according to (\ref{281}) (or (\ref{282})),
then the difference scheme (\ref{259}), (\ref{260}),
is monotone, and the difference solution
satisfies the a priori estimate (\ref{262}) in the $L_\infty$ (or in $L_1$)
for any $\tau >0$ if $\sigma =1$, and if $0 \le \sigma < 1$
then this is true under the constraints on a time step (\ref{273})
with
\[
  \gamma = \max_{x \in \omega} 
  \frac{1}{h^2} (k (x+0.5 h) \exp ( \theta (x,t) h) +  k (x-0.5 h) \exp ( - \theta (x,t) h) ) .
\] 
\end{theorem}
\begin{proof}
In the case of (\ref{281}, for the matrix elements, we have
\[
  a_{ii} = \frac{1}{h^2}(k_{i+1/2} \exp(\theta_i)  + k_{i-1/2} \exp(-\theta_i)) ,
\] 
\[
  a_{i,i-1} = - \frac{1}{h^2} k_{i-1/2} \exp(-\theta_i) ,
\] 
\[
  a_{i,i+1} = - \frac{1}{h^2}k_{i+1/2} \exp(\theta_i) .
\]
Checking diagonal dominance by rows and the non-negativity of the off-diagonal
elements is evident.
 
In the case (\ref{282}), we obtain
\[
  a_{ii} = \frac{1}{h^2}(k_{i+1/2} \exp(\theta_i)  + k_{i-1/2} \exp(-\theta_i)) ,
\] 
\[
  a_{i,i-1} = - \frac{1}{h^2} k_{i-1/2} \exp(\theta_{i-1}) ,
\] 
\[
  a_{i,i+1} = - \frac{1}{h^2}k_{i+1/2} \exp(-\theta_{i+1}) .
\]
In view of the non-negativity of the off-diagonal elements,
the condition of diagonal dominance by columns (\ref{257}) takes the form
\[
  a_{ii} \geq  - a_{i-1,i} - a_{i+1,i} ,
\] 
and it is obviously true.
\end{proof}

Thus, the conditions for stability and monotonicity are the same
as for schemes with the upwind approximations of convective terms
(Theorem~\ref{t-26}).
However, discretization in space is of second order 
as for schemes with the central-difference approximations
(Theorem~\ref{t-25}).
Some complications in evaluating coefficients of the difference operator
leads to a slight increasing of the computational costs.

\subsection{Multidimensional problems} 

Possibilities of constructing  second-order monotone schemes
for time-dependent equations of convection-diffusion are examined
on the model 2D problems (\ref{36})--(\ref{38}) and (\ref{37})--(\ref{39})
in the rectangle $\Omega$.

The convection-diffusion operators in multidimensional problems are represented as
the sum of the 1D convection-diffusion operators.
Because of this, in constructing monotone schemes for multidimensional problems, we can
apply the above approximations designed for the 1D operators of convection-diffusion.

Similarly to (\ref{277}), (\ref{278}), rewrite equation (\ref{36}) as
\begin{equation}\label{283}
   \frac{\partial u}{\partial t} -
   \sum_{\alpha =1}^{2} 
   \frac{1}{\chi_\alpha  ({\bm x},t)} 
   \frac{\partial }{\partial x_\alpha } \left ( k({\bm x}) \chi_\alpha ({\bm x},t) \frac{\partial u}{\partial x_\alpha } \right ) =
  f({\bm x},t) ,
\end{equation} 
where now
\begin{equation}\label{284}
\begin{split}
  &\chi_1  ({\bm x},t) = \exp \left  ( - \int\limits_0^{x_1} \frac{v_1(s,x_2,t)}{k(s,x_2)} ds \right ) , \\
  &\chi_2  ({\bm x},t) = \exp \left  ( - \int\limits_0^{x_2} \frac{v_2(x_1,s,t)}{k(x_1,s)} ds \right ) .  
\end{split}
\end{equation} 
A similar transformation for (\ref{39}) yields
\begin{equation}\label{285}
   \frac{\partial u}{\partial t} -
   \sum_{\alpha =1}^{2} 
   \frac{\partial }{\partial x_\alpha } \left ( \frac{k({\bm x})}{\chi_\alpha  ({\bm x},t)}  
   \frac{\partial (\chi_\alpha ({\bm x},t) u)}{\partial x_\alpha } \right ) =
  f({\bm x},t) .
\end{equation} 

For simplicity, we use a uniform grid in each spatial direction.
For grids in separate directions
$x_\alpha,~ \alpha =1,2$, we use notation introduced above:
\[
   \bar{\omega} \equiv \omega \cup \partial \omega = 
   \bar{\omega}_1 \times \bar{\omega}_2, 
   \quad \omega = \omega_1 \times \omega_2 .
\]

After discretization in space of the boundary value problems
(\ref{37}), (\ref{38}), (\ref{283} and
(\ref{37}), (\ref{38}), (\ref{284}), 
we arrive at the problem (\ref{252}), (\ref{253}), 
where
\begin{equation}\label{286}
  A = A_1 + A_2 ,
\end{equation} 
and $A_\alpha, \ \alpha =1,2$ are 1D grid operators
of convection-diffusion.
On the set of grid functions such that
\begin{equation}\label{287}
  w({\bm x},t) = 0,
  \quad {\bm x} \in \partial\omega  ,
\end{equation}
for equation (\ref{283}), similarly to (\ref{281}), we put
\begin{equation}\label{288}
\begin{split}
  A_1 w =  &- \frac{1}{h_1^2} k (x_1+0.5 h_1,x_2) \exp ( \theta ({\bm x},t) h_1) w(x_1+h_1,x_2, t)  \\
  &+ \frac{1}{h_1^2} k (x_1+0.5 h_1,x_2) \exp ( \theta ({\bm x},t) h_1) w({\bm x},t) \\
  &+ \frac{1}{h_1^2} k (x_1-0.5 h_1,x_2) \exp ( - \theta ({\bm x},t) h_1)w({\bm x}, t) \\
  &- \frac{1}{h_1^2} k (x_1-0.5 h_1,x_2) \exp ( - \theta ({\bm x},t) h_1) w(x_1-h_1,x_2,t),
\end{split}
\end{equation} 
\begin{equation}\label{289}
\begin{split}
  A_2 w =  &- \frac{1}{h_2^2} k (x_1,x_2+0.5 h_2) \exp ( \theta ({\bm x},t) h_2) w(x_1,x_2+h_2, t)\\
  &+ \frac{1}{h_2^2} k (x_1,x_2+0.5 h_2) \exp ( \theta ({\bm x},t) h_2) w({\bm x},t) \\
  &+ \frac{1}{h_2^2} k (x_1,x_2-0.5 h_2) \exp ( - \theta ({\bm x},t) h_2) w({\bm x}, t) \\
  &- \frac{1}{h_2^2} k (x_1,x_2-0.5 h_2) \exp ( - \theta ({\bm x},t) h_2) w(x_1,x_2-h_2,t),
\end{split}
\end{equation} 
where
\[
  \theta ({\bm x},t) = \frac {v({\bm x},t)}{2 k({\bm x})} ,
  \quad {\bm x} \in \omega  .
\]

In the case of (\ref{285}), we have (see (\ref{280}))
\begin{equation}\label{290}
\begin{split}
  A_1 w =  &- \frac{1}{h_1^2} k (x_1+0.5 h_1, x_2) 
  \exp ( - \theta (x_1+ h_1,x_2,t) h_1) w(x_1+h_1,x_2, t) \\
  & + \frac{1}{h_1^2} k (x_1+0.5 h_1,x_2) \exp (\theta ({\bm x},t) h)w({\bm x},t) \\
  &+ \frac{1}{h_1^2} k (x_1-0.5 h_1,x_2) 
  \exp ( - \theta ({\bm x},t) h_1) w({\bm x}, t) \\
  & - \frac{1}{h_1^2} k (x_1-0.5 h_1,x_2) \exp (\theta (x_1-h_1,x_2,,t) h)w(x_1-h_1,x_2,t) ,
\end{split}
\end{equation} 
\begin{equation}\label{291}
\begin{split}
  A_2 w =  &- \frac{1}{h_2^2} k (x_1, x_2+0.5 h_2) 
  \exp ( - \theta (x_1,x_2+ h_2,t) h_1) w(x_1,x_2+h_2, t) \\
  & + \frac{1}{h_2^2} k (x_1,x_2+0.5 h_2) \exp (\theta ({\bm x},t) h)w({\bm x},t) \\
  &+ \frac{1}{h_2^2} k (x_1,x_2-0.5 h_2) 
  \exp ( - \theta ({\bm x},t) h_1) w({\bm x}, t) \\
  & - \frac{1}{h_2^2} k (x_1,x_2-0.5 h_2) \exp (\theta (x_1,x_2-h_2,,t) h)w(x_1,x_2-h_2,t) .
\end{split}
\end{equation} 

Similarly to Theorem~\ref{t-27}, the following statement is proved.

\begin{theorem}\label{t-28}
If on the set of grid functions (\ref{287})
the operator $A$ is defined according to (\ref{286}), (\ref{288}), (\ref{289})
(or (\ref{286}), (\ref{290}), (\ref{291})), 
then the difference scheme (\ref{259}), (\ref{260})
is monotone, and the difference solution
satisfies the a priori estimate (\ref{262}) in the $L_\infty$  (or in $L_1$)
for any $\tau >0$ if $\sigma =1$, and if $0 \le \sigma < 1$,
then this is true under the constraints on a time step (\ref{272})
with
\[
\begin{split}
  \gamma= \max_{{\bm x} \in \omega} \Big \{
  & \frac{1}{h_1^2} k (x_1+0.5 h_1,x_2) \exp ( \theta ({\bm x},t) h_1) \\
  & + \frac{1}{h_1^2} k (x_1-0.5 h_1,x_2) \exp ( - \theta ({\bm x},t) h_1)  \\
  & + \frac{1}{h_2^2} k (x_1,x_2+0.5 h_2) \exp ( \theta ({\bm x},t) h_2) \\
  & + \frac{1}{h_2^2} k (x_1,x_2-0.5 h_2) \exp ( - \theta ({\bm x},t) h_2)  \Big \} .
\end{split}
\] 
\end{theorem}

\subsection{Locally one-dimensional schemes} 

Computational implementation of the exponential schemes
(\ref{259}), (\ref{260}), (\ref{286})--(\ref{289}) and
(\ref{259}), (\ref{260}), (\ref{286}), (\ref{287}), (\ref{290}), (\ref{291}))
involves the inversion of the non-selfadjoint elliptic grid operators
$E + \sigma\tau A$, where the matrix has strong diagonal dominance either by rows or by columns.
To determine the numerical solution at a new time level,
we can apply iterative methods.
Another possibility is to use locally one-dimensional schemes, which are based
on the splitting (\ref{286}) \cite{Yanenko,2001Samarskii}.
Intending to 3D generalizations, we restrict ourselfs to componentwise splitting schemes
\cite{Marchuk,SamVabAdditive}.

Rewrite the difference equation (\ref{259}) as follows:
\begin{equation}\label{292}
  y^{n+1} = S y^{n} + \tau \varphi^n,  
\end{equation} 
where $S$ is the transition operator. For the scheme with weights (\ref{259}), we have
\begin{equation}\label{293}
  S = (E + \sigma\tau A)^{-1} (E + (\sigma-1) \tau A) . 
\end{equation}
From the stability condition (\ref{260}), (\ref{292}), we get
\begin{equation}\label{294}
  \|S \| \leq 1. 
\end{equation}
Monotonicity is ensured by the fact that the matrices
$(E + \sigma\tau A)^{-1}$ and $E + (\sigma-1) \tau A$
are M-matrices.

Splitting schemes are constructed using transition operators for the individual
terms in the additive representation (\ref{286}). Let us define
\begin{equation}\label{295}
  S_\alpha (\tau)   = (E + \sigma\tau A_\alpha )^{-1} (E + (\sigma-1) \tau A_\alpha ) ,
  \quad \alpha = 1,2 .   
\end{equation} 
Instead of (\ref{293}), we will employ
\begin{equation}\label{296}
  S = S_1 (\tau) S_2 (\tau) .
\end{equation} 
The stability condition (\ref{294}) is true if
\begin{equation}\label{297}
  \|S_\alpha \| \leq 1,
  \quad \alpha = 1,2 .     
\end{equation}
For the monotonicity of the scheme (\ref{292}), (\ref{296}), it is sufficient to require
that the individual matrices $S_\alpha, \ \alpha = 1,2$ will be M-matrices.
For any value of $\sigma$, only the first-order accuracy with respect to $\tau$ is possible.

Numerical implementation of the scheme (\ref{260}),  (\ref{292}), (\ref{295}), (\ref{296})
can be conducted using locally one-dimensional schemes with weights, i.e.,
\begin{equation}\label{298}
\begin{split}
  \frac{y^{n+\alpha/2} - y^{n+(\alpha-1)/2}  }{\tau } 
  &+ A_\alpha (\sigma   y^{n+\alpha/2} + (1 - \sigma) y^{n+(\alpha-1)/2} ) \\
  & =
  \varphi^n_\alpha ,   \quad \alpha = 1,2 ,   
\end{split} 
\end{equation}  
where, e.g.,
\[
  \varphi^n_1 = 0,
  \quad \varphi^n_2 = \varphi^n . 
\] 

\begin{theorem}\label{t-29}
If on the set of grid functions (\ref{287})
the operators $A_\alpha, \ \alpha = 1,2$ are defined according to (\ref{288}), (\ref{289})
(or (\ref{280}), (\ref{281})), 
then the locally one-dimensional difference scheme (\ref{260}),  (\ref{298})
is monotone, and the difference solution
satisfies the a priori estimate (\ref{262}) in $L_\infty$ (or in $L_1$)
for any $\tau >0$ if $\sigma =1$, and if $0 \le \sigma < 1$,
then this is true under the constraints on a time step (\ref{273})
with
\[
\begin{split}
  \gamma= \max_{{\bm x} \in \omega} \Big \{
  & \frac{1}{h_1^2} k (x_1+0.5 h_1,x_2) \exp ( \theta ({\bm x},t) h_1) \\
  & + \frac{1}{h_1^2} k (x_1-0.5 h_1,x_2) \exp ( - \theta ({\bm x},t) h_1) , \\
  & \frac{1}{h_2^2} k (x_1,x_2+0.5 h_2) \exp ( \theta ({\bm x},t) h_2) \\
  & + \frac{1}{h_2^2} k (x_1,x_2-0.5 h_2) \exp ( - \theta ({\bm x},t) h_2)  \Big \} .
\end{split}
\] 
\end{theorem}
\begin{proof}
Conditions for stability and monotonicity are verified for each individual equation
(\ref{298}). In particular, for the first equation, we have
\[
  \| y^{n+1/2} \| \leq \| y^{n} \|  
\] 
for
\[
\begin{split}
  \gamma= \max_{{\bm x} \in \omega} \Big \{
  & \frac{1}{h_1^2} k (x_1+0.5 h_1,x_2) \exp ( \theta ({\bm x},t) h_1) \\
  & + \frac{1}{h_1^2} k (x_1-0.5 h_1,x_2) \exp ( - \theta ({\bm x},t) h_1) \Big \}  .
\end{split}
\]  
For the second equation, we get
\[
  \| y^{n+1} \| \leq \| y^{n+1/2} \| + \tau \| \varphi^n\| 
\] 
for
\[
\begin{split}
  \gamma= \max_{{\bm x} \in \omega} \Big \{
  & \frac{1}{h_2^2} k (x_1,x_2+0.5 h_2) \exp ( \theta ({\bm x},t) h_2) \\
  & + \frac{1}{h_2^2} k (x_1,x_2-0.5 h_2) \exp ( - \theta ({\bm x},t) h_2)  \Big \}  .
\end{split}
\]
Monotonicity of locally one-dimensional schemes under consideration is established in a similar way.
\end{proof}

Another classes of splitting schemes can be applied, too.
In this regard, we highlight the class of additively averaged schemes.

Instead of a multiplicative representation of the transition operator (\ref{296}),
we can employ the additive representation
\begin{equation}\label{299}
  S = \frac{1}{2}  (S_1 (2\tau) + S_2 (2\tau)) 
\end{equation}
with preserving the first-order approximation in time for the scheme (\ref{292}).

For the scheme (\ref{292}), (\ref{295}), (\ref{299}), we present another variant of numerical implementation.
Define the auxiliary functions $y_\alpha^{n+1}, \ \alpha=1,2$ from
\begin{equation}\label{300}
  \frac{y_\alpha^{n+1} - y_\alpha^{n}}{2 \tau } 
  + A_\alpha (\sigma   y_\alpha^{n+1} + (1 - \sigma) y_\alpha^{n} ) \\
  = 0 .
\end{equation}
For the approximate solution at a new time level, we put
\begin{equation}\label{301}
  y^{n+1} = \frac{1}{2} (y_1^{n+1} + y_2^{n+1}) + \tau  \varphi^n .
\end{equation} 

Conditions of stability and monotonicity
for this additively averaged locally one-dimensional scheme
are formulated in the following theorem.

\begin{theorem}\label{t-30}
If on the set of grid functions (\ref{287})
the operators $A_\alpha, \ \alpha = 1,2$ are defined according to (\ref{288}), (\ref{289})
(or (\ref{290}), (\ref{291})), 
then the additively averaged locally one-dimensional difference scheme (\ref{260}),  (\ref{300}), (\ref{301}),
is monotone, and the difference solution
satisfies the a priori estimate (\ref{262}) in $L_\infty$ (or in $L_1$)
for any $\tau >0$ if $\sigma =1$, and if $0 \le \sigma < 1$,
this is true under the constraints on a time step (\ref{273})
with
\[
\begin{split}
  \gamma= 2 \max_{{\bm x} \in \omega} \Big \{
  & \frac{1}{h_1^2} k (x_1+0.5 h_1,x_2) \exp ( \theta ({\bm x},t) h_1) \\
  & + \frac{1}{h_1^2} k (x_1-0.5 h_1,x_2) \exp ( - \theta ({\bm x},t) h_1) , \\
  & \frac{1}{h_2^2} k (x_1,x_2+0.5 h_2) \exp ( \theta ({\bm x},t) h_2) \\
  & + \frac{1}{h_2^2} k (x_1,x_2-0.5 h_2) \exp ( - \theta ({\bm x},t) h_2)  \Big \}  .
\end{split}
\] 
\end{theorem}

Additively average schemes, on the one hand, demonstrate lower accuracy in comparison
with schemes of componentwise splitting, but on the other hand, 
they are more promising in terms of parallel Computing --- 
the components $y_\alpha^{n+1}, \ \alpha=1,2$ are determined (see (\ref{300})) independently of each other.

\end{document}